\documentclass[11pt]{article}

\usepackage[a4paper, left=3.5cm,top=3.0cm,right=3.5cm,bottom=3.5cm]{geometry}

\usepackage{concmath}
\usepackage[T1]{fontenc}
\usepackage{siunitx}

\usepackage{amsmath,amsthm,amssymb,cite}
\usepackage{enumitem}
\setitemize{label=\scriptsize{$\blacksquare$}, topsep=0em}
\setenumerate{label=(\alph*), topsep=0em}

\def\plus{{\boldsymbol{\texttt{+}}}}

\usepackage{graphicx}
\usepackage{soul}
\usepackage{authblk}
\usepackage{framed}

\newcommand{\trans}{\mathsf{T}}
\newcommand{\edot}{\,\cdot \,}
\newcommand{\eps}{\epsilon}

\newtheorem{theorem}{Theorem}
\newtheorem{remark}[theorem]{Remark}
\newtheorem{proposition}[theorem]{Proposition}
\newtheorem{definition}[theorem]{Definition}
\newtheorem{lemma}[theorem]{Lemma}

\theoremstyle{definition}
\newtheorem{assumption}[theorem]{Assumption}

\newcommand{\R}{\mathbb R}

\newcommand{\N}{\mathbb N}
\newcommand{\YY}{\mathcal Y}
\newcommand{\XX}{\mathcal X}

\newcommand{\sph}{\mathbb S}
\newcommand{\coloneqq}{:=}

\newcommand{\PP}{\mathcal P}
\newcommand{\QQ}{\mathcal Q}

\newcommand{\Vo}{V}

\newcommand{\Ko}{K}

\newcommand{\Xo}{R}

\newcommand{\resfun}{\Phi}
\newcommand{\aop}{\mathcal A}

\newcommand{\kop}{\mathcal K}
\newcommand{\hop}{\mathcal H}
\newcommand{\vop}{\mathcal V}
\newcommand{\xop}{\mathcal R}
\newcommand{\eop}{\mathcal E}
\newcommand{\uop}{\mathcal U}

\newcommand{\X}{X}
\newcommand{\Y}{Y}
\newcommand{\data}{ y }
\newcommand{\ndata}{ y^\delta}
\newcommand{\res}{r}

\newcommand{\rmd}{\mathrm d}

\DeclareMathOperator{\nr}{ker}
\DeclareMathOperator{\ran}{ran}

\newcommand{\nd}{D}
\newcommand{\id}{d}

\newcommand{\nb}{B}
\newcommand{\ib}{b}
\newcommand{\ibs}{b}

\DeclareMathOperator{\argmin}{arg\, min}

\newcommand\abs[1]{\left\vert#1\right\vert}

\newcommand\norm[1]{\left\Vert#1\right\Vert}
\newcommand\snorm[1]{\Vert#1\Vert}

\newcommand\rest[2]{{#1}|{#2}}

\newcommand{\kl}[1]{\left(#1\right)}
\newcommand{\set}[1]{\left\{#1\right\}}

\newcommand{\skl}[1]{(#1)}
\newcommand{\inner}[2]{\left\langle#1,#2\right\rangle}

\usepackage{soul}
\usepackage{xcolor}
\colorlet{lred}{red!40}
\colorlet{lgreen}{green!40}
\colorlet{lblue}{blue!40}

\newcommand{\Id}{\operatorname{Id}}                         

\newcommand{\rot}{}

\numberwithin{equation}{section}
\numberwithin{figure}{section}
\numberwithin{theorem}{section}

\allowdisplaybreaks

\title{Analysis of the Block Coordinate Descent Method for {\rot Linear} Ill-Posed Problems}

\author{Simon Rabanser}

\affil{Department of Mathematics, University of Innsbruck\authorcr
Technikerstrasse 13, 6020 Innsbruck, Austria\authorcr
E-mail: {simon.rabanser@uibk.ac.at}}

\author{Lukas Neumann}

\affil{Institute of Basic Sciences in Engineering Science, University of Innsbruck\authorcr
Technikerstrasse 13, 6020 Innsbruck, Austria\authorcr
E-mail: {lukas.neumann@uibk.ac.at}}

\author{Markus Haltmeier}

\affil{Department of Mathematics, University of Innsbruck\authorcr
Technikerstrasse 13, 6020 Innsbruck, Austria\authorcr
E-mail: {markus.haltmeier@uibk.ac.at}}

\date{}

\begin{document}

\maketitle
\begin{abstract}
Block coordinate  descent (BCD) methods   approach  optimization problems by  performing gradient steps along alternating  subgroups of coordinates.
This is in contrast to full gradient descent,  where  a gradient step
updates all coordinates simultaneously.  BCD has  been demonstrated to accelerate the  gradient
method in many practical large-scale applications. Despite its success
no convergence analysis  for inverse problems is known so far.
In  this paper, we investigate the BCD method for solving linear  inverse problems. As main theoretical result, we show that for operators having a particular tensor product form,
the BCD method combined with an appropriate stopping criterion yields a convergent regularization method.
To illustrate the theory, we  perform numerical experiments comparing the BCD and the full gradient descent method   for a system  of integral equations.
We also present numerical tests for a non-linear inverse problem not covered by our theory, namely one-step inversion in multi-spectral X-ray tomography.

\medskip \noindent \textbf{Keywords:} ill-posed problems, convergence analysis, regularization theory, coordinate descent, multi-spectral CT

\medskip \noindent \textbf{MSC2010:} 65J20, 44A12, 47J06.

\end{abstract}

\section{Introduction}
\label{sec:intro}

We consider the solution of inverse problems
of the form
\begin{equation}\label{IP}
    y^\delta  =  \aop(x[1], \dots, x[\nb] ) + z
\end{equation}
by block coordinate gradient descent (BCD) methods.
Here $\aop \colon \XX  \to \YY$  is a  {\rot linear} forward operator
between Hilbert spaces $\XX = \X_1 \times \cdots \times \X_\nb$ and  $\YY$.
Moreover, $x = (x[1], \dots, x[\nb])\in\XX$ is the  vector of
blocks $x[\ib] \in  \X_{\ib}$ of unknown variables, $\ndata \in\YY$
are the given noisy data, and  $z$ denotes the data perturbation that satisfies
$\norm{z} \leq \delta$  for some noise level
$\delta\geq 0$.

For many inverse problems, the individual
blocks $x[\ib]$  arise  in a natural manner and
might correspond to $x[\ib] = f[\ib]$, where
$f[\ib] \colon \Omega_\ib \to \R$ are functions modeling unknown
spatially varying parameter distributions.
The blocks might also be  formed by applying domain decomposition
$ \Omega  =\Omega_1 \cup \Omega_2  \cup \ldots \cup\Omega_\nb$ to  a single
function $f  \colon \Omega  \to \R$, and defining $x[\ib] = \rest{f}{\Omega_\ib}$ as the
restriction of $f$ to $\Omega_\ib$.

\subsection{Iterative regularization methods}

The characteristic  feature of inverse problems is their
ill-posedness which means that the solution of \eqref{IP}
is  unstable with respect to data perturbations.
In such a situation, one has to apply
regularization methods   to obtain solutions  in a
stable way.  There are at least two basic classes of  regularization
methods:  iterative  regularization and variational regularization \cite{Engl96,scherzer2009variational}.
In this paper we consider iterative  regularization
and introduce and analyze BCD as new member of this class
of regularization methods.

The most established  iterative regularization approaches for inverse problems
are the Landweber iteration and its variants
\cite{kaltenbacher08,landweber1951iteration,hanke1995convergence,Neubauer17}
\begin{equation} \label{eq:iterL}
	x_{k+1}^\delta
	\coloneqq  x_k^\delta - s^\delta_k \aop^*
	\kl{ \aop (x_k^\delta) - \data^\delta } \,,
\end{equation}
where $x_0^\delta \coloneqq  x_0 \in \XX$  is an initial guess,
$s^\delta_k$ is the step size and $\aop^*$ denotes the adjoint of $\aop$.
If the step size is taken constant, then \eqref{eq:iterL}
is the Landweber iteration \cite{landweber1951iteration,hanke1995convergence}.
Other step size rules yield
the steepest descent and  the minimal error method
\cite{neubauer1995convergence} or  a more recent variant
analyzed in \cite{Neubauer17}.  Kaczmarz type  variants of \eqref{eq:iterL}
for systems  of ill-posed equations have been analyzed in
\cite{decesaro2008steepest,haltmeier2007kaczmarz1,Hal09b,kowar2002,leitao16projective,li2018averaged}.
Kaczmarz methods  make use of a product structure of the image space $\YY$,
and are in this sense dual to BCD methods which exploit the product structure  of the pre-image space $\XX$.

We consider the  product form
$\XX = \X_1 \times \cdots \times \X_\nb$, {\rot
where the forward operator can be written as $\aop = [\aop_1 , \dots, \aop_B]$. }
As a consequence, the Landweber iteration takes the form
\begin{equation} \label{eq:lw}
	\begin{pmatrix}
	x_{k+1}^\delta[1] \\
	x_{k+1}^\delta[2] \\
	\vdots \\
	x_{k+1}^\delta[\nb]
	\end{pmatrix}
	=
	\begin{pmatrix}
	x_k^\delta[1] \\
	x_k^\delta[2] \\
	\vdots \\
	x_k^\delta[\nb]
	\end{pmatrix}
	- s^\delta_k
	\begin{pmatrix}
	\aop_1^*   \\
	\aop_2^*   \\
	\vdots \\
	\aop_B^*
	\end{pmatrix}
	\kl{ \aop (x_k^\delta) - \data^\delta}
	\,.
\end{equation}
We see that each iterative update requires computing
$\nb$ separate updates, one for each of the blocks.

\subsection{Block coordinate descent (BCD)}

In order to {\rot simplify}  the iterative update in \eqref{eq:lw},
a natural idea is to update only a single
block in each iteration.
This results in the BCD iteration
 \begin{equation} \label{eq:cd}
	x_{k+1}^\delta[\ib]
	\coloneqq
	x_k^\delta[\ib]
	-
	s^\delta_k
	\begin{cases}
	 \aop_\ib^*
	\kl{ \aop (x_k^\delta) - \data^\delta}
	& \text{ if  } \ib = \ibs(k)
	\\
	0
	& \text{ otherwise } \,,
	\end{cases}
	\end{equation}
where the control $\ibs(k) \in \set{1, \dots, \nb}$ selects the
block that is updated in the $k$th  iteration.
 If we apply the BCD iteration to {\rot exact data where $\delta =0$,}  we write $x_k$ instead of $x_k^\delta$.
Rigorously studying the iteration \eqref{eq:cd} in the context of ill-posed problems
is the main aim of this paper. To guarantee convergence in the noisy
case we will  slightly modify the update rule of the BCD iteration by including a loping
strategy which  skips the $k$th iterative  step if a certain residual term is
 sufficiently  small (see Definition \ref{def:loping}).
{\rot Under the reasonable assumption that the complexity of evaluating  $\aop$ is essentially $B$-times the complexity $M$ of  evaluating  $\aop^*_\ibs$, then
one step of the Landweber Method has complexity    $\mathcal{O} ( 2BM) $,
 whereas one step of the BCD method has  complexity  $\mathcal{O} ( (B+1)M)$.
 For  the special form of $\aop$ considered in the following section, the complexity of one step of the  BCD method even reduces to   $\mathcal{O} ( 2 M )$; see Remark
\ref{rem:complexity}.
}

Note that the iteration \eqref{eq:cd} arises by applying the
 block gradient {\rot descent}  method, well known in optimization \cite{beck2013convergence,nesterov2012efficiency,saha2013nonasymptotic,wright2015coordinate},
 to the residual functional  $
\frac{1}{2}  \| y^\delta  -  \aop(x) \|^2$.
In a finite dimensional  setting, BCD and other coordinate descent  type
methods are well studied. However, existing convergence
results mostly analyze convergence in the objective value.
This only implies  convergence  in pre-image space, if the residual
functional is  strongly  convex.  Strong convexity does not hold for ill-posed problems.
Therefore, existing convergence results and methods cannot be applied to
ill-posed inverse  problems. {\rot Note that removing the strict convexity assumption
can also also be achieved by coupling the BCD method with a proximal term; see   \cite{bolte2014proximal} and the references therein.}

To the best of our knowledge, no convergence result for
\eqref{eq:cd} in the  ill-posed setting is available.
As the main   contribution in this   paper we will present a convergence
analysis of BCD applicable to the ill-posed case. We show that under
assumptions specified in Section \ref{sec:prelim},
for operators having a particular tensor product form, the BCD iteration
yields a regularization method for solving ill-posed linear problems.

\subsection{Outline}

This paper is organized as follows.
In Section \ref{sec:prelim} we present the
main assumptions made in this paper,
derive an auxiliary  results and
introduce the  loping strategy.
In Section~\ref{sec:analysis} we
present the convergence analysis.
In the  exact data case, we
show that the BCD iteration converges  to a solution
$x^*$ of the given equation as $k \to \infty$.
In  the  noisy data case we show that the stopping index
of the loping BCD iteration is finite and the corresponding
iterates converge to $x^*$ as $\delta \to 0$.
To illustrate the theory, in Section~\ref{sec:int} {\rot we compare
  the BCD method with the gradient method for a system  of integral
  equations.    Additionally, in Section~\ref{sec:xray} we consider
a non-linear  example} not covered by our theory,  namely one-step inversion in  multi-spectral  X-ray
tomography \cite{rigie2015joint,kazantsev2018joint,atak2015dual,barber2016algorithm}.
The paper concludes with a short  discussion  presented
in Section \ref{sec:conclusion}.

\section{Preliminaries}
\label{sec:prelim}

In this section we formulate the main assumptions and
derive basic results that we will use in the convergence
analysis presented in Section~\ref{sec:analysis}.

Note that for any Hilbert space
$\X$ we can write $\X^\nb  \simeq \R^\nb \otimes \X$.
For any $\ib \in \set{1, \dots, \nb}$ we define the projection operators
\begin{equation} \label{eq:pr}
    \PP_\ib = ( e_\ib e_\ib^\trans)  \otimes \Id_\X  \colon \X^\nb   \to \X^\nb
    \colon
    	\begin{pmatrix}
	x[1] \\
	\vdots \\
	x[\ib] \\
	\vdots \\
	x[\nb]
	\end{pmatrix}
	\mapsto
	\begin{pmatrix}
	0 \\
	\vdots \\
	x[\ib] \\
	\vdots \\
	0
	\end{pmatrix}\,,
\end{equation}
where $e_\ib$ denotes the $\ib$th standard basis
vector in $\R^{\nb}$, defined by $e_\ib[\ib] =1$ and $e_\ib[\ib'] =0$ for $\ib' \neq \ib$.
Using \eqref{eq:pr}, the BCD method  \eqref{eq:cd} can be written in the
compact form
\begin{equation} \label{eq:iter}
    x_{k+1}^\delta
    \coloneqq  x_k^\delta - s^\delta_k \PP_{\ibs(k)} \aop^*(\aop (x_k^\delta)-y^\delta)\,.
\end{equation}
Here  $\ibs(k) \in \{1,\dots,\nb\}$ is the selected block at the $k$th
iteration, $s^\delta_k >0 $ is the step size, and  $x_0^\delta:=x_0 \in \XX$   is  some initial guess. Recall that  in the case of exact data
we write $x_k$ instead of $x_k^\delta$.

\subsection{Main assumptions}

{\rot
We note that the main difficulty we encountered in the
convergence analysis of  the BCD  method for ill-posed problems
is that even  for exact data $\data = \aop (x^*)$,
the error   $\norm{x_k - x^*}$ is  not monotonically decreasing,
except for some very special cases.
This can be easily verified  for linear operators in $\R^\nb$. On the other hand,
the BCD is monotonically decreasing in the objective value, which is used in existing
convergence theory for optimization problems \cite{beck2013convergence,nesterov2012efficiency,saha2013nonasymptotic,wright2015coordinate}.
However, this cannot be used directly for the  convergence
analysis in the ill-posed setting where the value of the residual functional
gives no bounds for the error $\norm{x_k - x^*}$.
}

{\rot We present a complete convergence analysis  under the following assumption that allows to separate the difficulties
due to the ill-posedness and due to the non-monotonicity. }

\begin{assumption}[Main\label{ass} conditions for the convergence analysis] \mbox{}
\begin{enumerate}[label=(A\arabic*)]
\item \label{ass1}
$\XX$, $\YY$
are  Hilbert spaces of the form
$\XX  = \X^\nb$,
{\rot $\YY  = \Y^\nd$
with $\nd, \nb\in \N$.}

\item \label{ass2} $\aop \colon  \XX \to \YY$
has  the form  $\aop = \Vo \otimes \Ko$, where
\begin{itemize}
\item {\rot $\Ko \colon  \X \to \Y$ is bounded linear};

\item
{\rot $\Vo  \in \R^{\nd \times \nb}$ has rank $\nb$ and non-vanishing  columns
$v_\ib \in \R^\nd$;}

 \end{itemize}

\item  \label{ass4}
The control $b \colon \N \to \set{1, \dots, B}$ satisfies\\
$\exists p \in \N \; \forall k \in \N \colon
\set{\ibs(k),  \dots, \ibs(k  + p-1) } = \set{1, \dots, \nb}$.
\end{enumerate}
\end{assumption}

Let us introduce the operators {\rot
\begin{align*}
&\kop_{\nb}	  \coloneqq  \Id_{\R^\nb} \otimes \, \Ko \colon \XX \to \YY
\colon\begin{pmatrix}
x[1] \\ \vdots \\ x[\nb]
\end{pmatrix}
\mapsto
\begin{pmatrix}
\Ko(x[1]) \\ \vdots \\ \Ko(x[\nb])
\end{pmatrix}
\\
& \vop_\Y  \coloneqq  \Vo \otimes \Id_{\Y} \colon \YY \to \YY \colon
y
\mapsto
\sum_{\ib=1}^\nb v_\ib \, y[\ib] \,.%
\end{align*}
In a similar manner we denote $\kop_{\nd}  \coloneqq  \Id_{\R^\nd} \otimes \, \Ko$ and
$ \vop_\X \coloneqq  \Vo \otimes \Id_{\X}$.
Then we have
$\aop = \vop_\Y \circ \kop_{\nb} = \kop_{\nd}  \circ \vop_\X   $.}

To overcome the above mentioned obstacles in the convergence analysis
we will study the auxiliary sequence  $(\vop_\X x_k^\delta )_{k\in \N}$ which, by linearity,
satisfies
\begin{multline} \label{eq:iter1}
    \vop_\X x_{k+1}^\delta
     =
     \vop_\X x_k^\delta - s^\delta_k  \vop_\X \PP_{\ibs(k)} \aop^*(\aop (x_k^\delta)-\ndata)
     \\
     =\vop_\X x_k^\delta - s^\delta_k \snorm{v_{\ibs(k)}}^2  \QQ^{\X}_{\ibs(k)} \kop_{\nd}^*(\aop (x_k^\delta)-\ndata)
     \,.
\end{multline}
Here we have set {\rot
\begin{equation}
    \QQ^{\X}_\ib \coloneqq  \frac{1}{\norm{v_\ib}^2}(v_\ib v_\ib^\trans)
    \otimes \Id_\X \colon  \XX \to  \XX \,.
\end{equation}
We will also use the notation $\QQ^{\Y}_\ib \coloneqq  \snorm{v_\ib}^{-2} (v_\ib v_\ib^\trans)
    \otimes \Id_\Y$.} As an important auxiliary result   we
will show monotonicity for $(\vop_\X x_k^\delta )_{k\in \N}$.
This allows us to show that the BCD  method combined with a loping
strategy is a convergent regularization  method.
In fact, this is the reason for requiring the forward operator $\aop$
to have the   particular tensor product form specified in assumption \ref{ass2}. The convergence analysis in the more general setting is still an open and challenging problem.

{\rot Note that the assumption $\operatorname{rank} (\Vo) = \nb$
is only necessary for the convergence of $(\vop_\X x_k)_{k\in \N}$ implying
convergence of $(x_k )_{k\in \N}$. In the case that
$\Vo$  has arbitrary rank, the main convergence results   still hold true
for the semi-norm $\norm{\vop_\X (\edot )}$ in place of the norm $\norm{\edot }$. }

{ \rot
\begin{remark}[Numerical  complexity] \label{rem:complexity}
For  the considered form $\aop =  \Vo \otimes \Ko$ and a cyclic control $\ibs(k) =  ((k-1) \operatorname{mod} B)+1$, one cycle of updates  with the BCD method for $k \in \set{\ell B, \dots , (\ell+1) B -1}$   has essentially the same numerical complexity as one iteration with the Landweber iteration. To see this, we implement the BCD method in the following manner:
\begin{enumerate}[label=(S\arabic*)]
\item Initialization:  $\forall b = 1, \dots,  B$   do
\begin{itemize}
\item $x_{\rm BCD}[\ib] \gets x_0[\ib]$
\item $h_{\rm BCD}[\ib] \gets \Ko(x_{\rm BCD}[\ib])$.
\end{itemize}
\item Updates: $\forall  i_0 = 1, \dots, N_{\rm cycle} \forall  b = 1, \dots,  B$   do
\begin{itemize}
\item $x_{\rm BCD}[b] \gets  x_{\rm BCD}[b]  - s_k \Ko^*(  (\vop_\Y^* ( \vop_\Y  h_{\rm BCD} - y^\delta ))[b] )$
\item $h_{\rm BCD}[b] \gets  \Ko( x_{\rm BCD}[b] )$.
\end{itemize}
\end{enumerate}
Complexity of the  above procedure is dominate by  the evaluation of
$\Ko$, $\Ko^*$ and the evaluation of $\vop_\Y$, $\vop_\Y^*$.
Unless $B$ is very large (or evaluating $\Ko$, $\Ko^*$  is cheap),
for typical inverse problems, the dominating parts are  $\Ko$, $\Ko^*$. This shows that the complexity
  of one cycle of the BCD iteration in fact is similar to the complexity of  one iteration of the Landweber iteration.
\end{remark}
}

\subsection{Monotonicity}

The following lemma is an important auxiliary result, which will be used at several places throughout this article.

\begin{lemma}[Monotonicity]\label{lem:mon}
Let $x^* \in \XX$ satisfy  $\aop(x^*) = y$ and set
\begin{equation} \label{eq:res}
    \res^\delta_k
    \coloneqq
    \norm{\QQ^{\Y}_{\ibs(k)} \kl{ \ndata-\aop (x_k^\delta) } }
    \,.
\end{equation}
Then, the following estimate holds:
 \begin{multline}\label{eq:mon}
    \frac{1}{2}\norm{\vop_\X x_{k+1}^\delta - \vop_\X x^*}^2 - \frac{1}{2}\norm{\vop_\X x_k^\delta - \vop_\X x^*}^2
    \leq
         - s^\delta_k  \res^\delta_k \snorm{v_{\ibs(k)}}^2 \kl{ \res^\delta_k - \delta_{\ib(k)}   }
          \\
         +\frac{(s^\delta_k)^2}{2} \, \norm{\vop_\X \PP_{\ibs(k)}\aop^*\kl{ \ndata - \aop (x_k^\delta)}}^2 \,.
\end{multline}
In particular, if $\snorm{\QQ^{\Y}_{\ib} (y-y^\delta)} \leq \delta_\ib$ and
$\res^\delta_k \geq   \delta_{\ibs(k)} $ and if the step size is chosen  such that
\begin{equation}\label{eq:stepsize}
    0 \leq s^\delta_k\leq  \frac{2 \res^\delta_k  \snorm{v_{\ibs(k)}}^2    \left(\res^\delta_k  -\delta_{\ib(k)} \right)}{\norm{\vop_\X \PP_{\ibs(k)}\aop^*(\ndata - \aop (x_k^\delta))}^2}\,,
\end{equation}
then $\snorm{\vop_\X x_{k+1}^\delta - \vop_\X x^*}^2 \leq  \snorm{\vop_\X x_k^\delta - \vop_\X x^*}^2$.
\end{lemma}

\begin{proof}
Equation~\eqref{eq:iter1} implies
\begin{multline} \label{eq:monaux}
        \frac{1}{2}\norm{\vop_\X x_{k+1}^\delta - \vop_\X x^*}^2 - \frac{1}{2}\norm{\vop_\X x_k^\delta - \vop_\X x^*}^2 \leq
\inner{ \vop_\X  x_k^\delta - \vop_\X x^*  }{ \vop_\X  x_{k+1}^\delta - \vop_\X x_k^\delta} \\
+\frac{(s^\delta_k)^2}{2} \, \norm{\vop_\X \PP_{\ibs(k)}\aop^*\kl{ \ndata - \aop (x_k^\delta)}}^2 \,.
\end{multline}
We have
\begin{align*}
&\inner{\vop_\X  x_k^\delta - \vop_\X x^*  }{\vop_\X x_{k+1}^\delta - \vop_\X x_k^\delta}
\\ & =
s^\delta_k  \snorm{v_{\ibs(k)}}^2
 \inner{ \vop_\X  (x_k^\delta - x^* ) }{ \QQ^{\Y}_{\ibs(k)} \kop_{\nd}^*
 \kl{ \data^\delta - \aop (x_k^\delta)  } }
\\ & =
s^\delta_k  \snorm{v_{\ibs(k)}}^2
 \inner{ \kop_{\nb} \vop_\X  (x_k^\delta - x^* ) }{ \QQ^{\Y}_{\ibs(k)}
 \kl{ \data^\delta - \aop (x_k^\delta)  } }
\\ & =
s^\delta_k  \snorm{v_{\ibs(k)}}^2
 \inner{ \aop  (x_k^\delta) - \aop(x^*) }{ \QQ^{\Y}_{\ibs(k)}
 \kl{ \data^\delta - \aop (x_k^\delta)  } }
\\ & =
s^\delta_k  \snorm{v_{\ibs(k)}}^2
 \inner{ \aop  (x_k^\delta) - \data^\delta + \data^\delta - \aop(x^*) }{ \QQ^{\Y}_{\ibs(k)}
 \kl{ \data^\delta - \aop (x_k^\delta)  } }
\\ & \leq
s^\delta_k  \snorm{v_{\ibs(k)}}^2
 \kl{
 - \snorm{ \QQ^{\Y}_{\ibs(k)} (\data^\delta - \aop (x_k^\delta) ) }^2
 + \delta_{\ib(k)} \snorm{ \QQ^{\Y}_{\ibs(k)} (\data^\delta - \aop (x_k^\delta) ) }
 }
 \end{align*}
 By combining \eqref{eq:monaux} with the  above estimate, we obtain
 \begin{multline*}
    \frac{1}{2}\norm{\vop_\X x_{k+1}^\delta - \vop_\X x^*}^2 - \frac{1}{2}\norm{\vop_\X x_k^\delta - \vop_\X x^*}^2
    \leq
         s^\delta_k  \res^\delta_k \snorm{v_{\ibs(k)}}^2 \Bigl( \delta_{\ib(k)}
         - \res^\delta_k \Bigr)
         \\
        +\frac{(s^\delta_k)^2}{2} \, \norm{\vop_\X \PP_{\ibs(k)}\aop^*\kl{ \ndata - \aop (x_k^\delta)}}^2 \,,
\end{multline*}
which is the desired estimate \eqref{eq:mon}.
If $s^\delta_k$ is chosen according to \eqref{eq:stepsize},  then the right hand side in inequality  \eqref{eq:mon} is less or equal to 0, which implies  $\snorm{\vop_\X x_{k+1}^\delta - \vop_\X x^*}^2 \leq  \snorm{\vop_\X x_k^\delta - \vop_\X x^*}^2$.
\end{proof}

\subsection{Loping BCD and discrepancy principle}

From Lemma~\ref{lem:mon} we see that if  the residual term
$\res^\delta_k = \snorm{\QQ^{\Y}_{\ibs(k)} \skl{ \ndata-\aop (x_k^\delta) } } $
satisfies \eqref{eq:res}, then the error  $\snorm{ \vop_\X x_k^\delta - \vop_\X x^*}$
is decreasing. In the case that  \eqref{eq:res} does not hold, then
an iterative update might increase the value  of $\snorm{ \vop_\X x_k^\delta - \vop_\X x^*}$.
Following a similar strategy introduced in
\cite{haltmeier2007kaczmarz1,decesaro2008steepest} for Kaczmarz type iterative
method we therefore modify \eqref{eq:iter} by introducing a loping strategy as follows.

\begin{definition}[Loping BCD]
We\label{def:loping} define the loping BCD  method by
\begin{align} \label{eq:iterS1}
 	x_{k+1}^\delta
 	&\coloneqq
	x_k^\delta - d^\delta_k s^\delta_k  \PP_{\ibs(k)} \aop^*(\aop (x_k^\delta)-\ndata)
\\ \label{eq:iterS2}
    d^\delta_k
    &\coloneqq
    \begin{cases}
    	1
	& \text{ if }
	\res^\delta_k \geq   \tau \delta_{\ib(k)}
	\\
	0 & \text{ otherwise  }\,,
	\end{cases}
\end{align}
where $\res^\delta_k= \snorm{\QQ^{\Y}_{\ibs(k)} (\ndata-\aop( x_{k}^\delta))} $ is as in Equation~\eqref{eq:res}, and
\begin{equation}\label{eq:tau}
 {\rot \tau   >  1 }\,.
\end{equation}
\end{definition}

In the case of exact data, we have $d^\delta_k =1$ and the loping BCD
iteration reduces to the standard BCD. In the noisy data case the loping parameters $d^\delta_k$ ensure that no update is made if we cannot guarantee that an
update would decrease $\snorm{ \vop_\X x_k^\delta - \vop_\X x^*}$.
Note that  the  choice of $\tau$ as in \eqref{eq:tau} implies
that condition \eqref{eq:res} is satisfied whenever
we have $d^\delta_k =1$.
For the loping BCD,  Lemma~\ref{lem:mon} therefore
implies that the error term
$\snorm{ \vop_\X x_k^\delta - \vop_\X x^*}$  is in fact monotonically decreasing.
Moreover, we can show the following.

\begin{lemma}[Summability of squared residuals]\label{conv1}
Let $x^* \in \XX$ satisfy   $\aop(x^*) = y$.
Then the residuals $\res^\delta_k  \coloneqq  \snorm{ \QQ^{\Y}_{\ibs(k)}(\ndata - \aop(x_k^\delta))}$ of
 the loping  BCD iteration \eqref{eq:iterS1}, \eqref{eq:iterS2}
satisfy
    \begin{equation}\label{SumRes}
        \sum_{k\in \N}
        d^\delta_k s^\delta_k\snorm{v_{\ibs(k)}}^2 (\res^\delta_k)^2 \leq
        \frac{\norm{\vop_\X x_0 - \vop_\X x^*}^2}{\gamma_{\rm min}  (2- \theta_{\rm max})}\,,
    \end{equation}
 where,  $s^\delta_k$, $\gamma_{\rm min}$, $\theta_{\rm max}$
 are chosen such that
 \begin{enumerate}[label=(S\arabic*), leftmargin =3em]
 \item \label{sum1}
 $\forall k \in \N \colon d^\delta_k =1 \Rightarrow  s^\delta_k  \in (0, 2 A_k^\delta) $
 with  $ A_k^\delta \coloneqq
 \frac{ \snorm{v_{\ibs(k)}}^2 \res^\delta_k( \res^\delta_k -\delta_{\ib(k)} )}{ \norm{ \vop_\X \PP_{\ibs(k)}
	  \aop^*(
		\ndata - \aop(x_k^\delta) )}^2}$;

\item \label{sum2}
    $\forall k \in \N \colon d^\delta_k =1 \Rightarrow \theta_k \coloneqq s^\delta_k / A_k^\delta  \leq  \theta_{\rm max} <  2$;

\item \label{sum3}
    $1-1 / \tau \geq \gamma_{\rm min} >0 $.
 \end{enumerate}
\end{lemma}

\begin{proof}
We first show
\begin{multline} \label{eq:saux}
        \norm{\vop_\X x_k^\delta - \vop_\X x^*}^2 - \norm{\vop_\X x_{k+1}^\delta - \vop_\X x^*}^2
    \\ \geq
           (2- \theta_{\rm max}) d^\delta_k s^\delta_k
           \snorm{v_{\ibs(k)}}^2(\res^\delta_k)^2
           \kl{ 1 - 1/\tau } \,.
    \end{multline}
If $\res^\delta_k  <   \tau \delta$, then  $d^\delta_k = 0$ and $x_{k+1}^\delta = x_k^\delta$ and  therefore \eqref{eq:saux} holds with equality.
If  $\res^\delta_k  \geq    \tau \delta$, application of
Lemma~\ref{lem:mon}, \ref{sum2} and \ref{sum1} yield
\begin{align*}
        \lVert \vop_\X x_k^\delta & - \vop_\X x^*\rVert^2
        -
        \lVert \vop_\X x_{k+1}^\delta - \vop_\X x^*\rVert^2     \\
    &\geq
        2s^\delta_k \snorm{v_{\ibs(k)}}^2\res^\delta_k (-\delta_{\ib(k)} +\res^\delta_k)    -(s^\delta_k)^2 \snorm{\vop_\X \PP_{\ibs(k)}\aop^*(\ndata - \aop (x_k^\delta))}^2
  \\&    \geq
        2s^\delta_k \snorm{v_{\ibs(k)}}^2\res^\delta_k  (-\delta_{\ib(k)} + \res^\delta_k )       -s^\delta_k  \theta_{\rm max}  A_k^\delta \snorm{\vop_\X \PP_{\ibs(k)}\aop^*(\ndata - \aop (x_k^\delta))}^2
 \\ &
    = 2s^\delta_k \snorm{v_{\ibs(k)}}^2\res^\delta_k  (-\delta_{\ib(k)} + \res^\delta_k )       -s^\delta_k \theta_{\rm max} \snorm{v_{\ibs(k)}}^2\res^\delta_k  (-\delta_{\ib(k)} + \res^\delta_k )
 \\ &
=
        (2- \theta_{\rm max})s^\delta_k\snorm{v_{\ibs(k)}}^2\res^\delta_k  ( \res^\delta_k
        -\delta_{\ib(k)} )
 \\ &
    \geq
        (2- \theta_{\rm max})s^\delta_k\snorm{v_{\ibs(k)}}^2
        (\res^\delta_k)^2
        \kl{ 1 -1/\tau } \,.
\end{align*}
This shows \eqref{eq:saux} with $d^\delta_k = 1$ in \eqref{eq:iterS2}.

Summing \eqref{eq:saux} over all $k \in \N$ and using  \ref{sum3}
we obtain
\begin{equation*}
    \norm{\vop_\X x_0  - \vop_\X x^*}^2 \geq (2-\theta_{\rm max})\gamma_{\rm min}
    \sum_{k \in \N} d_k s^\delta_k\snorm{v_{\ibs(k)}}^2  (\res^\delta_k)^2 \,,
\end{equation*}
which shows \eqref{SumRes} after dividing by
$(2-\theta_{\rm max})\gamma_{\rm min}$.
\end{proof}

\begin{remark}
Note the  conditions  for   the step sizes in  Lemma~\ref{conv1}
are inspired by \cite{Neubauer17}, where  a new
step size rule for  the gradient method for ill-posed problems
has been introduced.
From the definitions of $\res^\delta_k, d^\delta_k$ we obtain
$\res^\delta_k  - \delta_{\ib(k)}  \geq   =   (1-1/\tau) \res^\delta_k  $.
Moreover, recall that $\vop_\X \PP_{\ibs(k)} \vop_\X^* =  \snorm{v_{\ibs(k)}}^2 \QQ^{\X}_{\ibs(k)}$.
Consequently,
\begin{multline*}
A_k^\delta
  =
 \frac{ \snorm{v_{\ibs(k)}}^2 \res^\delta_k( \res^\delta_k  -\delta_{\ib(k)} ) }{\snorm{ \vop_\X \PP_{\ibs(k)} \vop_\X^*
	  \kop_{\nd}^*(
		\ndata - \aop(x_k^\delta) )}^2}
 \geq
 \kl{ 1 - \frac{1}{\tau} }
 \frac{ \snorm{\QQ^{\Y}_{\ibs(k)} (\ndata-\aop( x_{k}^\delta))}^2 }
 {\snorm{
  \QQ^{\X}_{\ibs(k)} \kop_{\nd}^*  (
\ndata - \aop(x_k^\delta) )}^2}  		\\
  \geq
 \gamma_{\rm min} \frac{  \snorm{\QQ^{\Y}_{\ibs(k)} (\ndata-\aop( x_{k}^\delta))}^2 }
 {\snorm{
  \kop_{\nb}^* \QQ^{\Y}_{\ibs(k)}   (
\ndata - \aop(x_k^\delta) )}^2}  	 \geq
 \frac{ \gamma_{\rm min}}
 {\snorm{\kop_{\nb}^*}^2}
   \,.		
\end{multline*}
This implies that we can choose the step sizes  bounded from below.
In particular,  \eqref{SumRes} holds for any constant step size
choice  $s_k^\delta = s_\star \in (0, \gamma_{\rm min} / {\snorm{\kop_{\nb}^*}^2}]$.
\end{remark}

\section{Convergence Analysis of the BCD method}
\label{sec:analysis}

Throughout the following, let Assumption \ref{ass}
be satisfied. Moreover, we assume that the step sizes
satisfy $s_{\rm min} \leq s_k^\delta \leq s_{\rm max}$
for some numbers $s_{\rm min} \leq s_{\rm max}$ independent of
the iteration index $k \in \N $ and the noise level
$\delta \geq 0$, and that \ref{sum1}-\ref{sum3}
in Lemma \ref{conv1} hold.

\subsection{Convergence for exact data}

In this subsection we show convergence of the BCD iteration in the
noise free case. The proof closely follows \cite{decesaro2008steepest,kowar2002}.

\begin{theorem}[Convergence of BCD for exact data]\label{thm:exact}
  In the exact  data case   $\delta =0$,
  the BCD iteration  $(x_k)_{k\in \N}$
  defined by  \eqref{eq:iter},  satisfies
  $x_k \to x^\plus$, where $x^\plus$ is the solution of $\aop(x) = y$
  with minimal distance to $x_0$.
\end{theorem}

\begin{proof}
Let $x^* \in \XX$ satisfy $\aop(x^*) = y$ and
define $\xi_k \coloneqq \vop_\X x_k- \vop_\X x^*$.
We will show that $(\xi_k)_{k \in \N}$ is a
Cauchy sequence.
For $k =k_0 p+k_1$ and $l=l_0p+l_1$ with $k \leq l$ and $k_1,l_1 \in \{0,\dots,p-1\}$, let $n_0 \in \{k_0,\dots,l_0\}$ be such that
\begin{align}\label{eq:conv0}
\sum_{i_1=0}^{p-1}&\norm{\QQ^{\Y}_{\ibs(p n_0+i_1)}(y-\aop(x_{p n_0+i_1}))}\\ \notag
&\leq \sum_{i_1=0}^{p-1}\norm{\QQ^{\Y}_{\ibs(p i_0+i_1)} (y-\aop(x_{p i_0+i_1}))} \text{ for all } i_0 \in \{k_0,\dots,l_0\}\,.
\end{align}
With $n \coloneqq p n_0 + p - 1$ we have
\begin{equation}\label{errnorm1}
    \norm{\xi_k-\xi_l} \leq \norm{\xi_k-\xi_n}+\norm{\xi_l-\xi_n}
\end{equation}
and
\begin{align}\label{errnorm2}
    \norm{\xi_n-\xi_k}^2 &= 2\inner{\xi_n-\xi_k}{\xi_n}+ \norm{\xi_k}^2 - \norm{\xi_n}^2\,, \\ \label{errnorm2a}
    \norm{\xi_n-\xi_l}^2 &= 2\inner{\xi_n-\xi_l}{\xi_n} + \norm{\xi_l}^2 - \norm{\xi_n}^2\,.
\end{align}
According to Lemma~\ref{lem:mon}, the nonnegative sequence  $(\norm{\xi_k})_{k \in \N}$ is monotonically decreasing and therefore converges to some $\eps \geq 0$.
Consequently, the  last two terms in equations \eqref{errnorm2} and \eqref{errnorm2a}
converge to $\varepsilon^2-\varepsilon^2 = 0$ for $k \to \infty$. In order to show that also $\inner{\xi_n-\xi_k}{\xi_n}$ and $\inner{\xi_n-\xi_l}{\xi_n}$ converge to zero, we set
$i^* \coloneqq pn_0 + i_1$. Then using the definition of the BCD method in (\ref{eq:iter})
for $i \in \{ 0, \dots, p-1 \}$ we obtain
\begin{align}\label{eq:conv}
&  \abs{\inner{\xi_n-\xi_k}{\xi_n}}  = \abs{\inner{\vop_\X x_n-\vop_\X x_k}{\vop_\X x^*-\vop_\X x_n}} \\  \notag
& = \abs{\sum_{i=k}^{n-1}  s_i \inner{ \vop_\X \PP_{\ibs(i)} \aop^*(\aop(x_i)-y)}{ \vop_\Y (x^*-x_n)} } \\ \notag
& \leq v_{\rm max}^2 \sum_{i=k}^{n-1} s_i\abs{\inner{ \QQ^{\X}_{\ibs(i)} \aop^*(\aop(x_i)-y)}{x^*-x_n}} \\ \notag
& = v_{\rm max}^2 \sum_{i=k}^{n-1} s_i\abs{\inner{ \aop(x_i)-y}{\QQ^{\Y}_{\ibs(i)} \aop(x^*-x_n)}} \\ \notag
& = v_{\rm max}^2  \sum_{i=k}^{n-1} s_i\abs{\inner{\aop(x_i)-y}{\QQ^{\Y}_{\ibs(i)} \aop(x^*-x_{i^*}) + \QQ^{\Y}_{\ibs(i)} \aop(x_{i^*}-x_n)}}\\ \notag
&\leq v_{\rm max}^2 \sum_{i=k}^{n-1} s_i \norm{\QQ^{\Y}_{\ibs(i)}(\aop(x_i)-y)}\norm{\QQ^{\Y}_{\ibs(i)} \aop(x^*-x_{i^*})}\\ \notag
 &\hspace{0.1\textwidth} + v_{\rm max}^2 \sum_{i=k}^{n-1}s_i \norm{\QQ^{\Y}_{\ibs(i)}(\aop(x_i)-y)}\norm{\QQ^{\Y}_{\ibs(i)} \aop(x_{i^*}-x_n)}\,,
\end{align}
with $v_{\rm max} \coloneqq \max \set{\snorm{v_1}, \dots, \snorm{v_{\nb}}}$.
Further we obtain
\begin{align}\label{eq:conv3}
&\norm{\QQ^{\Y}_{\ibs(i)}\aop(x_{i^*}-x_n)}
\\ \notag
& = \norm{\QQ^{\Y}_{\ibs(i)} \kop_{\nd} \vop_\X (x_{i^*}-x_n)}
\\ \notag
& \leq \snorm{\kop_{\nd}} \norm{\vop_\X(x_{i^*}-x_n)}\\ \notag
&\leq \snorm{\kop_{\nd}}\sum_{j=i_1}^{p-2}s_j\norm{\vop_\X \PP_{\ibs(p n_0 +j)} \aop^*(y-\aop(x_{p n_0+j}))}\\ \notag
&= \snorm{\kop_{\nd}} \sum_{j=i_1}^{p-2}s_j\norm{\vop_\X \PP_{\ibs(p n_0 +j)}\vop_\X^{*} \kop_{\nd}^{*}(y-\aop(x_{p n_0+j}))}\\ \notag
&= \snorm{\kop_{\nd}} \sum_{j=i_1}^{p-2}s_j\norm{v_{\ibs(p n_0 +j)}}^2\norm{ \QQ^{\X}_{\ibs(p n_0 +j)} \kop_{\nd}^{*}(y-\aop(x_{p n_0+j}))}\\ \notag
&\leq \snorm{\kop_{\nd}} \sum_{j=i_1}^{p-2}s_j\norm{v_{\ibs(p n_0 +j)}}^2\norm{  \kop_{\nd}^{*}(\QQ^{\Y}_{\ibs(p n_0 +j)}(y-\aop(x_{p n_0+j})))}\\
&\leq \snorm{\kop_{\nd}}^2 s_{\rm max} v_{\rm max}^2 \sum_{j=0}^{p-1}\norm{\QQ^{\Y}_{\ibs(p n_0 +j)}(y-\aop(x_{p n_0+j}))} \,.
\end{align}
Substituting the estimate in (\ref{eq:conv}), using
the inequality $(\sum_{i=0}^{p-1}a_i)^2 \leq  p \sum_{i=0}^{p-1}a_{i}^{2}$ and (\ref{eq:conv0}) one concludes
\begin{align}\label{eq:conv4}
&\abs{\inner{\xi_n-\xi_k}{\xi_n}} \\ \notag
&\leq  2 p s_{\rm max} v_{\rm max}^2  \sum_{i_0=k_0}^{n_0-1}\sum_{i_1=0}^{p-1}\norm{\QQ^{\Y}_{\ibs(pi_0+i_1)}(y-\aop(x_{p i_0 + i_1}))}^2\\ \notag
& +s_{\rm max}^2 \snorm{\kop_{\nb}}^2 v_{\rm max}^4\sum_{i_0=k_0}^{n_0-1}\sum_{i_1=0}^{p-1}\norm{\QQ^{\Y}_{\ibs(pi_0+i_1)}(y-\aop(x_{p i_0 + i_1}))} \sum_{j=0}^{p-1}\norm{\QQ^{\Y}_{\ibs(p i_0 +j)}(y-\aop(x_{p i_0 + j}))} \\ \notag
&\leq 2 p   s_{\rm max} v_{\rm max}^2 \sum_{i_0=k_0}^{n_0-1}\sum_{i_1=0}^{p-1}\norm{\QQ^{\Y}_{\ibs(pi_0+i_1)}(y-\aop(x_{p i_0 + i_1}))}^2\\ \notag
& \qquad + s_{\rm max}^2 \snorm{\kop_{\nb}}^2 v_{\rm max}^4 \sum_{i_0=k_0}^{n_0-1}\Bigl(\sum_{i_1=0}^{p-1}\norm{\QQ^{\Y}_{\ibs(pi_0+i_1)}(y-\aop(x_{p i_0 + i_1}))}\Bigr)^2 \\ \notag
&\leq C \sum_{i_0=k_0}^{n_0-1}\sum_{i_1=0}^{p-1}\norm{\QQ^{\Y}_{\ibs(pi_0+i_1)}(y-\aop(x_{p i_0 + i_1}))}^2 \,,
\end{align}
where we defined $ C\coloneqq s_{\rm max} v_{\rm max}^2 (2p + s_{\rm max} \snorm{\kop_{\nb}}^2 v_{\rm max}^2 p)$. Finally, we have
\begin{equation*}
\abs{\inner{\xi_n-\xi_k}{\xi_n}} \leq \frac{C}{s_{\rm min}}\sum_{i_0=k_0}^{n_0-1}\sum_{i_1 =0}^{p-1}s_{p i_0 +i_1}\norm{\QQ^{\Y}_{\ibs(i_1)}(y-\aop(x_{p i_0 + i_1}))}^2\,.
\end{equation*}
Because of Lemma~\ref{conv1}, the last sum converges to zero for $k=pk_0+k_1 \to \infty$ which implies $\abs{\inner{\xi_n-\xi_k}{\xi_n}}\to 0$. Similarly, one shows $\abs{\inner{\xi_n-\xi_l}{\xi_n}}\to 0$. Therefore, $\xi_k$ is Cauchy sequence and $\vop_\X x_k = \vop_\X x^{*} - \xi_k$ tends to an element {\rot $\vop_\X x^\plus$ with  $x^\plus \in \XX$. Because $\Vo$ has rank $\nb$ and  $\snorm{\QQ^{\Y}_{\ibs(i)}(y-\aop(x_i))} \to 0$, the element $x^\plus$ is a  solution of $\aop(x)=\data$.}  Further,
\[x_{k+1}-x_k \in \ran(\aop^*) \subseteq \nr( \aop )^{\perp}
\qquad \text{ for all $k \in \N$}\,. \]
Because $\nr( \aop  )^{\perp}$ is closed, its follows that $ x^* - x_0 \in \nr( \aop  )^{\perp}$.
Since $x^\plus$ is the only solution for which the latter holds true,  we obtain $x_k \to x^\plus$.
\end{proof}

\subsection{Convergence for noisy data}

In the noisy data case, we consider the loping version of the BCD. The iteration is
terminated when for the first time all $x^\delta_k$ are equal within a cycle.
That is, we stop the iteration at
\begin{equation} \label{eq:def-discrep}
    k_*^\delta  :=  \argmin \set{ k  \in \N \mid
    x_{k}^\delta = x_{k+1}^\delta = \cdots = x_{k+p-1}^\delta  } \,.
\end{equation}
{\rot To simplify the notation, we assume that $\delta_\ib =\delta$ for all $\ib \in \set{1, \dots , \nb}$.} We first show that the stopping index is always finite.

\begin{proposition}[Existence of stopping index] \label{prop:st-f}
If $\delta > 0$, then the stopping index   $k_*^\delta$ defined in \eqref{eq:def-discrep}
is finite, and we have
\begin{equation} \label{eq:sdk-monot-res}
\forall \ib = 1, \dots, \nb \colon \quad
\norm{\QQ^{\Y}_{\ib } \kl{ \ndata-\aop \kl{ x_{k_*^\delta}^\delta } } }  < \tau \delta
 \, .
\end{equation}
\end{proposition}

\begin{proof}
If for every  $k  \in \N$, there exists $\ell  \in
\set{0,\dots, p-1}$ such that   $x_{k +  \ell} \neq  x_{k}$, then
from Lemma~\ref{conv1} we obtain
\begin{multline} \label{eq:bound1}
     \forall n \in \N \colon
    \quad
 \norm{\vop_\X x_0  - \vop_\X x^*}^2 \geq
    (2-\theta_{\rm max})
    \gamma_{\rm min}
   \sum_{k =0}^{ n p - 1} d^\delta_k s^\delta_k  { \rot  \snorm{v_{\ibs(k)}}^2 (\res^\delta_k)^2}
    \\
    \geq (2-\theta_{\rm max})
    \gamma_{\rm min}
    C n p \tau \delta \,,
  \end{multline}
where $C>0$ is a lower bound of $s^\delta_k   \snorm{v_{\ibs(k)}}^2$.
The right hand side of \eqref{eq:bound1} tends to infinity,
which gives a contradiction. Consequently, the set
$ \{ k  \in \N \mid    x_{k}^\delta = x_{k+1}^\delta = \cdots = x_{k+p-1}^\delta  \} $
is nonempty and therefore contains  a finite minimal element.

To prove (\ref{eq:sdk-monot-res}) note that  the finiteness
of the stopping index  and the definition of the loping BCD implies
$\snorm{\QQ^{\Y}_{ \ibs(k_*^\delta+ \ell) } \skl{ \ndata-\aop \skl{ x_{k_*^\delta}^\delta } } }
< \tau \delta $ for $\ell = 0, \dots , p - 1$. The assumption  \ref{ass4} on the  control
sequence $\ibs(k)$ thus gives \eqref{eq:sdk-monot-res}.
\end{proof}

We call the step size selection $(s^\delta_k)_{k \in\N}$
continuous at $\delta =0$  if for all
$k \in \N$ we have
\begin{equation} \label{eq:conts}
\lim_{\delta \to 0}
\sup \{ \snorm{ s^\delta_k - s_k } \mid y^\delta  \in \YY
\wedge \snorm{y - \ndata} \leq \delta \}
    = 0 \,.
\end{equation}
An example for a continuous step size selection is the constant strep size
$s_k^\delta =  \gamma_{\rm min} / \snorm{\kop_{\nb}}^2$. The next auxiliary result shows that the continuity  of the step
size selection implies continuity of $x_{k}^\delta$ at $\delta =0$.

\begin{lemma}[Continuity of the BCD iteration at $\delta =0$] \label{lem:cont}
Suppose the step selection is continuous at $\delta =0$,
and define
\begin{equation*}
    \Delta_k(\delta, y, \ndata)
    \coloneqq
    d^\delta_k s^\delta_k \vop_\X \PP_{\ibs(k)} \aop^*(\aop (x_k^\delta)-\ndata)
    -
    s_k \vop_\Y  \PP_{\ibs(k)} \aop^*(\aop(x_k)-y) \,.
\end{equation*}
Then, for all $k \in \N$, we have
\begin{equation}\label{eq:cont}
\lim_{\delta \to 0}
     \sup
    \left\{ \snorm{ \Delta_k(\delta, y, \ndata) } \mid
    y^\delta  \in \YY
    \wedge
    \snorm{y - \ndata} \leq \delta \right\}
    = 0
\,.
\end{equation}
Moreover,  $x_{k+1}^\delta \to  x_{k+1}$, as $\delta \to 0$.
\end{lemma}

\begin{proof}
We prove Lemma \ref{lem:cont} by induction.
Assume $k \geq  0$ and that (\ref{eq:cont}) holds for all $k' < k$.
First we note that \eqref{eq:cont}  implies $x_{k+1}^\delta \to  x_{k+1}$, as $\delta \to 0$.
For the proof of  \eqref{eq:cont}  we consider two cases. In the
first case,  $d^\delta_k =1$, we have
 \begin{align*}
        \norm{ \Delta_k(\delta, y, y^\delta) }
        &
        = \norm{ s^\delta_k \vop_\X  \PP_{\ibs(k)} \aop^*(\aop (x_k^\delta)-\ndata)
        -
        s_k \vop_\X  \PP_{\ibs(k)} \aop^*(\aop(x_k)-y))} \,.
\end{align*}
In the second case, $d^\delta_k =0$, we have
$\norm{ \QQ^{\Y}_{\ibs(k)} (\ndata-\aop( x_{k}^\delta)) } < \tau \delta$.
Consequently,
    \begin{align*}
    \|\Delta_k&(\delta, y, y^\delta)\| \\
    &=
    \norm{ s_k \vop_\X  \PP_{\ibs(k)} \aop^*(\aop(x_k)-y))} \\
    &=
    \norm{ s_k \vop_\X  \PP_{\ibs(k)} \vop_\X^* \kop_\nd^*(\aop(x_k)-y))} \\
    &=
    \norm{v_\ib}^2 \norm{  \QQ^{\Y}_{\ibs(k)} \kop_\nd^*(\aop(x_k)-y))} \\
    &\leq
    \norm{v_\ib}^2
    \norm{\kop_{\nd}} \,
    \Bigl(
    \norm{ \QQ^{\Y}_{\ibs(k)} (\aop (x_k^\delta)-\aop(x_k))}
    +\norm{ \QQ^{\Y}_{\ibs(k)} (\aop (x_k^\delta)- \ndata )}
    \\ & \hspace{0.2\textwidth} + \norm{ \QQ^{\Y}_{\ibs(k)} (\ndata- y )}
    \Bigr)
     \\
    &\leq
    \norm{v_\ib}^2 \norm{\kop_{\nd}} \,
    \Bigl(
    \norm{  \QQ^{\Y}_{\ibs(k)} (\aop (x_k^\delta)-\aop(x_k))}
    +(\snorm{\QQ^{\Y}_{\ibs(k)}} + \tau) \delta
    \Bigr)
    \,.
    \end{align*}
Now \eqref{eq:cont} follows from  the continuity of
$\aop$, and the induction hypothesis implying $x_k^\delta \to x_k$.
\end{proof}

\begin{theorem}[Convergence of the loping BCD for noisy data] \label{th:noisy}
Suppose the step selection $(s^\delta_k)_{k \in\N}$ is continuous at $\delta =0$.
Let  $(\delta_j)_{j\in \N} \in (0,\infty)^\N$ converge to  $0$ and let
$(y_j)\in \YY^\N$ be a sequence of noisy data with
$\snorm{\QQ_{\ib}^\Y( y_j - y) } \leq \delta_j$.
Let $(x_{j,k})_{k \in  \N}$ be defined by the loping BCD iteration with data $y_j$
and stopped at $k_j := k_*(\delta_j, y_j) $ according to 
\eqref{eq:def-discrep}.
Then $(x_{j,k_j})_{j \in \N} \to x^\plus$, where $x^\plus$ is the solution  of $\aop(x)=y$
with minimal distance to $x_0$.
\end{theorem}

\begin{proof}
From  Lemma
\ref{lem:cont} and the continuity of $\aop$ we have, for any
$k \in \N$,  that $x_{j,k} \to x_{k}$ and $\aop  ( x_{j,k} ) \to  \aop (x_{k})$
as $j \to \infty $.

To show that $x_{j,k_j} \to x^\plus$, we first assume that
$k_j$ has a finite accumulation point $k_*$. Without loss of
generality  we may assume that $k_j = k_*$ for all $j \in \N$. From
Proposition \ref{prop:st-f} we know that $\snorm{\QQ^{\Y}_{\ib} ( y_j -
\aop (x_{j_,k_*}))} < \tau \delta_j$. By taking
the limit $j \to \infty$, we obtain  $ y = \aop (x_{k_*}) $.
Consequently, $x_{k_*} = x^\plus$ and $x_{j_,k_*} \to x^{*}$ as
$j \to \infty$.
It remains to consider the case where $k_j \to \infty$ as $j \to
\infty$. To that end let $\eps >0$. Without loss of
generality we assume that $k_j$ is monotonically increasing.
According to Theorem \ref{thm:exact}  we can choose $n \in \N$ such
that $\norm{\vop_\X x_{k_n}  -  \vop_\X  x^\plus} < \eps/2$. Equation
\eqref{eq:cont} implies that there exists $j_0  >  n$ such that
$\norm{\vop_\X  x_{j,k_n} - \vop_\X  x_{k_n}} < \eps/2$ for all $j \geq j_0$.
Together with the monotonicity we obtain
\begin{multline*}
    \norm{\vop_\X  x_{j,k_j}  - \vop_\X  x^\plus}
    \leq
    \norm{\vop_\X  x_{j,k_n}  - \vop_\X  x^\plus}
    \\
    \leq \norm{\vop_\X  x_{j,k_n}  - \vop_\X  x_{k_n}}
    +
    \norm{\vop_\X x_{k_n}  - \vop_\X  x^\plus}
    < \frac{\eps}{2} + \frac{\eps}{2} = \eps
     \quad \text{ for } j \geq j_0 \,.
\end{multline*}
Because $\vop_\Y $ is non-singular, this shows $x_{j,k_j} \to x^\plus$ as $j \to \infty$.
\end{proof}

{\rot

\section{Example: System of  linear integral equation}
\label{sec:int}

In this section we   compare the BCD method to the standard Landweber method for
 an elementary system of linear integral equations.

\begin{figure}[htb!]\centering
\includegraphics[width=0.49\textwidth]{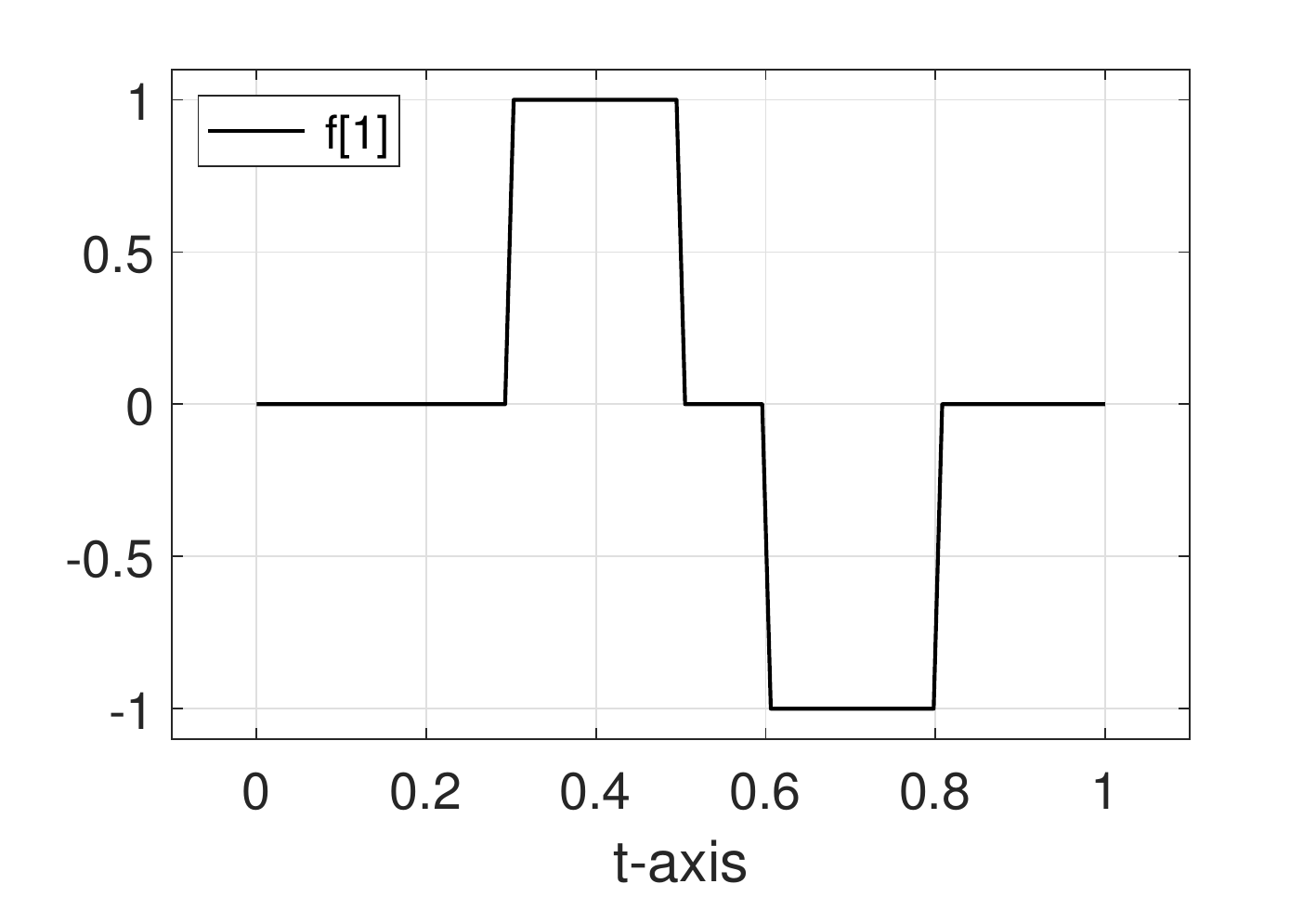}
\includegraphics[width=0.49\textwidth]{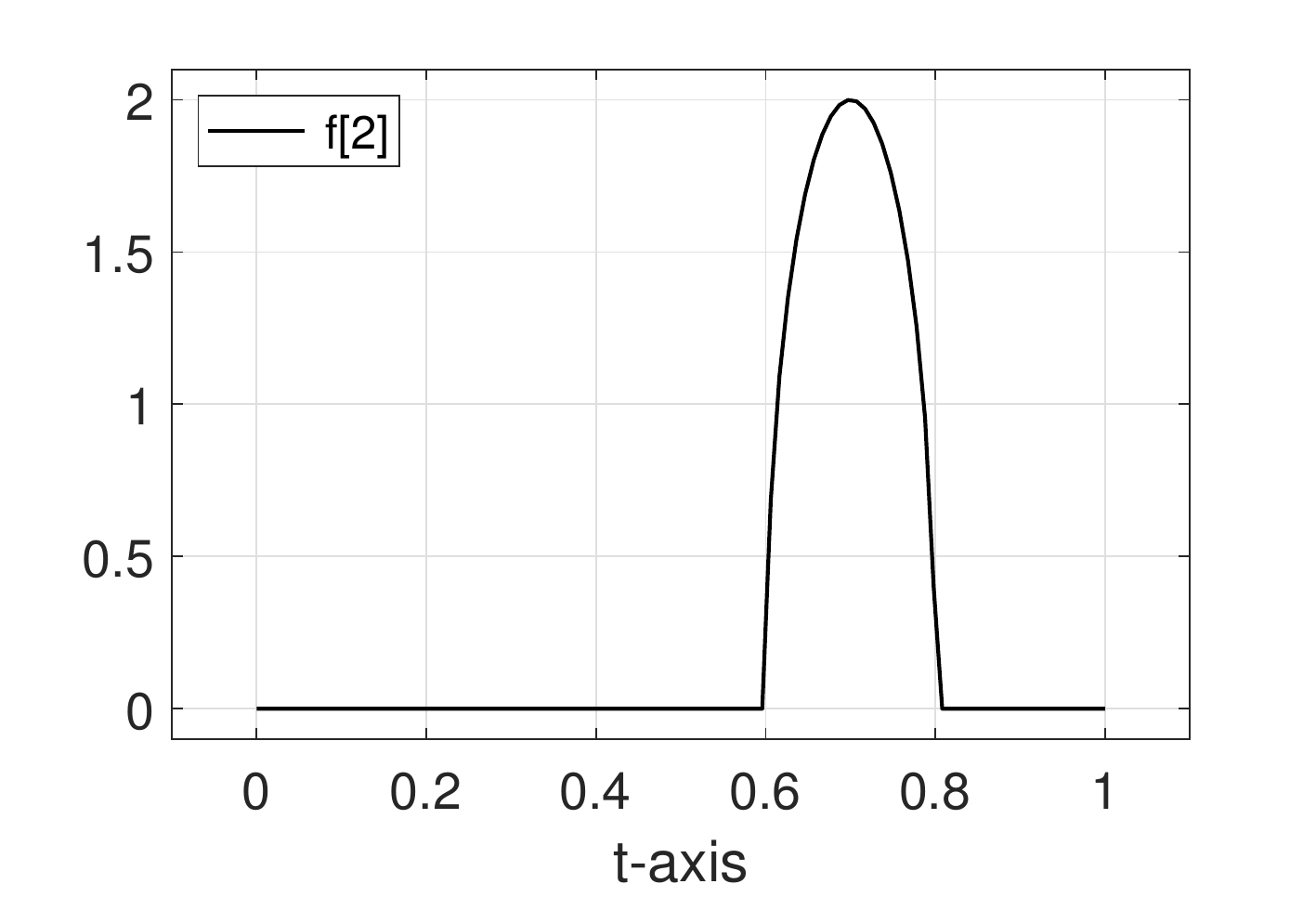}\\
\includegraphics[width=0.49\textwidth]{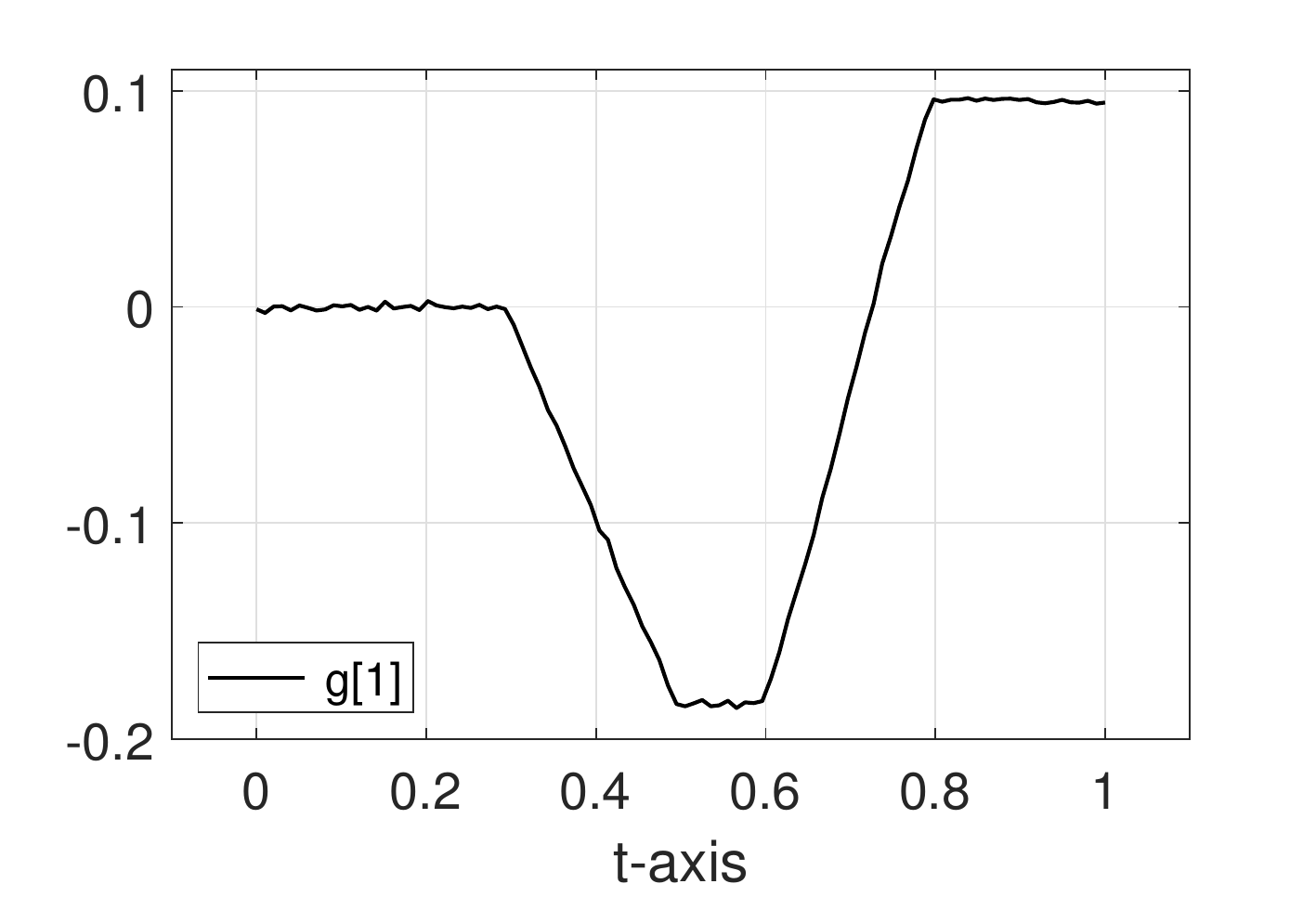}
\includegraphics[width=0.49\textwidth]{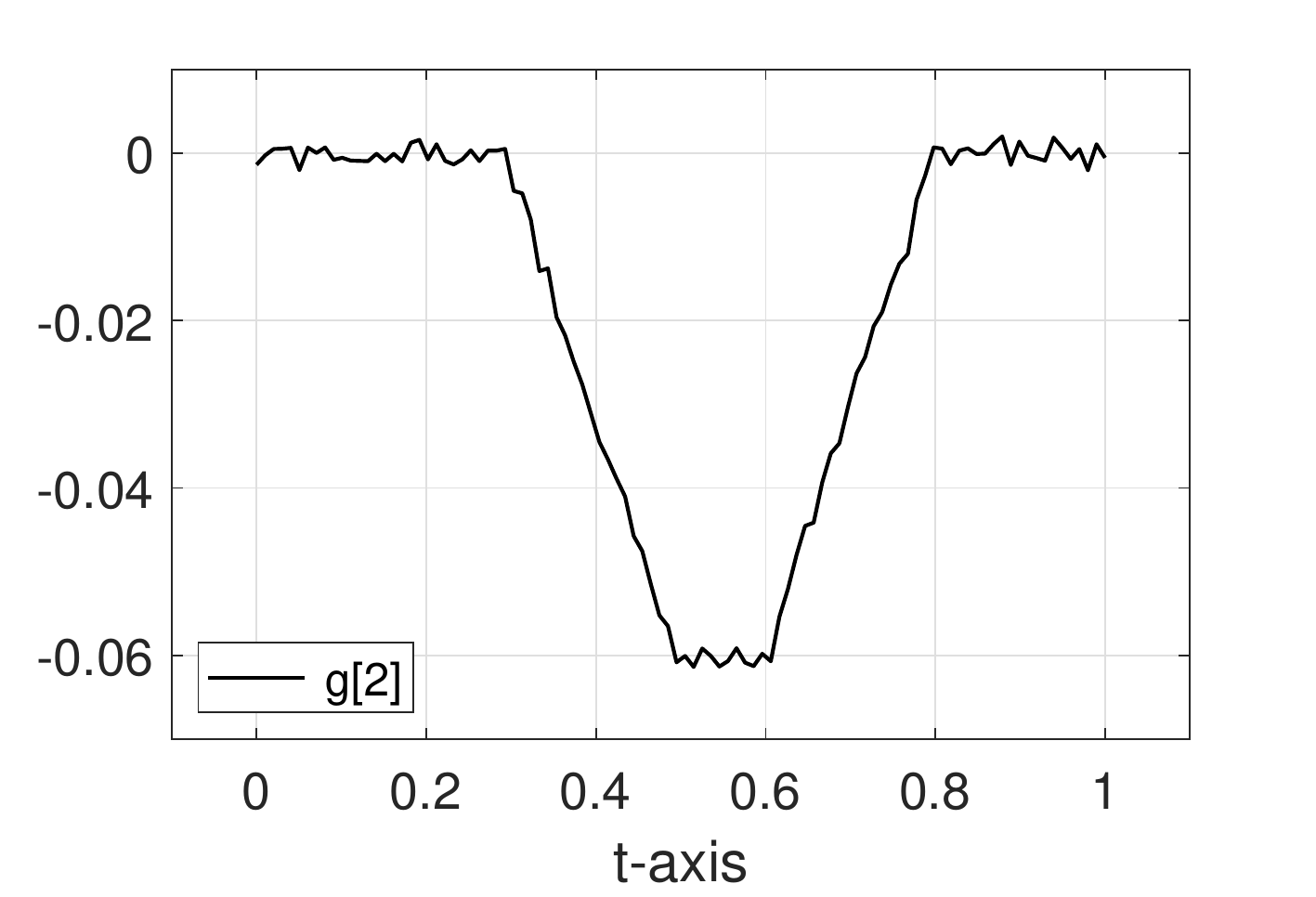}
\caption{\rot \textbf{Test phantoms and noisy data for a system of two integral equations.}
Top: The two components  $f^*[1]$ (left) and $f^*[2]$ (right) of the  true  unknown.
Bottom: The two components $g^\delta[1]$ (left) and $g^\delta[2]$ (right) of
the computed noisy data.}
\label{fig:I-data}
\end{figure}

 \subsection{Forward problem}

 Consider the integration operator
 $\Ko \colon L^2([0,1]) \to L^2([0,1])$ defined by
\begin{equation} \label{eq:intop}
\Ko(f) \colon [0,1] \to \R \colon s \mapsto \int_0^s f(t) \rmd t \,.
\end{equation}
According to the Cauchy-Schwarz inequality, we have
\begin{equation}
	 \norm{\Ko(f)}^2 = \int_0^1 \kl{\int_0^s   f(t) \rmd t }^2  \rmd s
	 \leq \int_0^1 s \int_0^1 f(t)^2 \rmd t \rmd s = \frac{1}{2} \norm{f}^2
\end{equation}
for all $ f \in L^2([0,1])$, which shows that $\Ko$ is a well-defined linear bounded operator.
Using the operator $\Ko$ we consider the following  forward model applied to
a vector  of functions $(f[\ib])_{\ib=1}^\nb  \in \skl{ L^2([0,1]) }^\nb$.

\begin{figure}[hbt!]\centering
\includegraphics[width=0.49\textwidth]{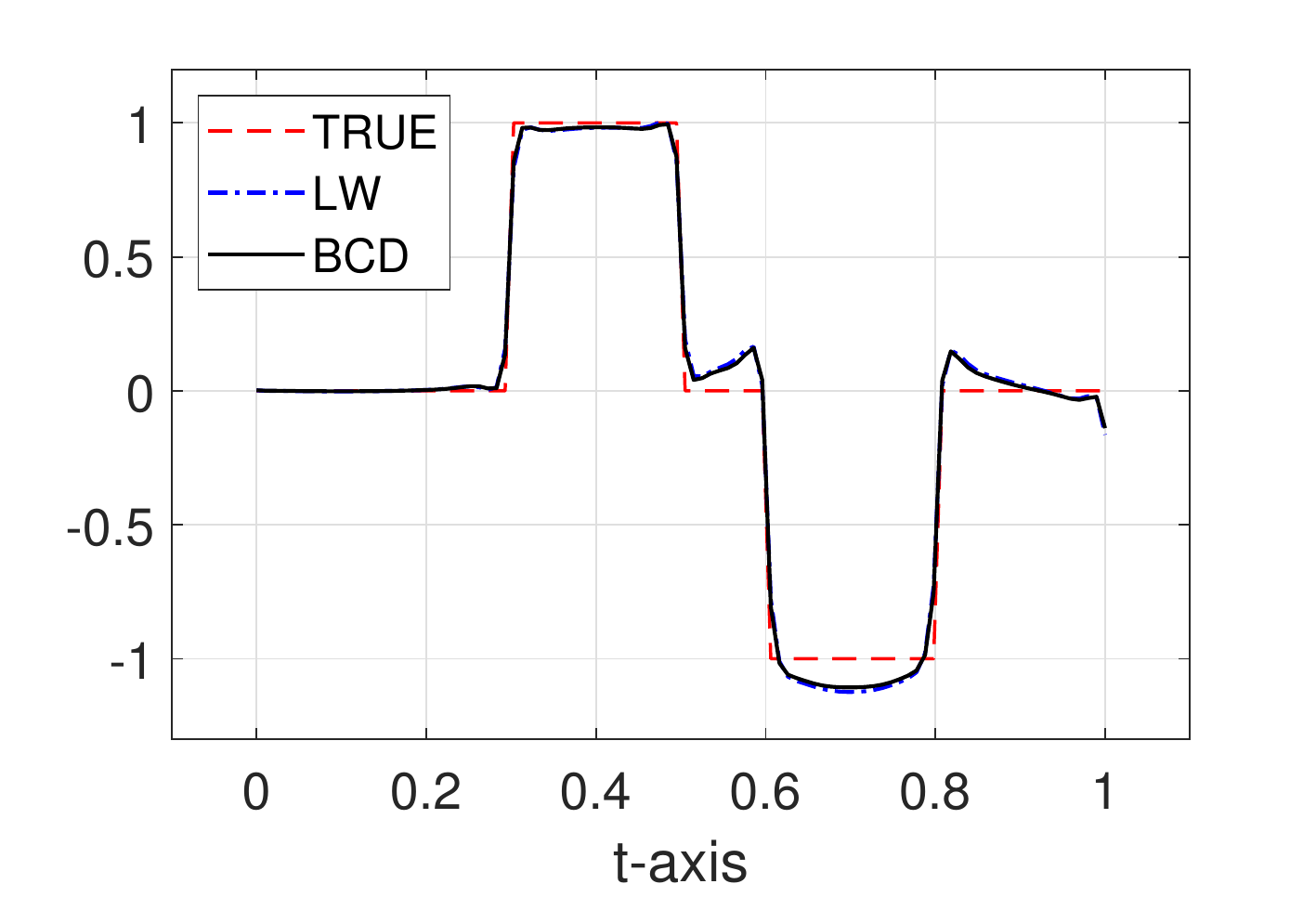}
\includegraphics[width=0.49\textwidth]{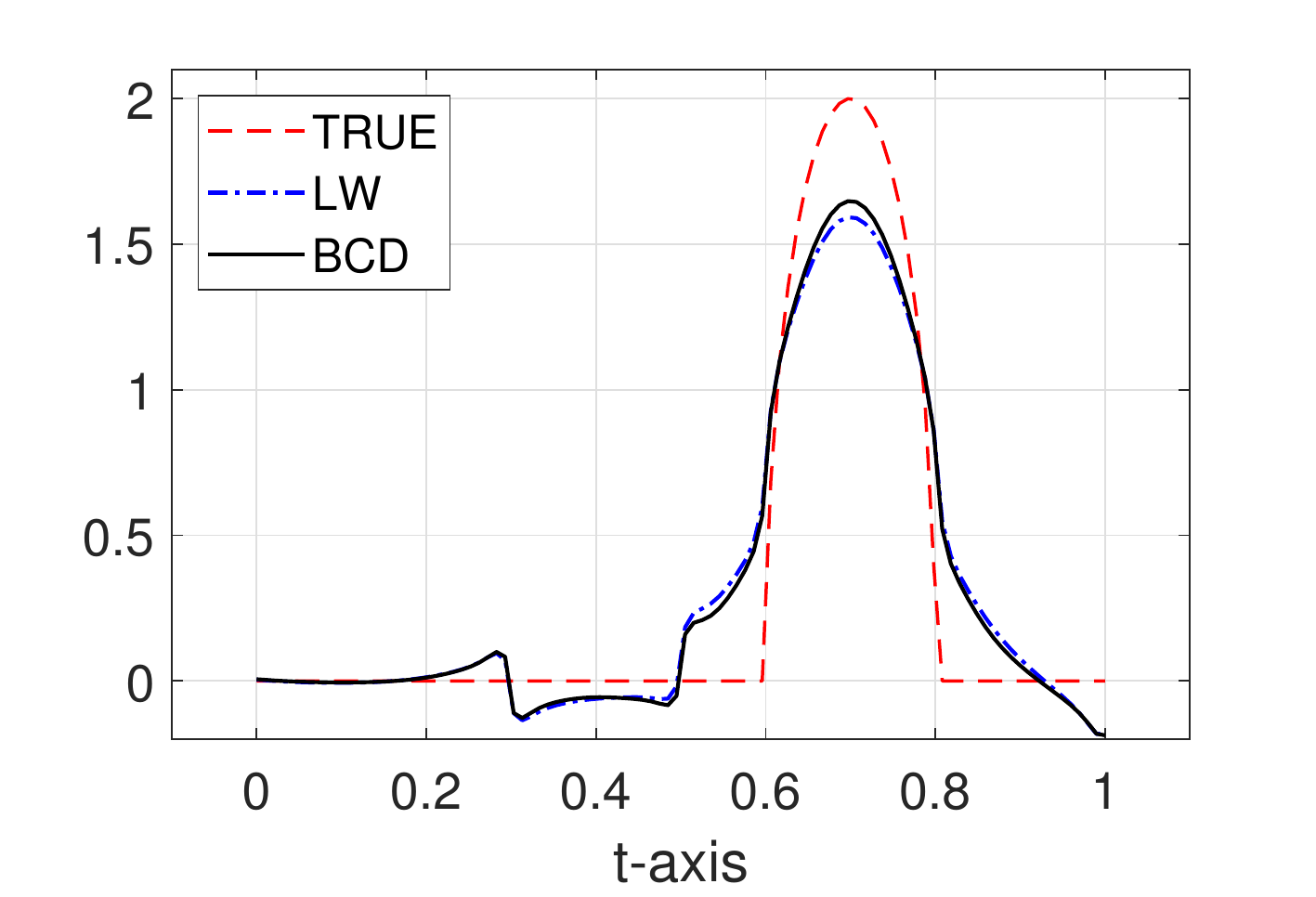}\\
\includegraphics[width=0.49\textwidth]{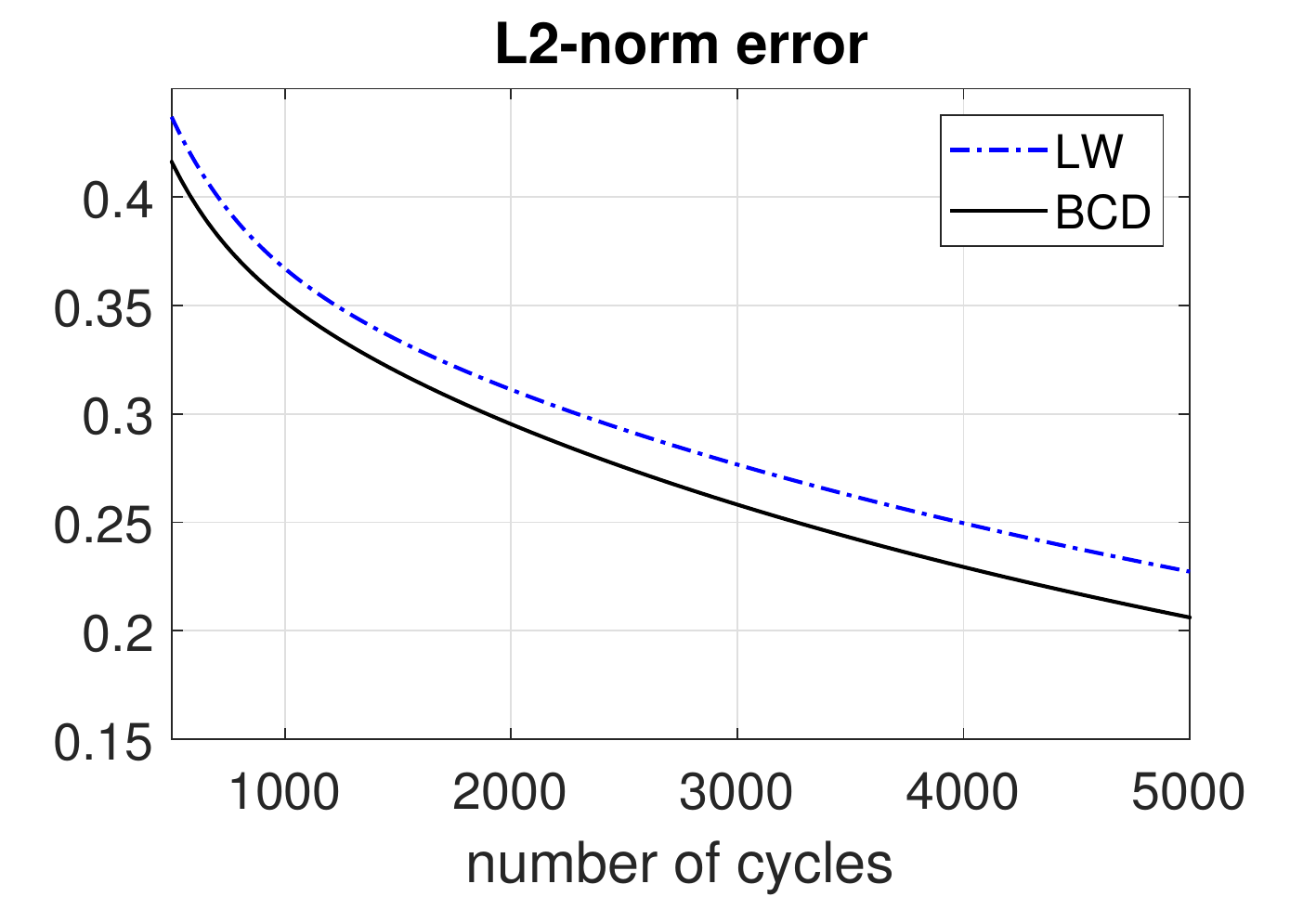}
\includegraphics[width=0.49\textwidth]{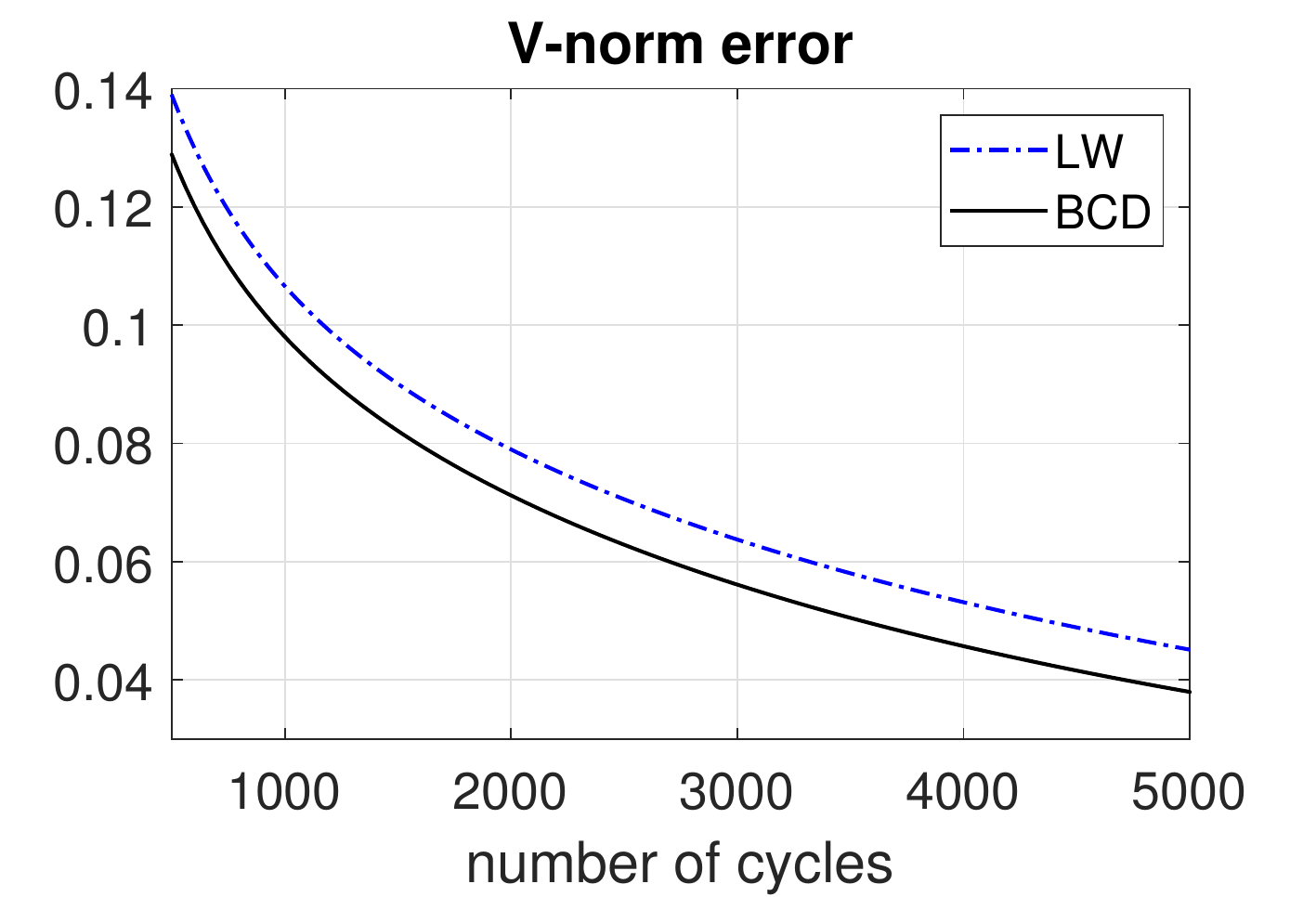}
\caption{\rot \textbf{Reconstruction from exact data using the Landweber (LW) and the BCD method.} Top: Reconstruction after 5000 cycles. Bottom: Reconstruction errors in the
standard $2$-norm (left) and the $\Vo$-norm (right) as a function of the iteration number.
For both error measures the  reconstruction error of BCD  is  smaller than the
one of the Landweber method.}
\label{fig:I-exact}
\end{figure}

\begin{definition}
For $\nd \geq \nb \geq 1$ and given matrix $V =(v_{d,b})_{d,b}\in \R^{D \times B}$
of  rank $\nb$, we define the forward operator
\begin{equation}
\aop \colon \skl{ L^2([0,1]) }^\nb \to \skl{ L^2([0,1])}^\nd  \colon f \mapsto \kl{ \sum_{b=1}^B  v_{b,c} \Ko(f[\ib])}_{\id=1}^\nd \,.
\end{equation}
\end{definition}

According to our general notion we have  $\aop  = V \otimes \Ko$ and the theory
presented in the previous section can be applied for solving the inverse
problem $\aop(f) = g$. Note that this equation clearly is  ill-posed because the
range of $\Ko$  is non-closed (and equal to the Sobolev space
$H^1_{\diamond}([0,1]) \coloneqq \set{g \in L^2([0,1]) \mid g' \in L^2([0,1]) \wedge g(0)=0}$
of  all weakly differentiable functions vanishing at $0$.)

More generally, one could replace the integration operator by  any  bounded
(integral) operator $\Ko \colon L^2([0,1]) \to L^2([0,1])$ with non-closed range.


\begin{figure}[hbt!]\centering
\includegraphics[width=0.49\textwidth]{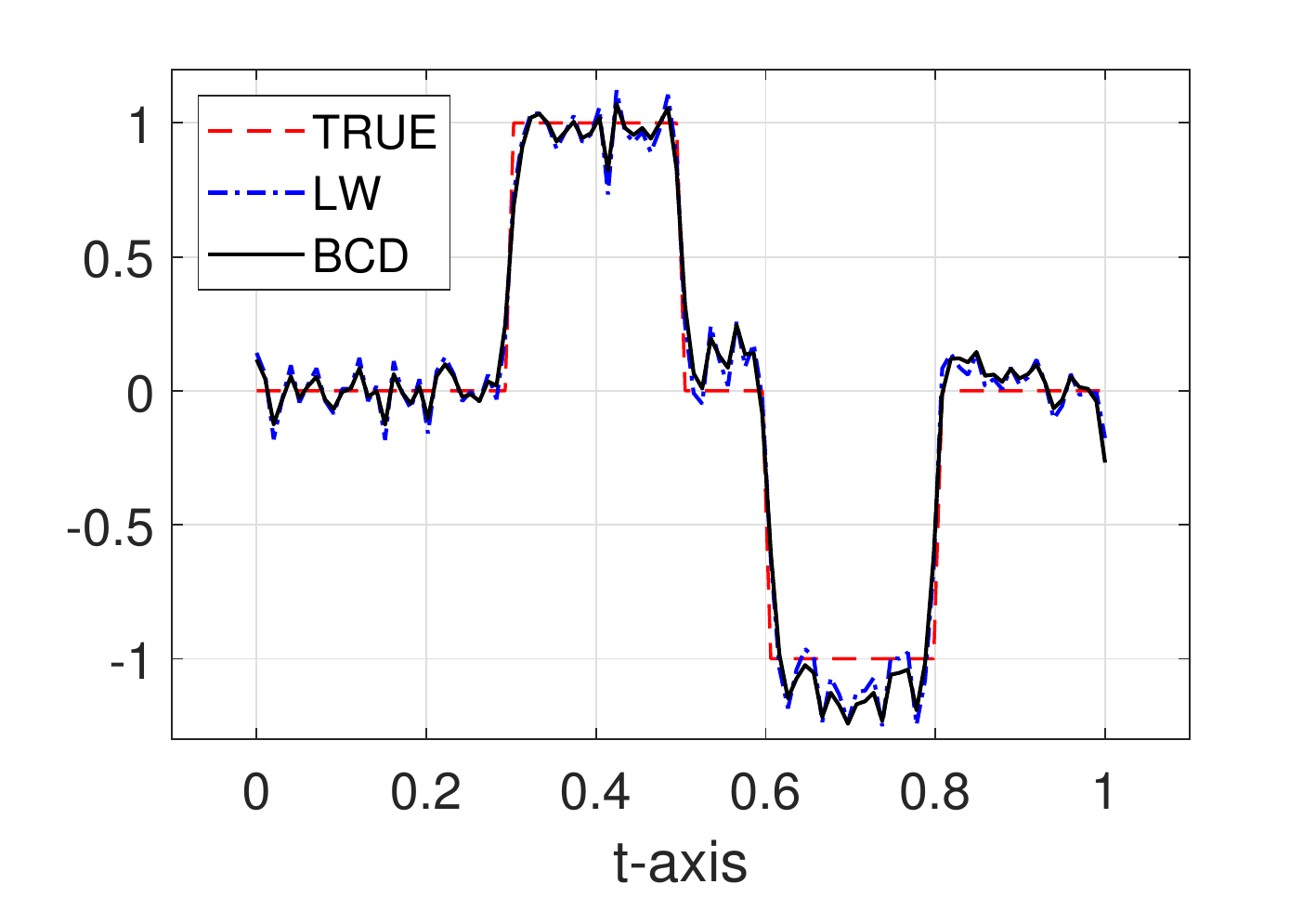}
\includegraphics[width=0.49\textwidth]{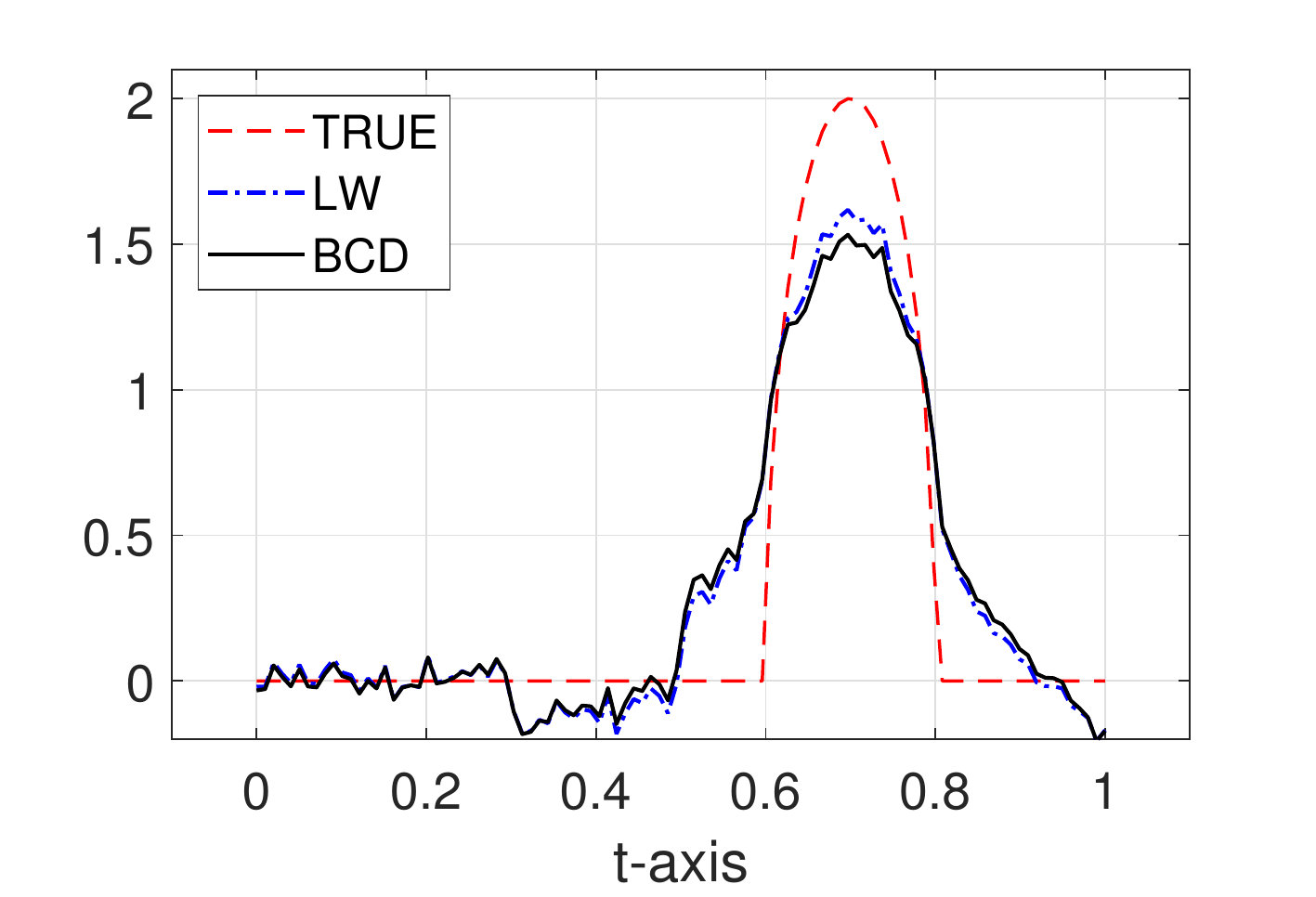}\\
\includegraphics[width=0.49\textwidth]{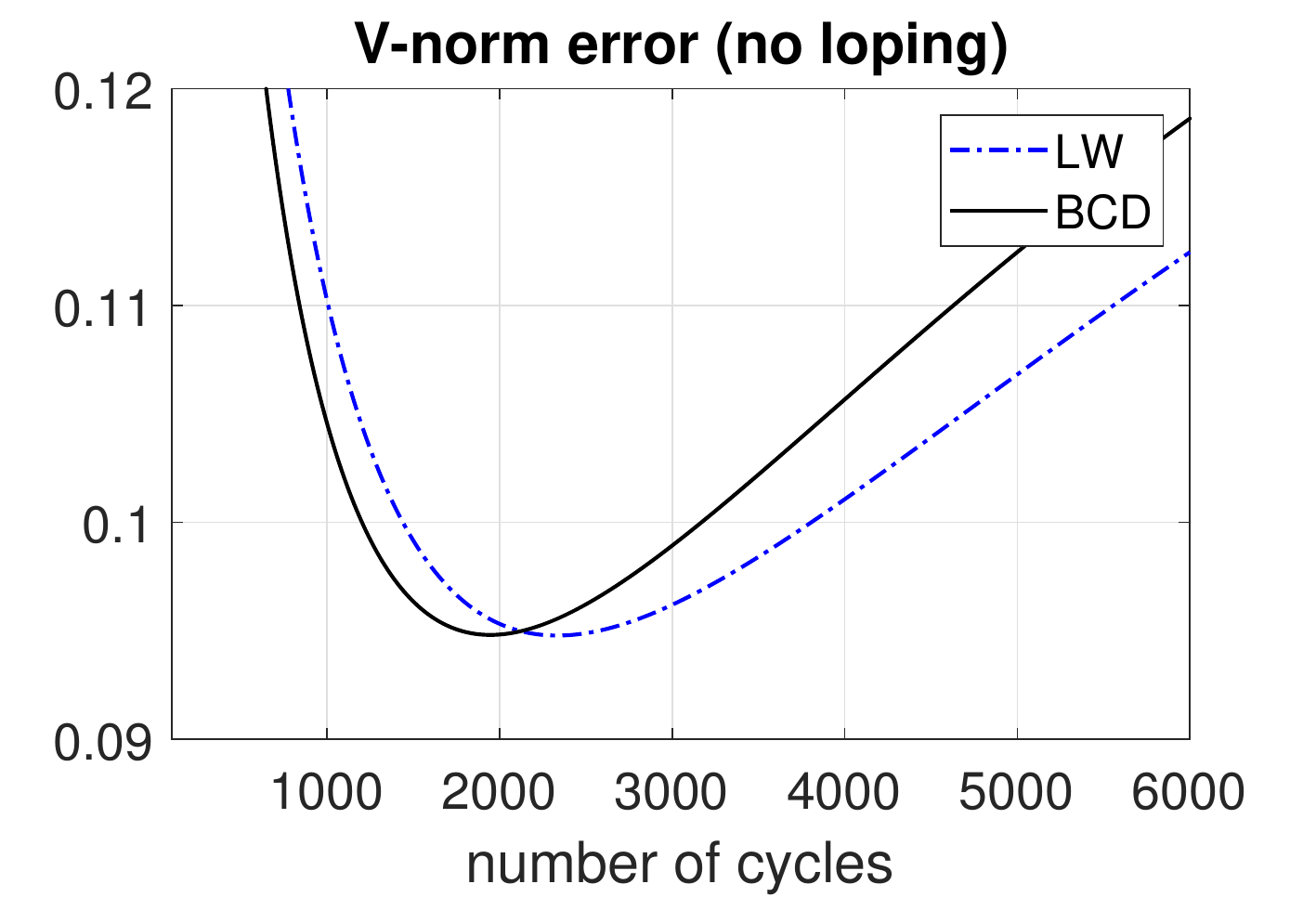}
\includegraphics[width=0.49\textwidth]{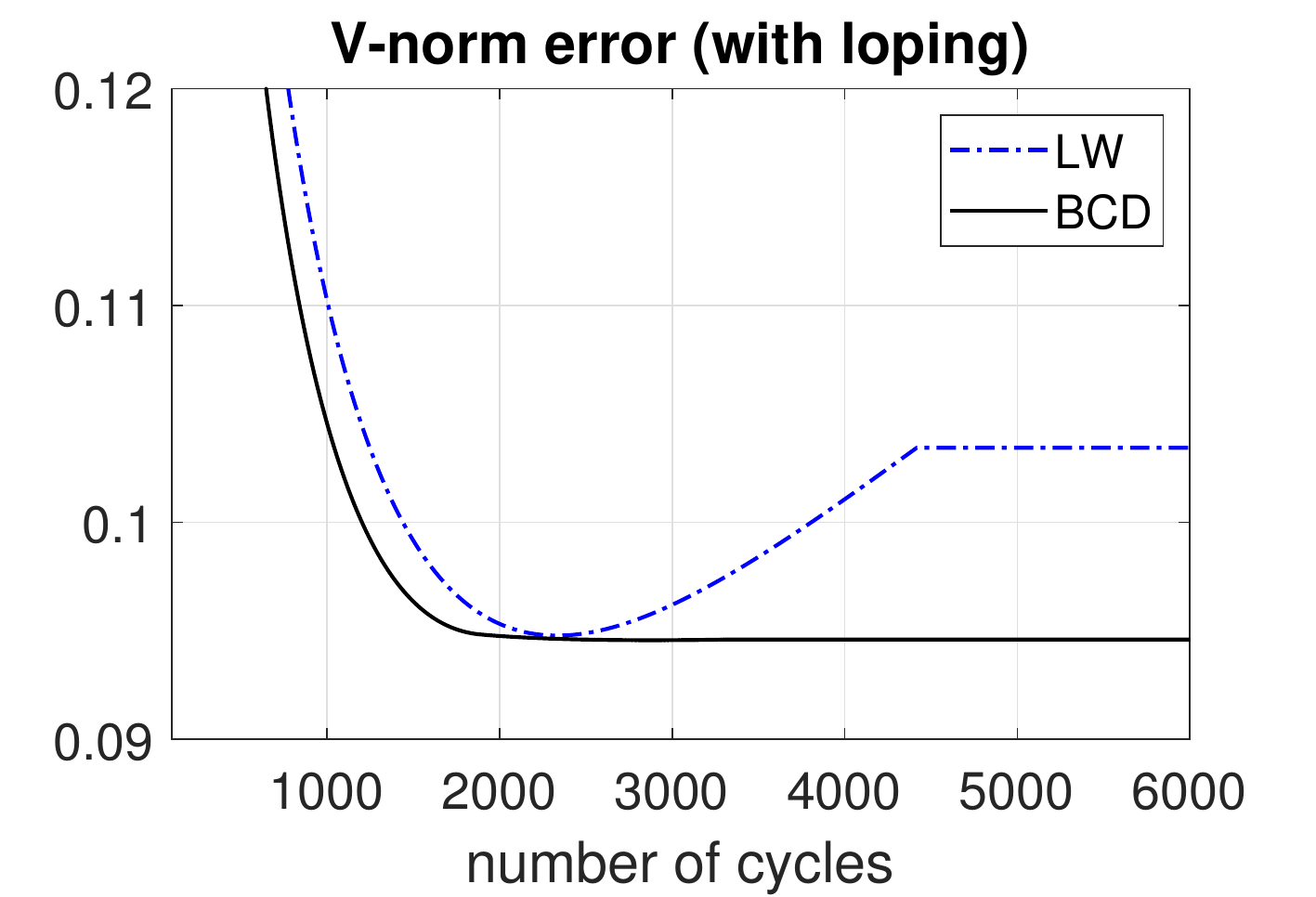}\\
\includegraphics[width=0.49\textwidth]{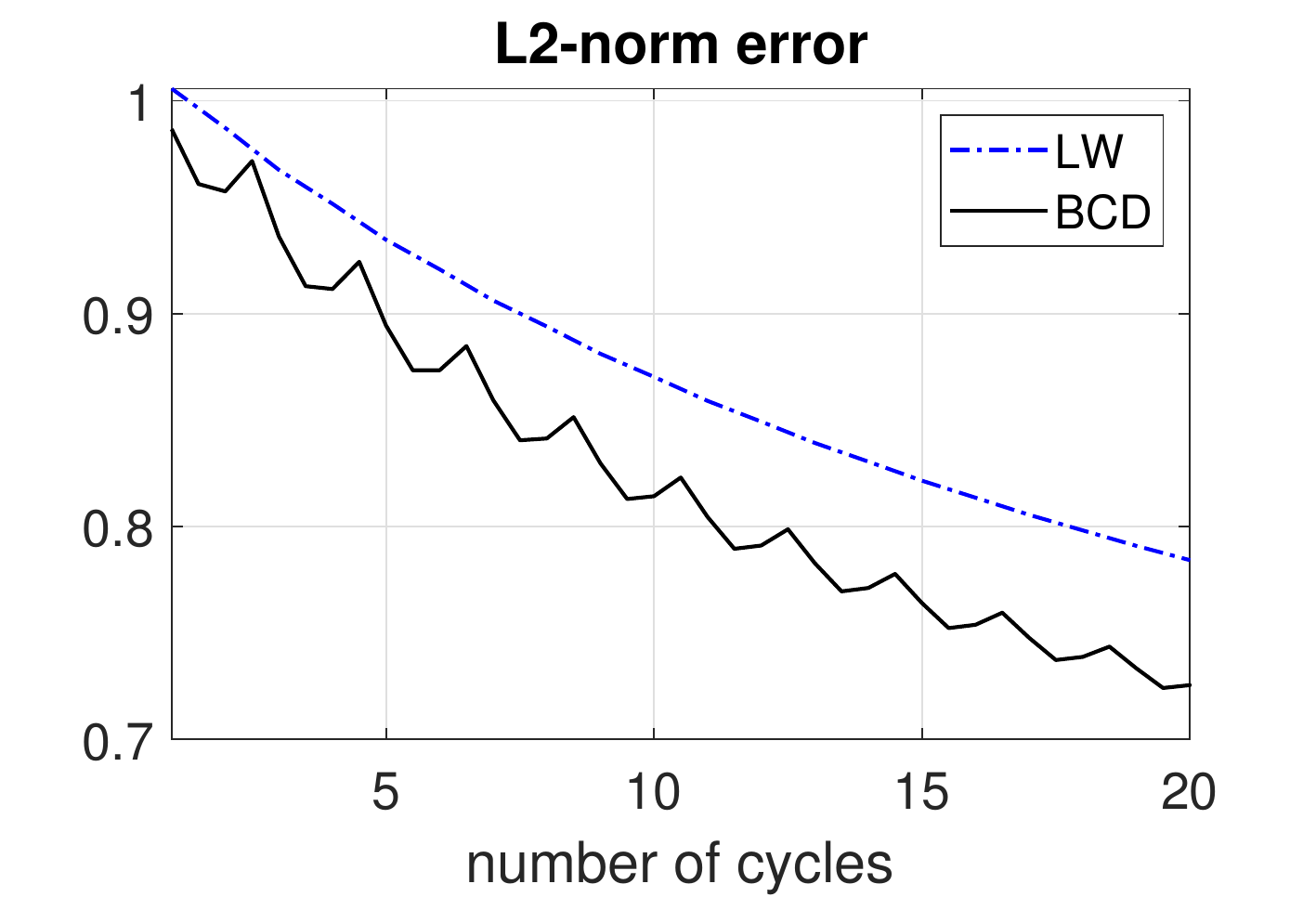}
\includegraphics[width=0.49\textwidth]{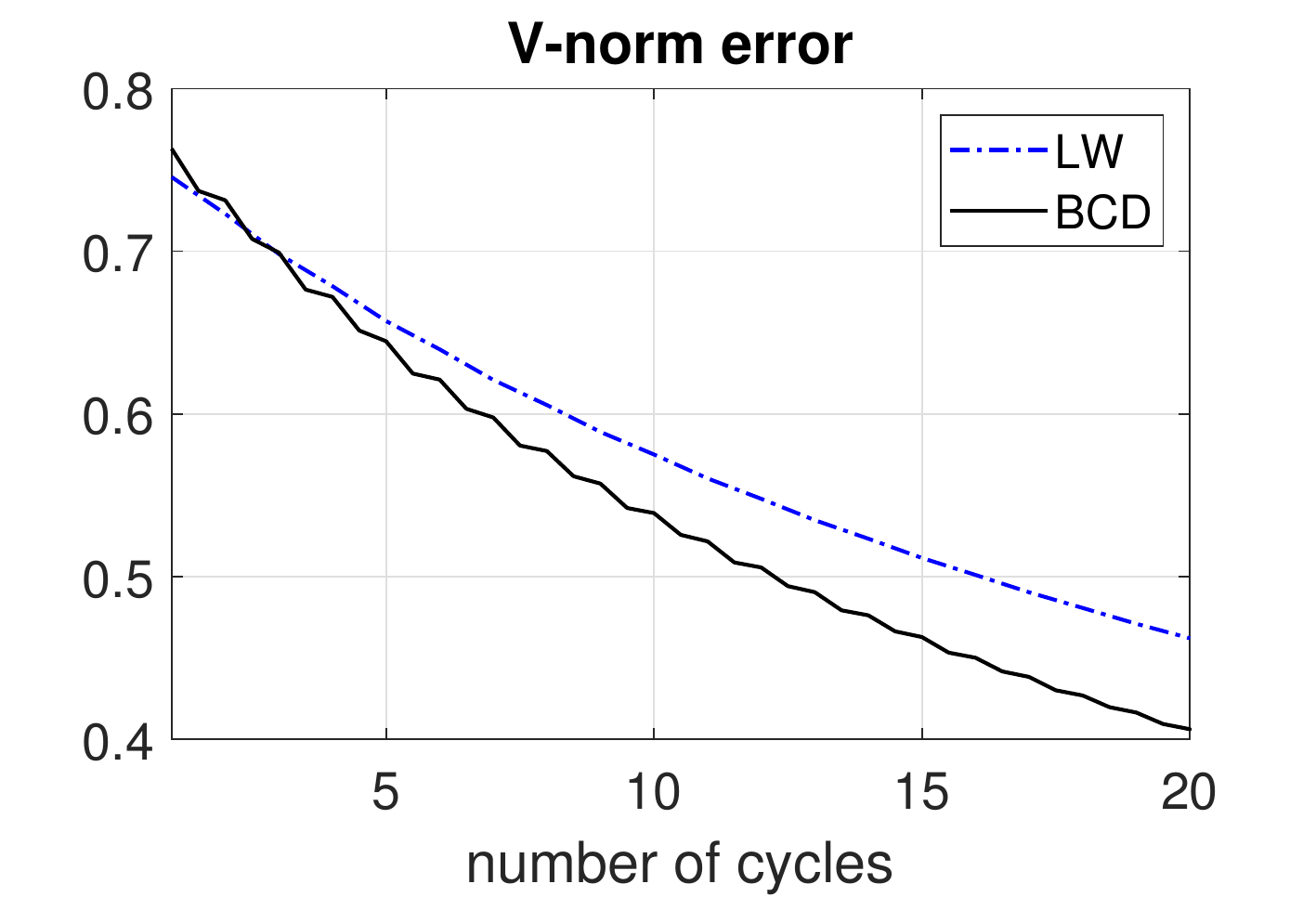}
\caption{\rot \textbf{Reconstructions from noisy data using the Landweber (LW) and the BCD method.} Top: Reconstruction using the BCD iteration (with the loping principle and the proposed stopping  rule) and the Landweber method using the  discrepancy principle as stopping rule. Middle: Reconstruction errors in the  $\Vo$-norm without (left) and
with (right) loping.  Bottom: reconstruction error for the first iterates
in the $2$-norm (not monotonically decreasing) and in the $\Vo$-norm (right).}
\label{fig:I-noisy}
\end{figure}

\subsection{Reconstruction results}

For all  presented numerical results we  use $\nb=\nd=2$
and   take $\Vo = \tilde \Vo /{\snorm{\tilde \Vo}}_{2,2}$ with
\begin{equation}
\tilde \Vo \coloneqq \begin{bmatrix} -3 &    1 \\  -1  &  0 \end{bmatrix} \,.
\end{equation}
We discretize $\Ko$ with the composite trapezoidal rule using $p = 100$ intervals
such that the data and the unknowns are elements in $(\R^p)^2 $.
The true unknown $f^* =(f^*[1], f^*[2])$ and the  noisy data $g^\delta  =(g^\delta[1], g^\delta[2])  $  are shown  in Figure \ref{fig:I-data}.   The exact data $g = \aop f^*$ have been computed via numerical integration followed by application of $\Vo$. Subsequently we computed noisy  data by adding random  white noise to  $\data$ with a standard deviation of $0.001$. The resulting relative data errors are  $\norm{g-g^\delta} / \norm{ g  } \simeq 0.015$,
$\norm{\QQ_1(g-g^\delta)} / \norm{\QQ_1 g } \simeq 0.011$ and
$\norm{\QQ_2(g-g^\delta)} / \norm{\QQ_2 g } \simeq 0.012$,
respectively.

Reconstruction using the BCD and Landweber methods   from simulated data
are shown  in Figure \ref{fig:I-exact}.  For each  case we have used the maximum constant step-size,
that lead to  stable reconstruction. We evaluate the reconstruction error (norm of $f_k -  f^*$) in terms of the
standard  2-norm  $\norm{\edot}_2$  and in  the  $\Vo$-norm $\norm{\edot}_{\Vo}$,
\begin{align}
	\norm{f}_2^2
	&\coloneqq
	\norm{f[1]}^2 + \norm{f[2]}^2\\
	\norm{f}_{\Vo}^2
	&\coloneqq
	\norm{v_{1,1} f[1] + v_{1,2} f[2]}^2 +\norm{v_{2,1} f[1] + v_{2,2} f[2]}^2 \,,
\end{align}
respectively.
As we can see from the bottom row in Figure \ref{fig:I-exact}, measured in both norms,
the reconstruction error  of the BCD   is smaller  than the error of
Landweber iteration.

Figure \ref{fig:I-noisy} shows  reconstruction results for nosy data. Again, the error in the BCD
method decreases faster than the one of the Landweber method. The BCD  therefore  requires less cycles than the Landweber  method. Moreover, in the middle column of
Figure~\ref{fig:I-noisy}   we illustrate the need for the loping (or another  regularization  strategy). Without loping, the BCD iteration as well as the Landweber
iteration start to diverge after around 2000 iterations. With loping (for the BCD method) and the with the discrepancy principle (for the Landweber method)   both iterations stop. (Note that here we only show the error in the $\Vo$-norm and that  the Landweber  method is monotonically decreasing in the $2$-norm when using the discrepancy principle.)
Finally, the bottom row in Figure~\ref{fig:I-noisy}  shows that the reconstruction
error for the BCD iteration is not monotonically decreasing in the standard norm,
whereas in the $\Vo$-norm it is.

%


}

\section{\rot A nonlinear test: Multi-spectral X-ray tomography}
\label{sec:xray}

In this section we apply {\rot a nonlinear generalization of the} BCD and the Landweber iteration to one-step inversion in  multi-spectral X-ray tomography. {\rot In particular, for nonlinear operators  $\aop$ in place of linear ones,  we use the following generalization of the BCD iteration
\begin{equation} \label{eq:iterX}
    x_{k+1}^\delta
    \coloneqq  x_k^\delta - s^\delta_k \PP_{\ibs(k)} \aop'(x_k^\delta)^*(\aop (x_k^\delta)-y^\delta)\,.
\end{equation}
Note that such problems are  not covered be our theoretical analysis.
We consider  extending   our  theory to this  class of examples a particularly
interesting topic of future research.}

In the following we denote by $D_R \subseteq \R^2$
the disc with radius $R < 1$ centered at the origin.
We define the fan beam Radon transform  $\Xo \mu \colon \sph^1 \times \sph^1  \to \R$
of a function $\mu \colon \R^2 \to \R$  supported in $D_R$ by
\begin{equation} \label{eq:radon}
    (\Xo\mu)(\alpha,\beta) \coloneqq  \int_{\R}\mu( \alpha + t\beta )\rmd t\,.
\end{equation}
It can be easily verified that the fan beam Radon transform
$\Xo \colon L^2(D_R) \to L^2(\sph^1  \times \sph^1)$
is linear and continuous \cite{natterer2001mathematics}.

\subsection{Mathematical modeling}

We assume that the tissue is composed of $\nb$ different materials
each of them having a different energy dependent X-ray  attenuation coefficient
$\mu_\ib (E)$ with $\ib = 1, \dots, \nb$. The combined X-ray attenuation coefficient is then given by
\begin{equation}
    \mu( E, \edot ) = \sum_{\ib = 1}^{\nb}\mu_\ib(E) f[\ib]    \,,
\end{equation}
where $f[\ib]  \colon \R^2 \to [0,1]$ is the fractional density map of the $\ib$th
material. Our goal is to determine the fractional
density maps  $f[\ib] $ from multi-spectral X-ray transmission measurements.

\begin{figure}[htb!]
\centering
\includegraphics[width=0.7\textwidth]{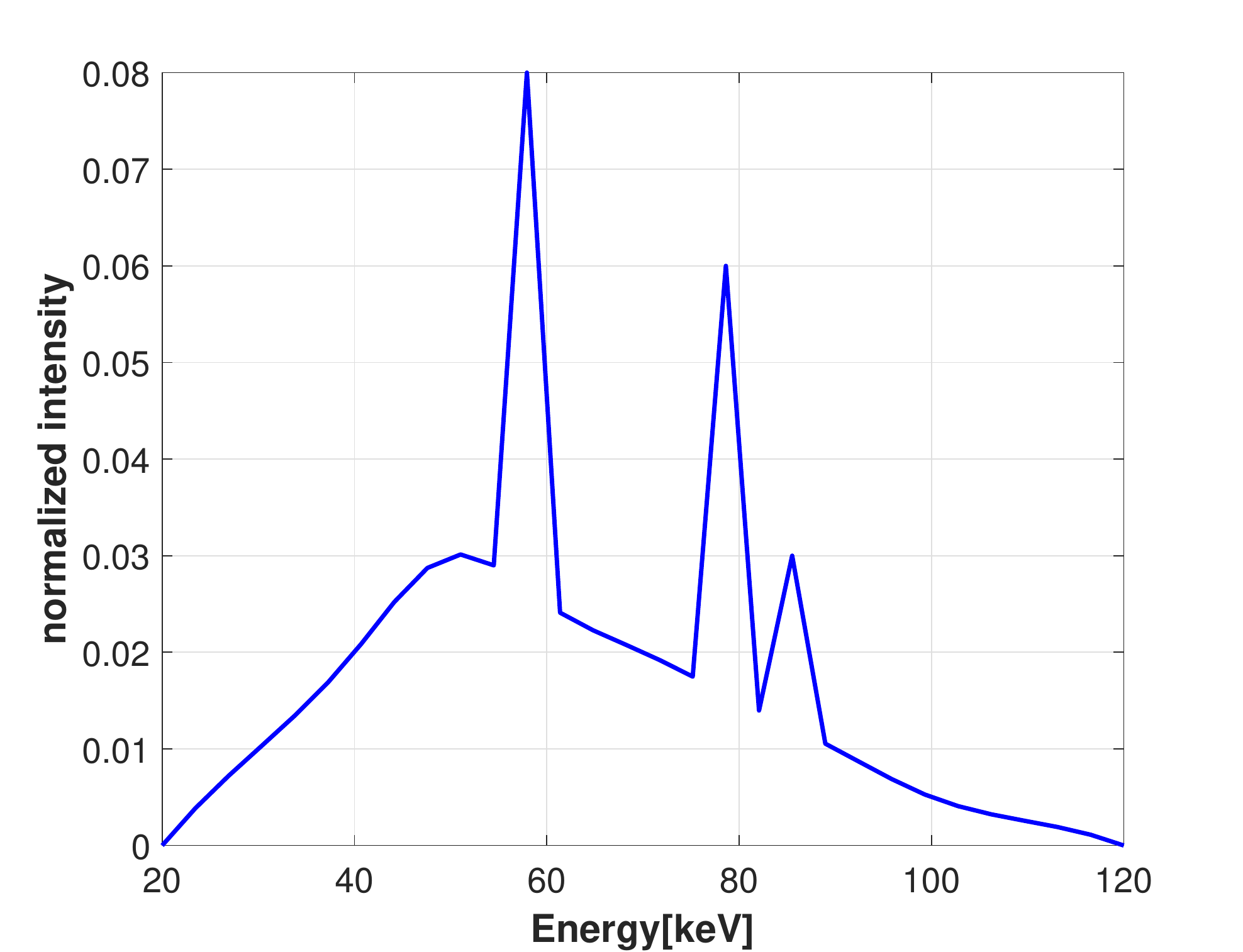}
\caption{Normalized\label{spec}  spectrum of a typical X-ray source \cite{atak2015dual,barber2016algorithm}. This spectral energy
distribution will be considered for our experiments.}
\end{figure}

The energy sensitive X-ray transmission  measurements result in the intensity~\cite{barber2016algorithm}
\begin{equation}\label{xrayint}
    I_W = \int_{W} s(E)\exp\kl{ -\Xo(\mu(E,\edot ))} \rmd E\,.
\end{equation}
Here $W \subseteq [0, \infty)$ denotes the energy window where the measurement is made and
$s \colon [0, \infty) \to \R$ is the product of X-ray beam spectrum
intensity and detector sensitivity.
We assume the detector sensitivity to be
constant and that the spectrum $s $ is known for energies ranging from
$\SI{20}{\keV}$ to $\SI{120}{\keV}$ covering any energy window.
The spectrum used for the numerical results is the same as in \cite{atak2015dual,barber2016algorithm}
and  shown in  Figure~\ref{spec}.

In order to recover multiple material densities,
we use multiple energy windows. We choose the same number
$\nb$ of  spectral windows as we have different materials.
Moreover, to simplify the mathematical formulation we
uniformly discretize the energy variable, $E_0 = \SI{20}{\keV}  < E_1 < \cdots <
 E_N =\SI{120}{\keV}$.
The X-ray measurements corresponding
to the $\ib$th energy  window is given by
\begin{equation}\label{xraysum}
    I[\ib] =
    \sum_{i \in W_\ib}
    s_i \exp(-\Xo( \mu_i ))
    \, \Delta E
    =
    \sum_{i \in W_\ib}
    s_i \exp \kl{ - \Xo \kl{ \sum_{\ib = 1}^{\nb} \mu_{i,\ib} f[\ib] } }  \,  \Delta E  \,.
\end{equation}
Here
$W_\ib \subseteq \set{1, \dots , N}$ model discrete energy windows,
$(s_i)_{i=1}^N$ is the discretized
beam spectrum intensity, and $\Delta E \coloneqq (\SI{120}{\keV})/N$.
Summarizing the above we define the following forward operator.

\begin{definition}[Multi-spectral \label{def:ms} X-ray measurement operator]
The measurement operator $\aop$
with respect to the energy  windows $W_1, \dots,W_\nb$ is given by
    \begin{align} \nonumber
         \aop \colon  (L^2(D_R))^\nb &\to L^2(\sph^1 \times \sph^1)^\nb
        \\ \label{eq:ms}
        f &\mapsto
        \kl{
        \sum_{i \in W_\ib}
        s_i \exp \kl{ - \Xo \kl{ \sum_{\ib = 1}^{\nb} \mu_{i,\ib} f[\ib] } }
        }_{\ib=1}^\nb \,.
    \end{align}
\end{definition}
We can decompose the operator $\aop$ in the form
\begin{equation}\label{eq:bera}
\aop(f)  = \kl{ \vop_\Y \circ \eop \circ \xop \circ  \uop } (f)
\end{equation}
where
\begin{itemize}
\item $\uop \colon L^2(D_R)^\nb \to L^2(D_R)^N\colon f \mapsto ( \sum_{\ib = 1}^{\nb} \mu_{i,\ib} f[\ib])_{i=1}^{N} $
\item $\xop \colon L^2(D_R)^N \to L^2(D_R)^N  \colon (\mu_i)_{i=1}^{N}
\mapsto (\Xo \mu_i)_{i=1}^{N}$
\item $\eop \colon L^2(D_R)^N \to L^2(D_R)^N  \colon (g_i)_{i=1}^{N}
\mapsto (\exp (- g_i))_{i=1}^{N}$
    \item $\vop_\Y \colon L^2(D_R)^N \to L^2(D_R)^\nb  \colon (g_i)_{i=1}^{N}
\mapsto (\sum_{i \in W_\ib}
        s_i g_i)_{\ib=1}^\nb$.
\end{itemize}

The operators $\vop_\Y, \xop, \uop$ are  linear and bounded. To show the continuity
and differentiability  of $\aop$ we have to verify that $\eop$ is continuous and
differentiable.

\begin{proposition}[Continuity and differentiability of $\aop$]
The operator $\aop$ is continuous and Fr\'echet differentiable.
For $f,h \in (L(D_R)^2)^\nb$ we have
\begin{equation}\label{eq:derexx}
    \aop'(f)(h) = \kl{  \vop_\Y \circ \eop'(\xop \uop f) \circ \xop \circ \uop } (h)
\end{equation}
with
\begin{equation}\label{eq:derex}
\eop'( g ) h  = - ( \exp(-g_i) h_i)_{i=1}^N \,.
\end{equation}
\end{proposition}

\begin{proof}
One only has to verify that $f \mapsto \exp(-f)$ is continuous  and Fr\'echet differentiable
on  $L^2(D_R)$ with derivative given by $\eop'( g ) h =  \exp(-g) h$. For that purpose,
let $\norm{h}_{2} \to 0$ which in particular implies its point wise convergence. Therefore
\begin{align*}
& \frac{\norm{\exp( -g -h) -  \exp(-g) + \exp( -g) h }_2}{\norm{h}_2}
\\&=
 \frac{\norm{\exp( -g)  \exp(-h) -  \exp(-g) + \exp( -g) h }_2}{\norm{h}_2} \\&
\leq
 \frac{\norm{\exp(-h) -  1 +  h }_2}{\norm{h}_2}
\\&
\leq    \frac{ \norm{\mathcal O (h^2) }_2 }{\norm{h}_2}
\leq    \frac{ \mathcal O \kl{ \norm{h}_2^2}  }{\norm{h}_2}
= \mathcal O ( \norm{h}_2 )    \,.
\end{align*}
This  shows \eqref{eq:derex}, and \eqref{eq:derexx}
follows by the chain rule.
\end{proof}

In the context of the BCD method, the fractional
density maps  $f[\ib] $ play the roles of the blocks $x[\ib]$.
The form \eqref{eq:bera} of the forward operator $\aop$ {\rot has some similarity
with the form that  we used in the theoretical analysis of the BCD method, in the sense
that the  infinite dimensional  smoothing operator is applied to several channels
of a function. However,  so far we have not been able to perform an analysis
accounting for the non-linearity.}
Additionally, we apply a preconditioning technique as outlined in the following subsection. Extending the convergence analysis of BCD
such that it applies to multi-spectral CT is subject of future research.

\subsection{Logarithmic scaling and preconditioning}
\label{sec:scaling}

The energy dependence of the mass-attenuation coefficient of   different
materials can be quite similar. In order to enhance the  dependence on
the different   materials we propose a  logarithmic scaling and preconditioning
technique (different from \cite{barber2016algorithm}).
For simplicity we consider only the case $\nb = 2$, the general case can be
treated in a similar manner.

The proposed preconditioned logarithmic data take the  form
\begin{multline}\label{precon}
    \hop(f) \coloneqq
    \begin{pmatrix}
    \hop_1(f)
    \\
    \hop_2(f)
    \end{pmatrix}=
    \begin{pmatrix}
    c_{1,1}  & c_{1,2}
    \\
    c_{2,1}  & c_{2,2}
    \end{pmatrix}
    \begin{pmatrix}
    \log ( \aop_1(f) )
    \\
    \log( \aop_2(f) )
    \end{pmatrix}
   \\ =
       \begin{pmatrix}
    c_{1,1} \log( \aop_1(f))+ c_{1,2} \log(\aop_2(f))
    \\
    c_{2,1} \log(\aop_1(f))+ c_{2,2} \log(\aop_2(f))
    \end{pmatrix} \,,
    \end{multline}
 where $f = (f[1],f[2])$ are the unknowns  and $c_{1,1}$, $c_{1,2}$, $c_{2,1}$, $c_{2,2}$ are parameters.  Moreover, recall  that   $\aop_1(f)$ and $\aop_2(f)$ are the X-ray intensities
 defined by  \eqref{eq:ms} corresponding to
 $W_1, W_2 \subseteq \set{1, \dots , N}$ modeling  the discrete energy windows.
The preconditioned inverse problem consists in solving the system
\begin{align}\label{eq:invms1}
v_1 &= \hop_1(f[1],f[2])+z_1
\\ \label{eq:invms2}
v_2 &= \hop_2(f[1],f[2]) +z_2\,,
\end{align}
where $v_1,v_2$ are data perturbed  by noise
$(z_1,z_2)$.

In order to solve the equations in \eqref{eq:invms1},
\eqref{eq:invms2} with  the BCD  method  we define the residual functionals
\begin{align*}
\resfun_1(f[1],f[2])&:= \frac{1}{2}\norm{ \hop_1(f[1],f[2])-v_1}^2\,,\\
\resfun_2(f[1],f[2])&:= \frac{1}{2}\norm{ \hop_2(f[1],f[2])-v_2}^2\,.
\end{align*}
Application of the BCD  method requires
the adjoint gradient of $\aop_1$ and $\aop_2$, that we
compute next.

\begin{proposition}[Derivative of the preconditioned residuals]\label{th:derivative}
Let $f, h  \in L(D_R)^2$. The directional  derivatives of  $\resfun_1$ and $\resfun_2$
at  $f$ in direction  $h$  are given by
    \begin{multline}
        \resfun_\ib'(f)(h)
      \\
       =
       - \sum_{m=1}^{2}
       \sum_{k=1}^2
       \sum_{i \in W_\ib}
       \Big\langle \hop_\ib(f)-v_\ib,
        \frac{c_{\ib,k}}{\aop_k(f)}
         s_i
        \exp(-\Xo (\uop f)_i)\Xo(\mu_{i, m} h_m)
        \Big\rangle_{L^2}
        \,.
    \end{multline}
\end{proposition}

\begin{proof}
This follows from the chain rule.
\end{proof}

From Proposition~\ref{th:derivative} we conclude that the partial gradients of  the
residual functionals  $\resfun_\ib$  are given by
\begin{multline}
\partial_{m}
\resfun_\ib(f) =  - \sum_{i \in W_\ib} \mu_{i,m}\Xo^{*}\left[s_i \exp(-\Xo(\uop f)_i)(\hop_\ib(f)-v_\ib)\frac{c_{\ib,1}}{\aop_1(f)}\right]
    \\
   - \sum_{i \in W_\ib}  \mu_{i,m} \Xo^{*}\left[s_i \exp(-\Xo(\uop f)_i)(\hop_\ib(f)-v_\ib)\frac{c_{\ib,2}}{\aop_2(f)}\right] \,.
\end{multline}
These expressions will be used for the implementations of the  BCD as well as the Landweber  method applied to the  preconditioned system  \eqref{eq:invms1}.

\begin{figure}[htb!]
\centering
\includegraphics[width=0.49\textwidth]{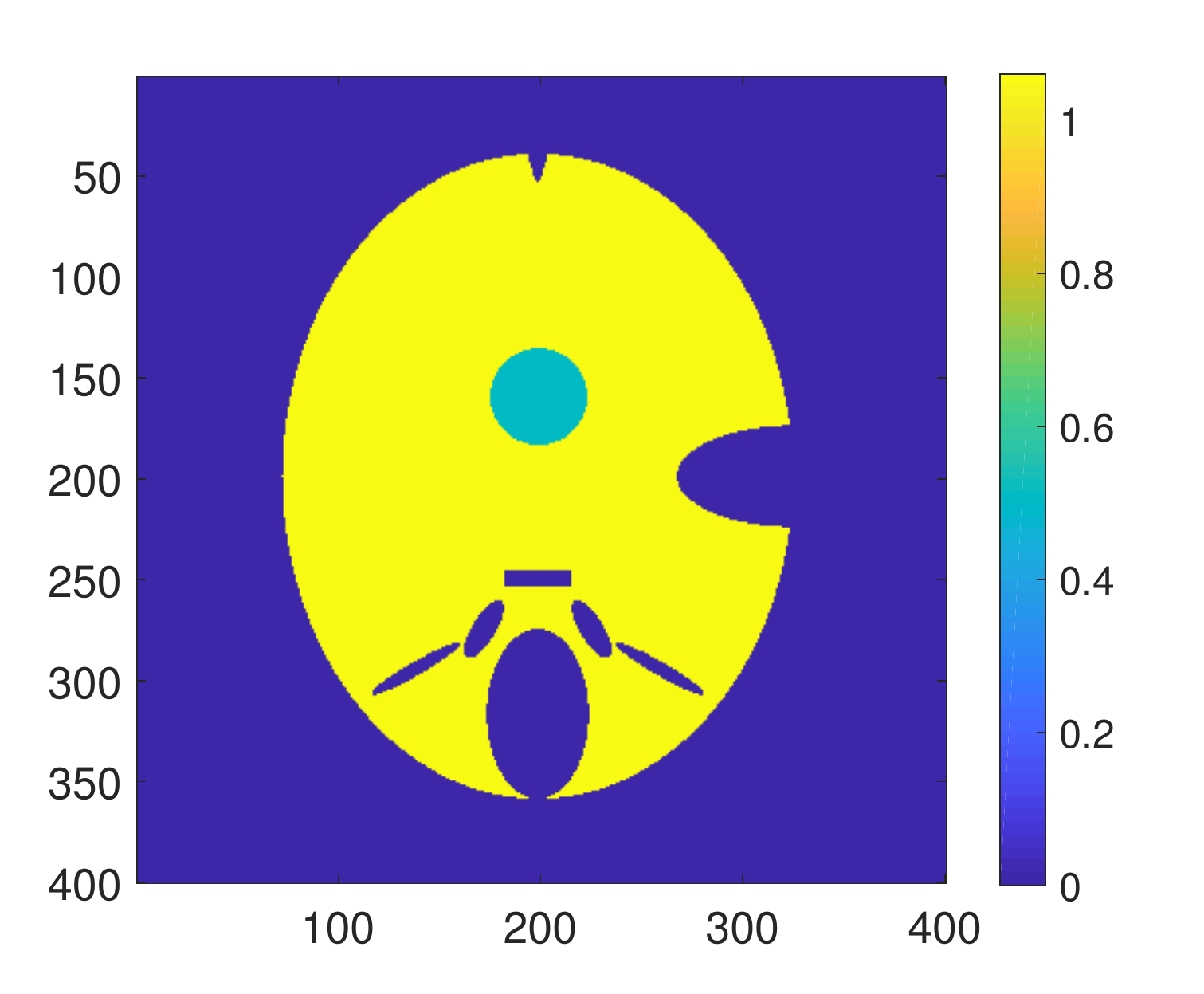}
\includegraphics[width=0.49\textwidth]{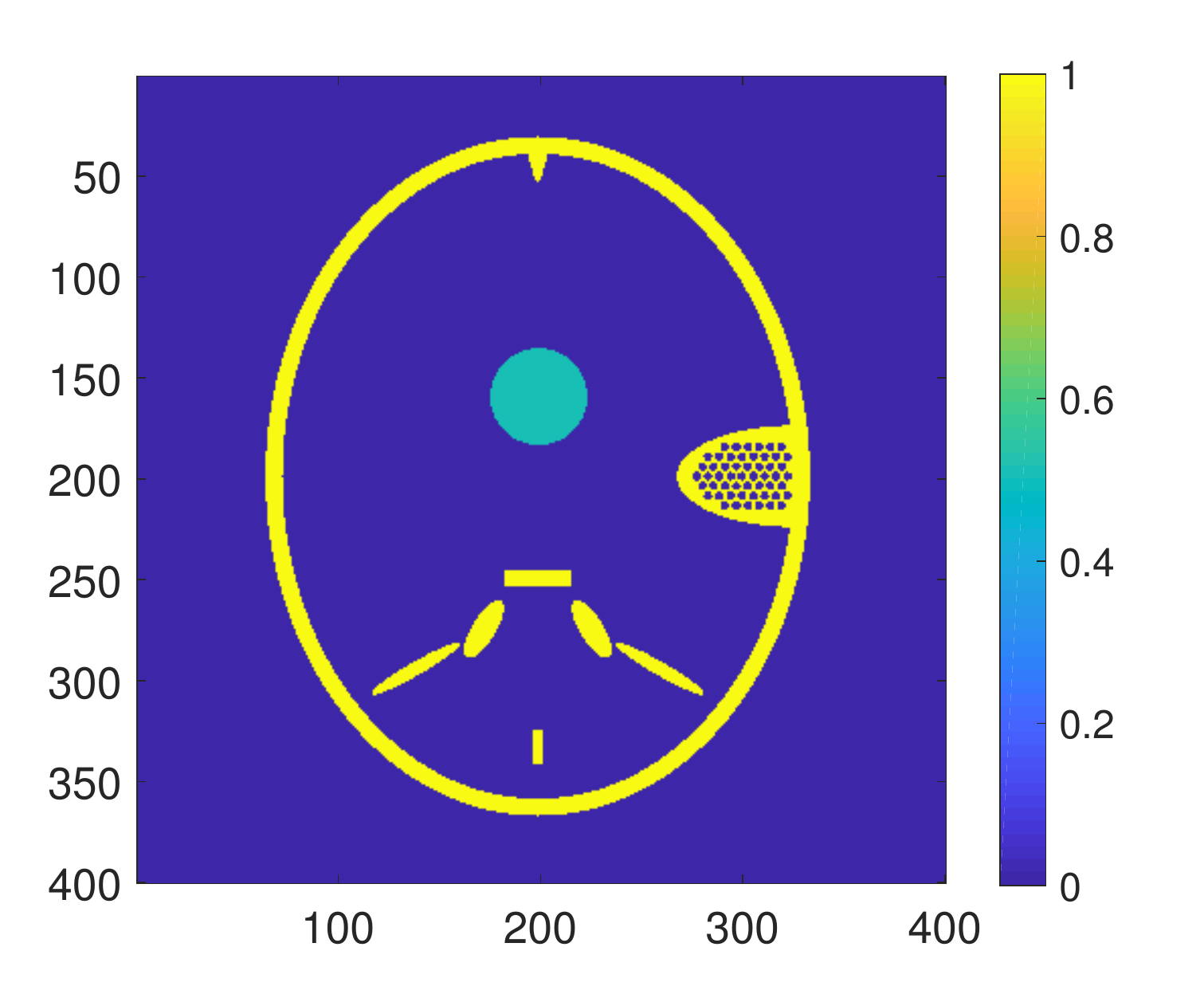}
\caption{\textbf{Phantom  $f = (f[1], f[2])$ used for the numerical results.}
Left: brain  density map $f[1]$. \label{phantom}
Right: bone density map $f[2]$. Both are derived from the FORBILD head phantom, where
a uniformly absorbing disc of value $1/2$ has been added to both channels.}
\end{figure}

\begin{figure}[htb!]
\centering
\includegraphics[width=0.49\textwidth]{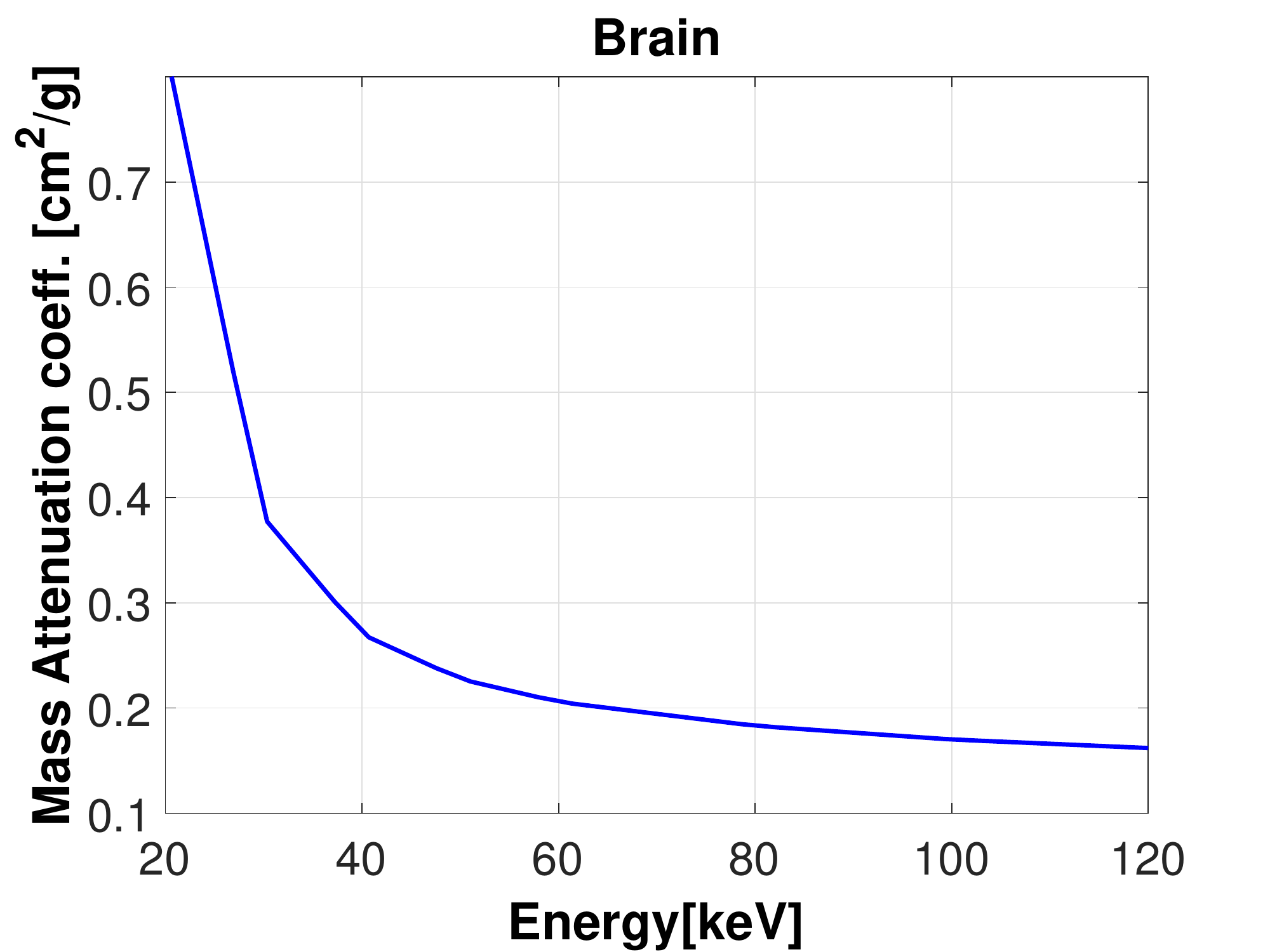}
\includegraphics[width=0.49\textwidth]{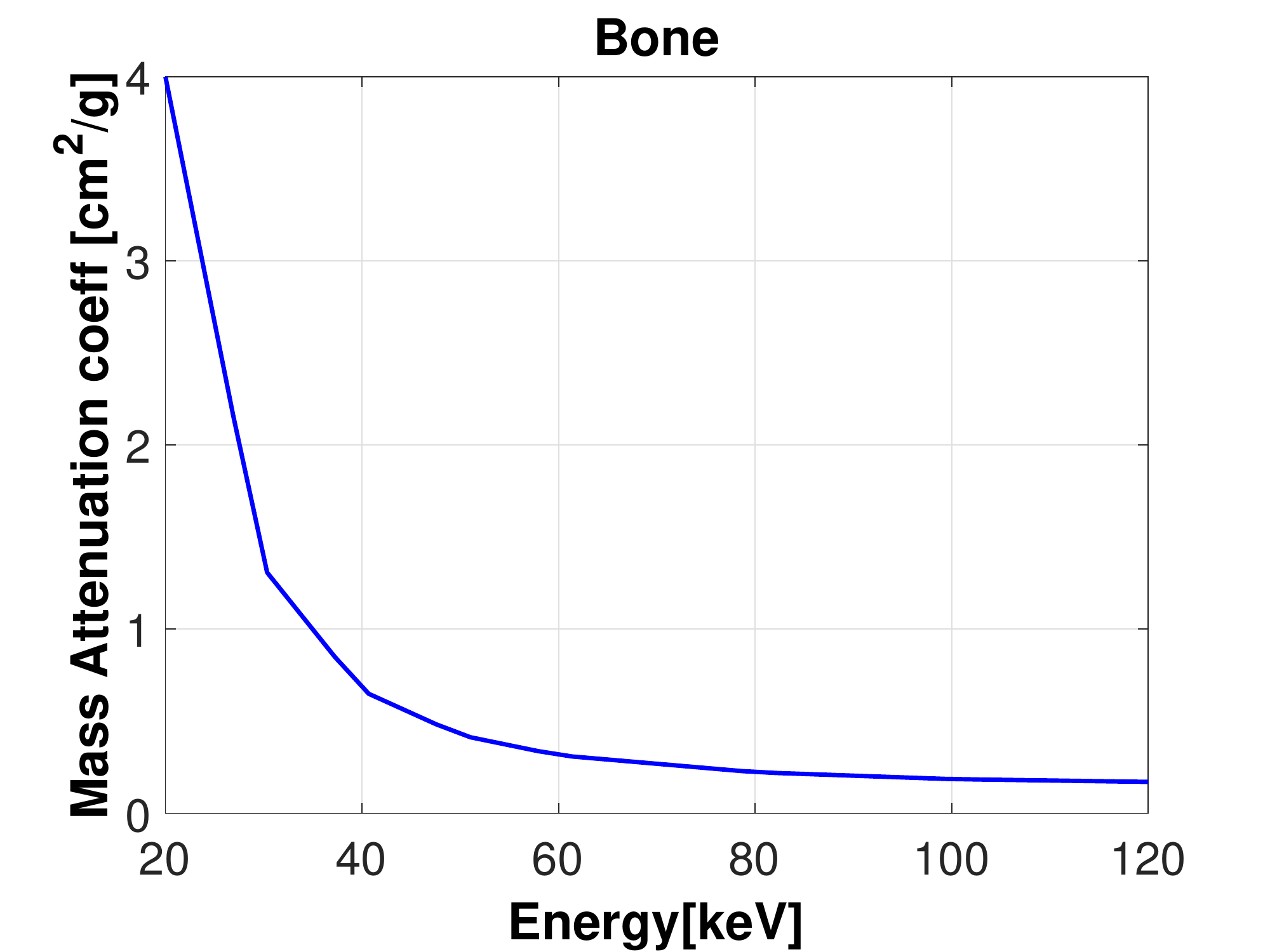}
\caption{\textbf{Attenuation coefficients  of brain and bone   taken from NIST tables \cite{hubell04}}.\label{massatt}
Left: Attenuation spectrum for brain.
Right: Attenuation spectrum  for bone.}
\end{figure}

\begin{figure}[htb!]
\centering
\includegraphics[width=0.49\textwidth]{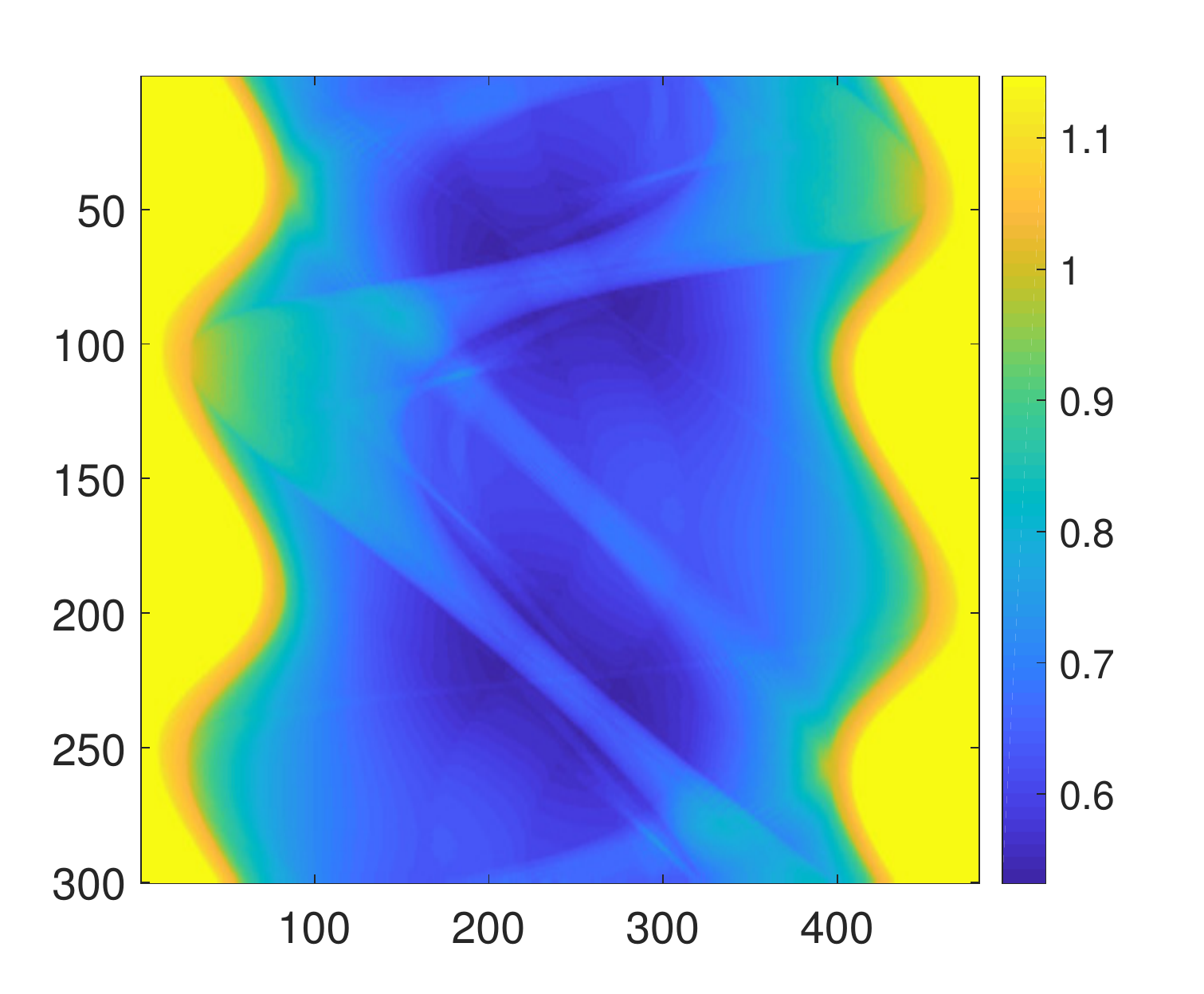}
\includegraphics[width=0.49\textwidth]{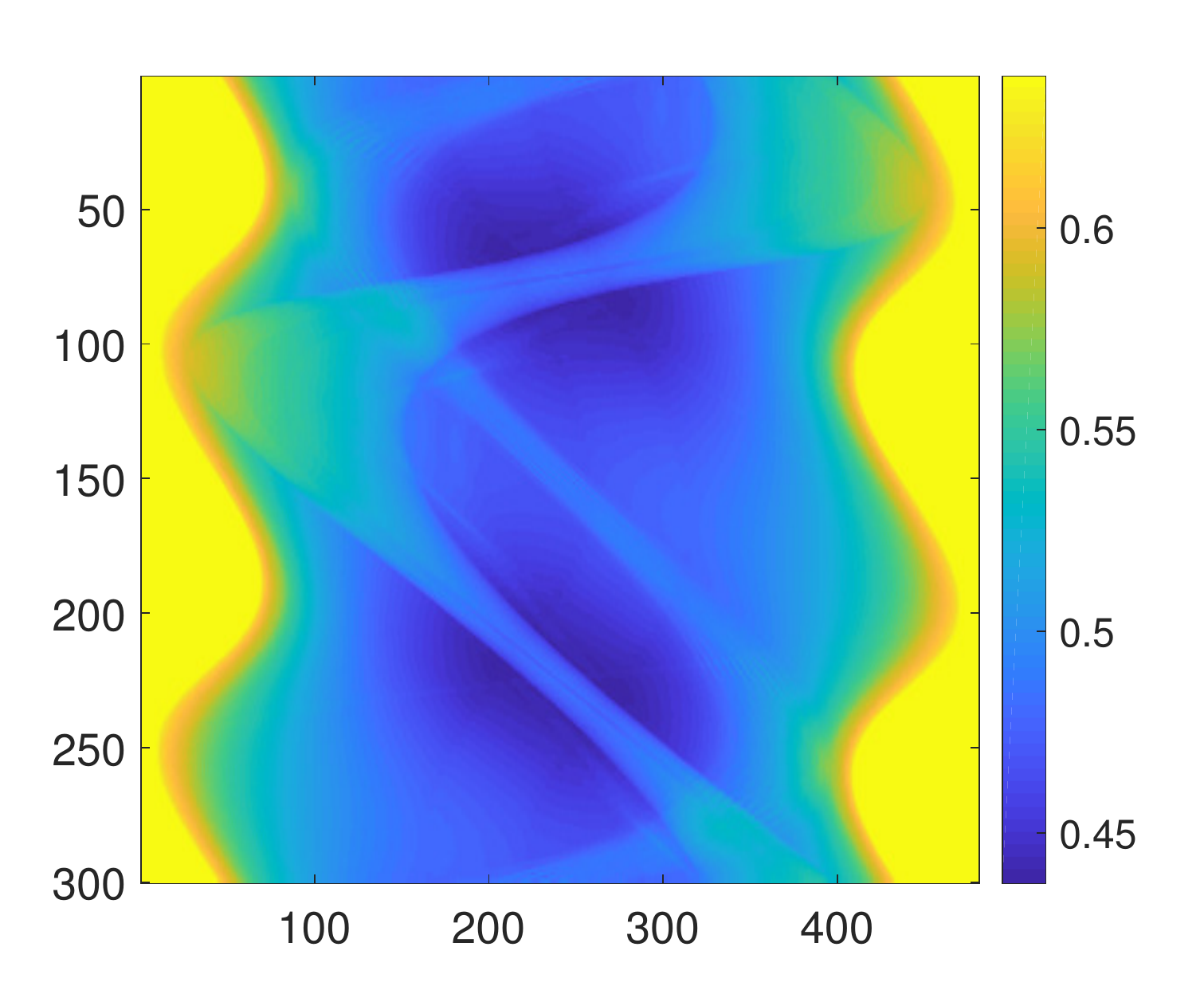}
\includegraphics[width=0.49\textwidth]{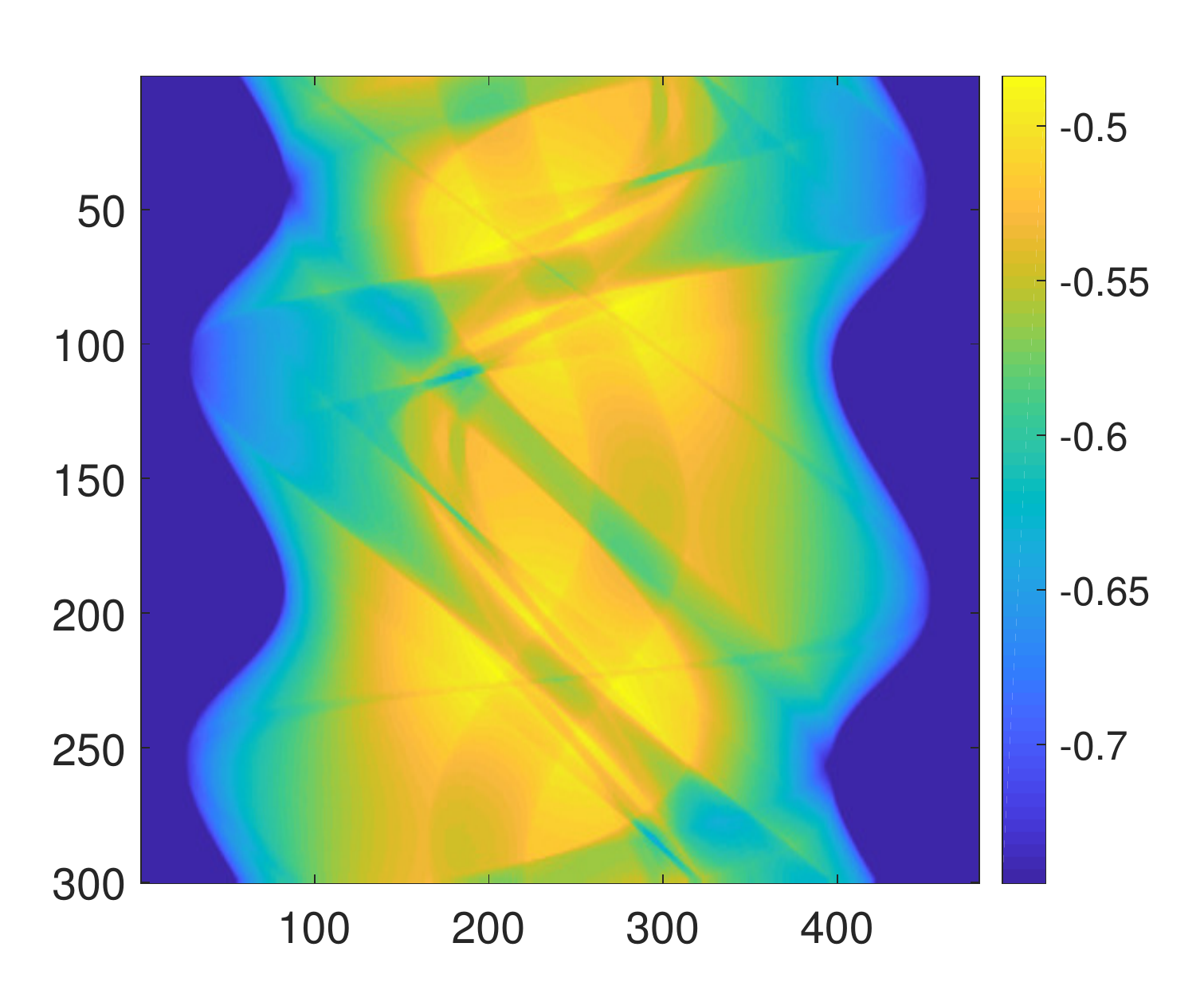}
\includegraphics[width=0.49\textwidth]{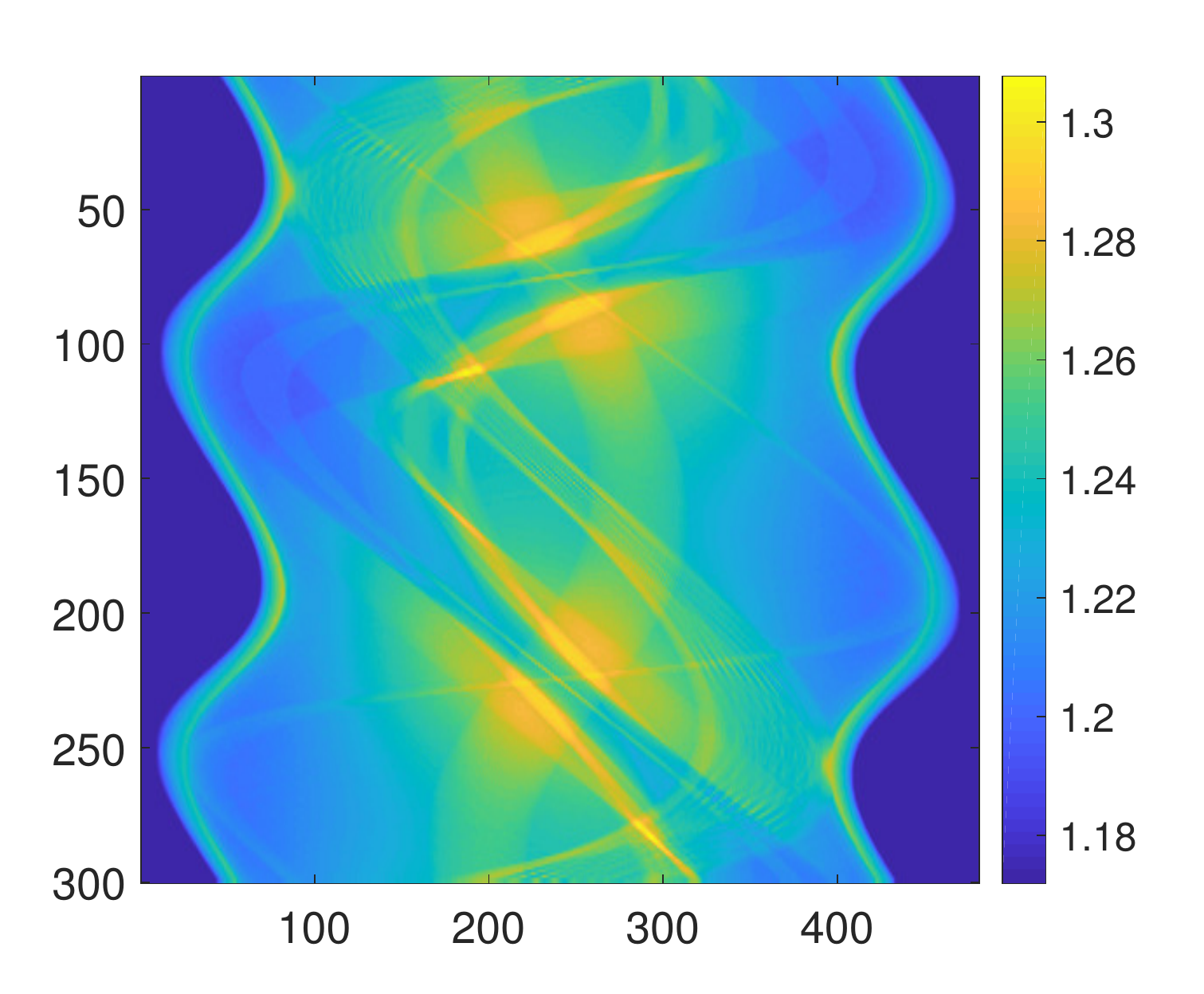}
\includegraphics[width=0.49\textwidth]{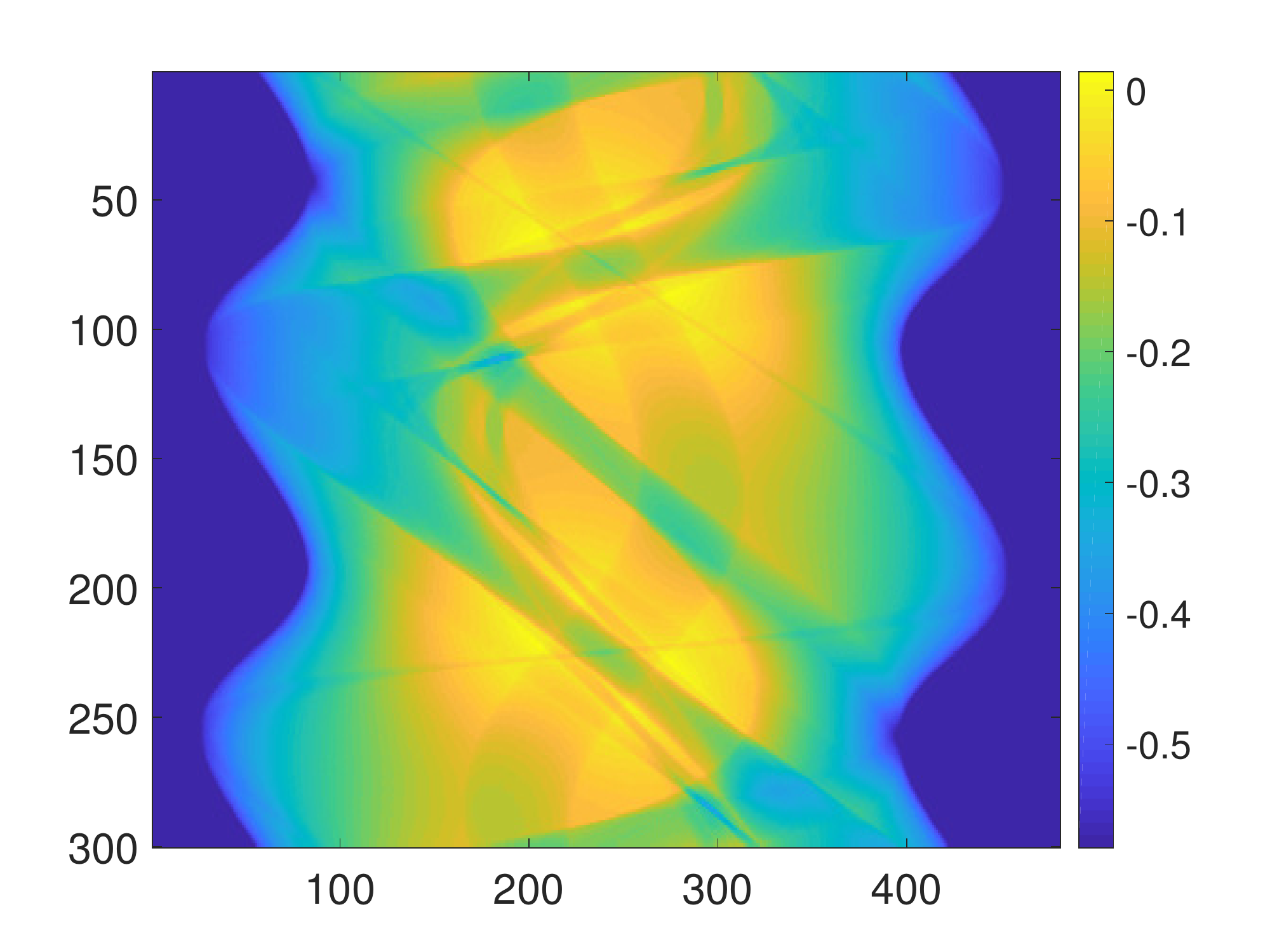}
\includegraphics[width=0.49\textwidth]{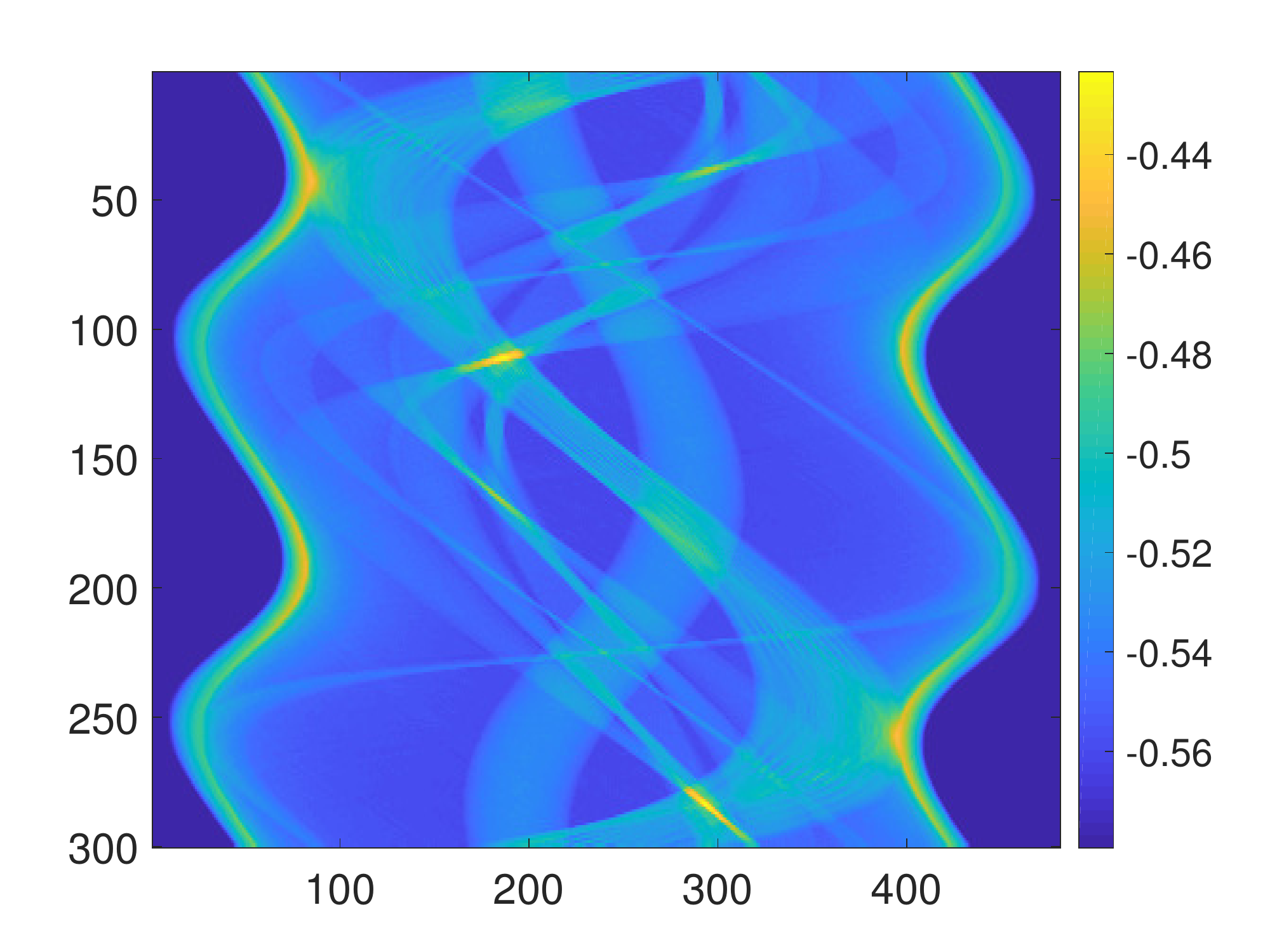}
\caption{\textbf{Simulated multi-energy X-ray data}. \label{data}
Top: Data for energy window $[\SI{20}{keV},\SI{70}{keV}]$ (left)  and $[\SI{70}{keV},\SI{120}{keV}]$ (right).
Middle: Corresponding preconditioned logarithmic data.
Third row:
Simulated data for the full energy window $[\SI{20}{keV},\SI{120}{keV}]$,
where the tissue consists only of the
brain map (left) and the bone map (right). }
\end{figure}

\subsection{Numerical implementation}

For all our experiments we used  fan beam geometry.
Each channel of the discrete phantom has
size $400 \times 400$.
We discretized $\Xo$ using 300 detector positions $\alpha_k$ equidistantly distributed on $\sph^1$.
For each detector position we compute $481$ line integrals for uniformly distributed  angles $\beta_\ell$ in the interval
$[-\pi/3,\pi/3]$. To actually compute $\Xo f(\alpha_k, \beta_\ell)$ we used the
trapezoidal rule and linear  interpolation where we discretized
the line integral using 400 equidistant sampling points in the interval $[0,2]$.
The adjoint $\Xo^* g$ is evaluated using the standard backprojection algorithm
with linear interpolation. We used $N = 30$ equidistant discrete energy positions
from $\SI{20}{keV}$ to $\SI{120}{keV}$.

For our numerical studies we apply one-step inversion in  multi-spectral CT tomography to
reconstruct a head phantom composed of two different material map derived from FORBID head.
The phantom is shown in Figure~\ref{phantom} and consists of the pair
$f = (f[1], f[2])$, where
$f[1]$ corresponds to the  fractional density of the brain and
$f[2]$  to the fractional density  of the bone material. We slightly modified the FORBID head phantom
by inserting a disk with value $1/2$ in both components  to demonstrate
that the method can actually  reconstruct  mixed material distributions.
The mass attenuation coefficients of the material maps (bone and brain)
are taken from NIST tables \cite{hubell04} and are  shown in Figure~\ref{massatt}.

Figure \ref{data} shows the   data   used for image reconstruction.
In the first row original data $\aop(f)  = (\aop_1(f)), \aop_2(f)$ according to Definition~\ref{def:ms} are plotted,
where the indices $1$ and $2$ corresponds to energy windows $[\SI{20}{keV},\SI{70}{keV}]$  and $[\SI{70}{keV},\SI{120}{keV}]$, respectively.
One can observe, the data for both energy windows look  quite   similar. This is because of the  similar energy dependence  of the mass attenuation coefficients for $f[1]$ and $f[2]$; compare  Figure~\ref{massatt}. For this reason, we make use of the proposed  scaling and preconditioning  outlined in Section~\ref{sec:scaling}.
The second row shows the preconditioned data we use for the reconstruction.
For comparison purpose, the last row in Figure~\ref{data} shows the negative logarithm of the X-ray intensities for the full energy window,
with in each case containing only one of the material maps.  We have chosen the  constants $c_{1,1}=1$, $c_{1,2}=-1.35$, $c_{2,1}=-1$ and $c_{2,2}=2.3$
in such a way that each of the modified data blocks highlights
different aspects of the material maps. Note that we have selected the constants
for data of a very  different  phantom  in order to avoid inverse crime.

\begin{figure}[htb!]
\centering
\includegraphics[width=0.49\textwidth]{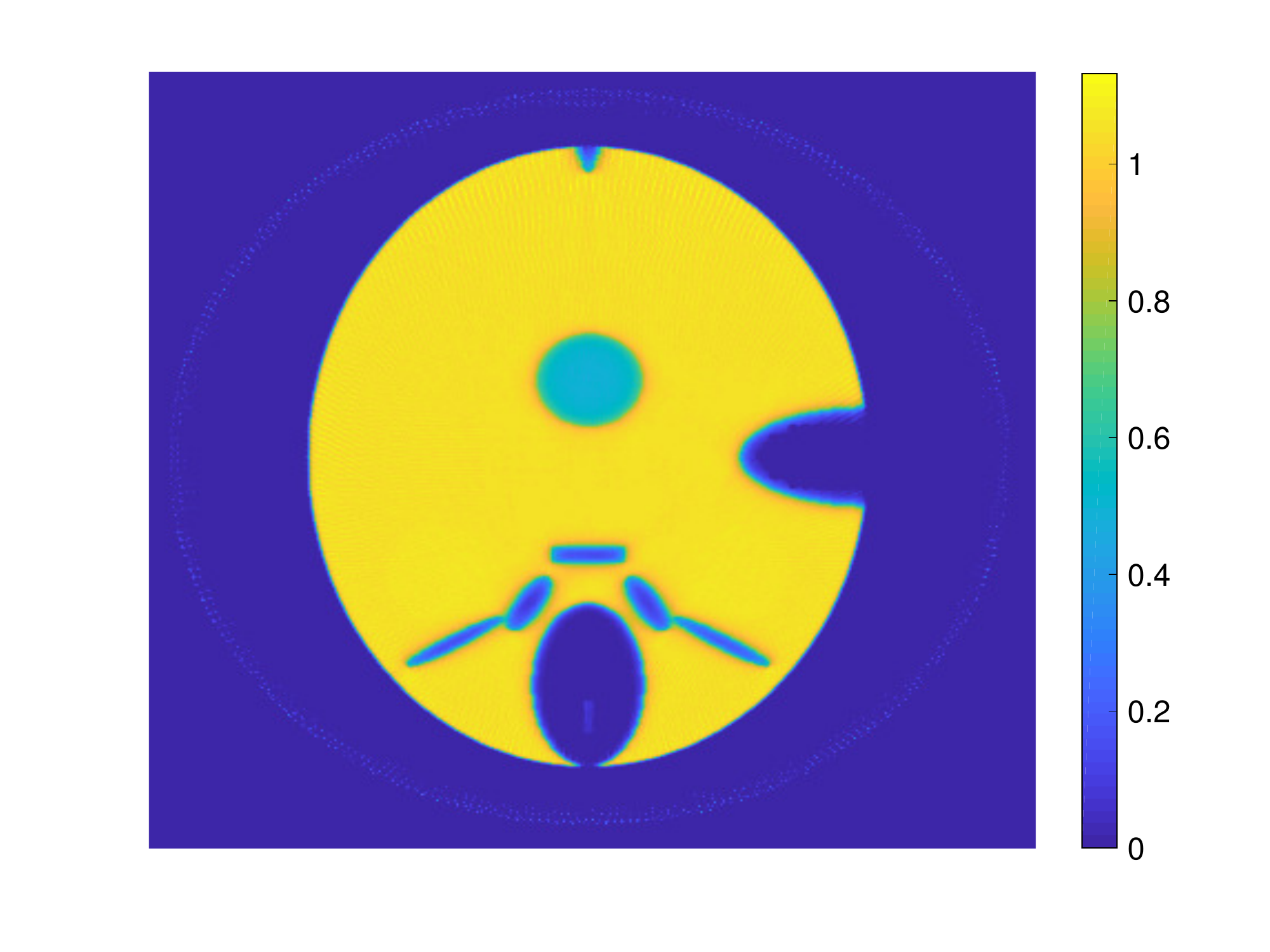}
\includegraphics[width=0.49\textwidth]{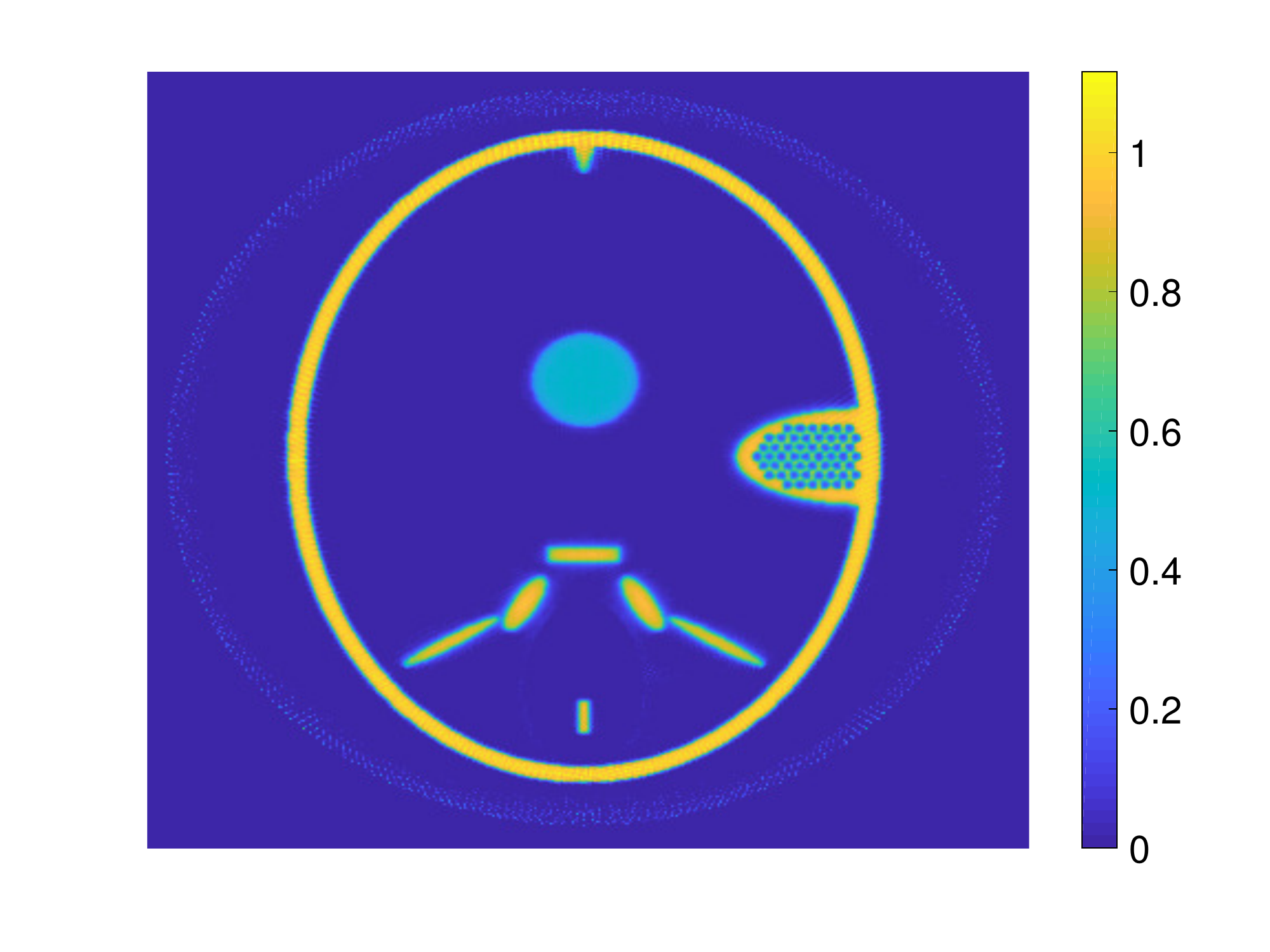}
\includegraphics[width=0.49\textwidth]{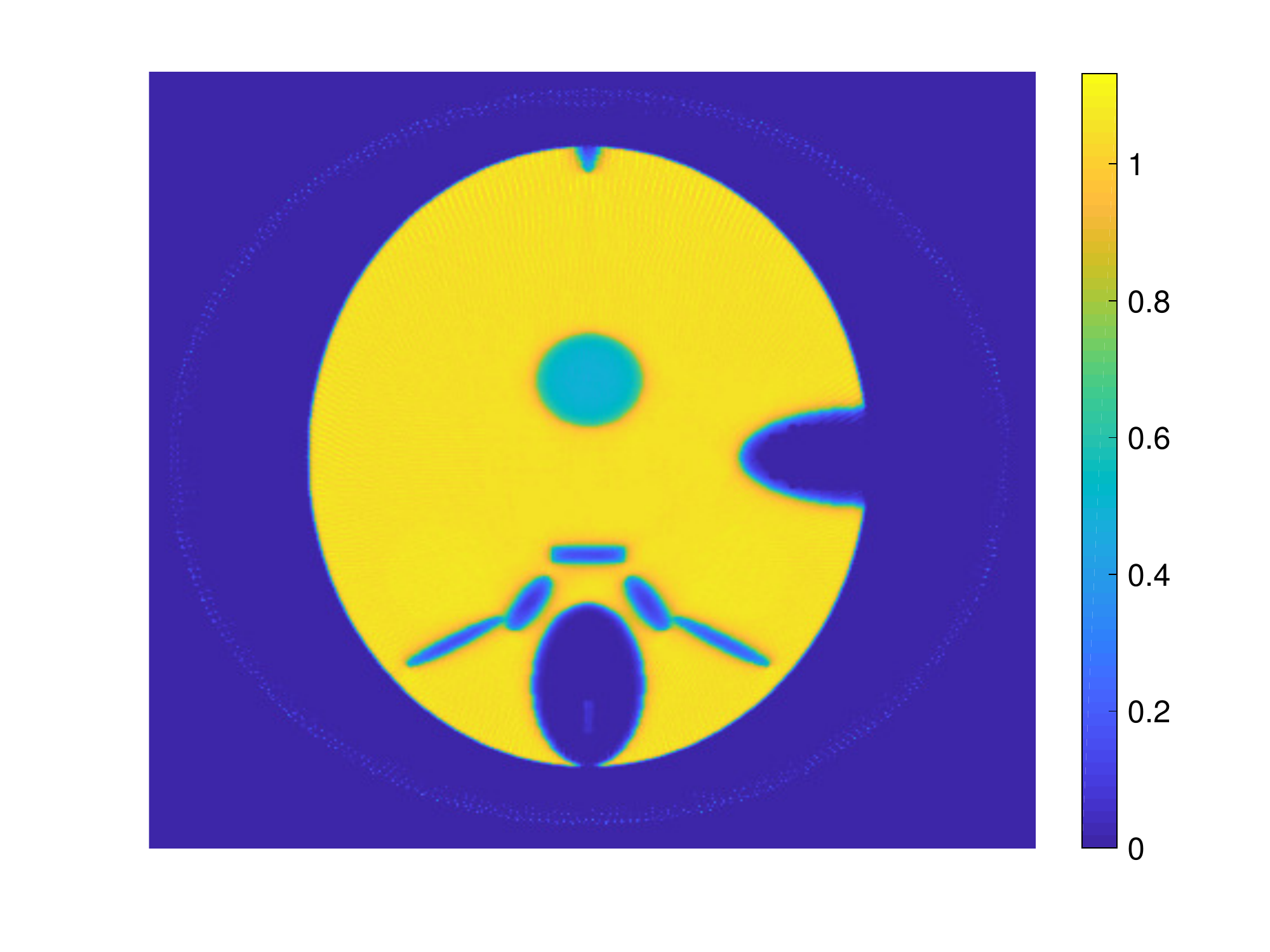}
\includegraphics[width=0.49\textwidth]{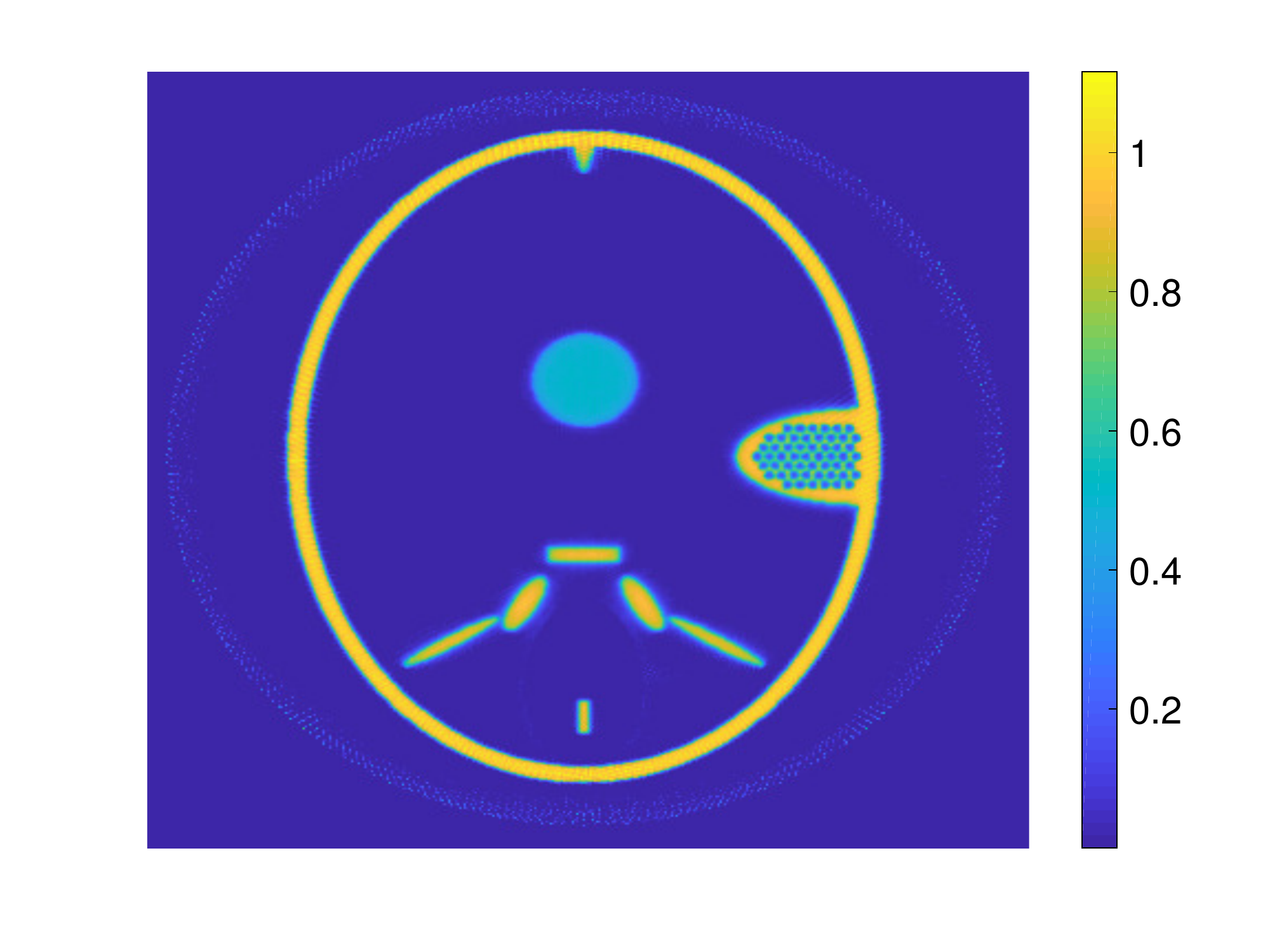}
\caption{
\textbf{Reconstruction results for simulated data.}\label{rec}
Top left: Reconstructed brain density with Landweber method.
Top right: Reconstructed bone density with Landweber method.
Bottom left: Reconstructed brain density with BCD method.
Bottom right: Reconstructed bone density with BCD method.
For the Landweber method we have used $300$ iterations,
for the BCD method $300$ cycles.}
\end{figure}

\begin{figure}[htb!]
\centering
\includegraphics[width=0.49\textwidth]{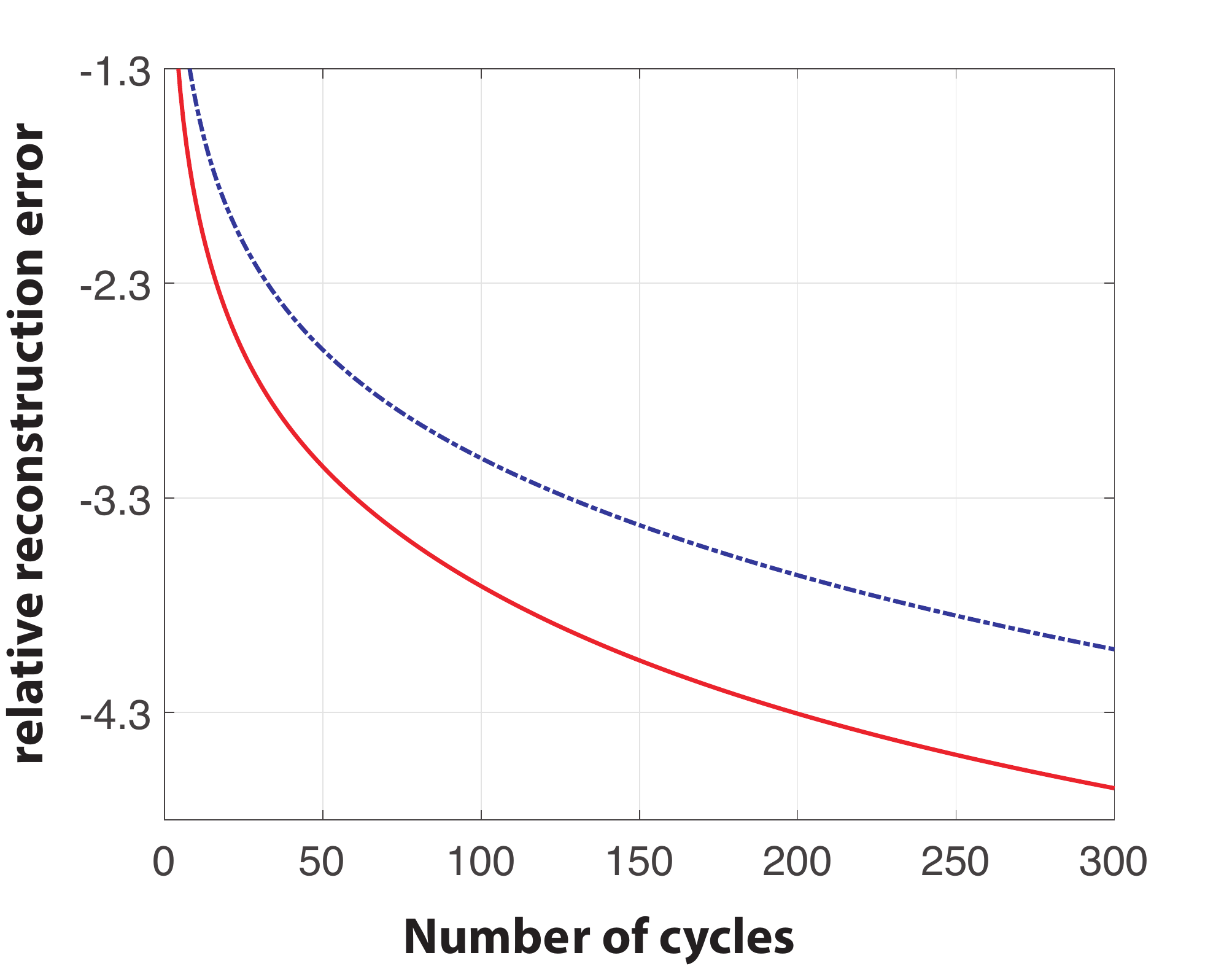} 
\includegraphics[width=0.49\textwidth]{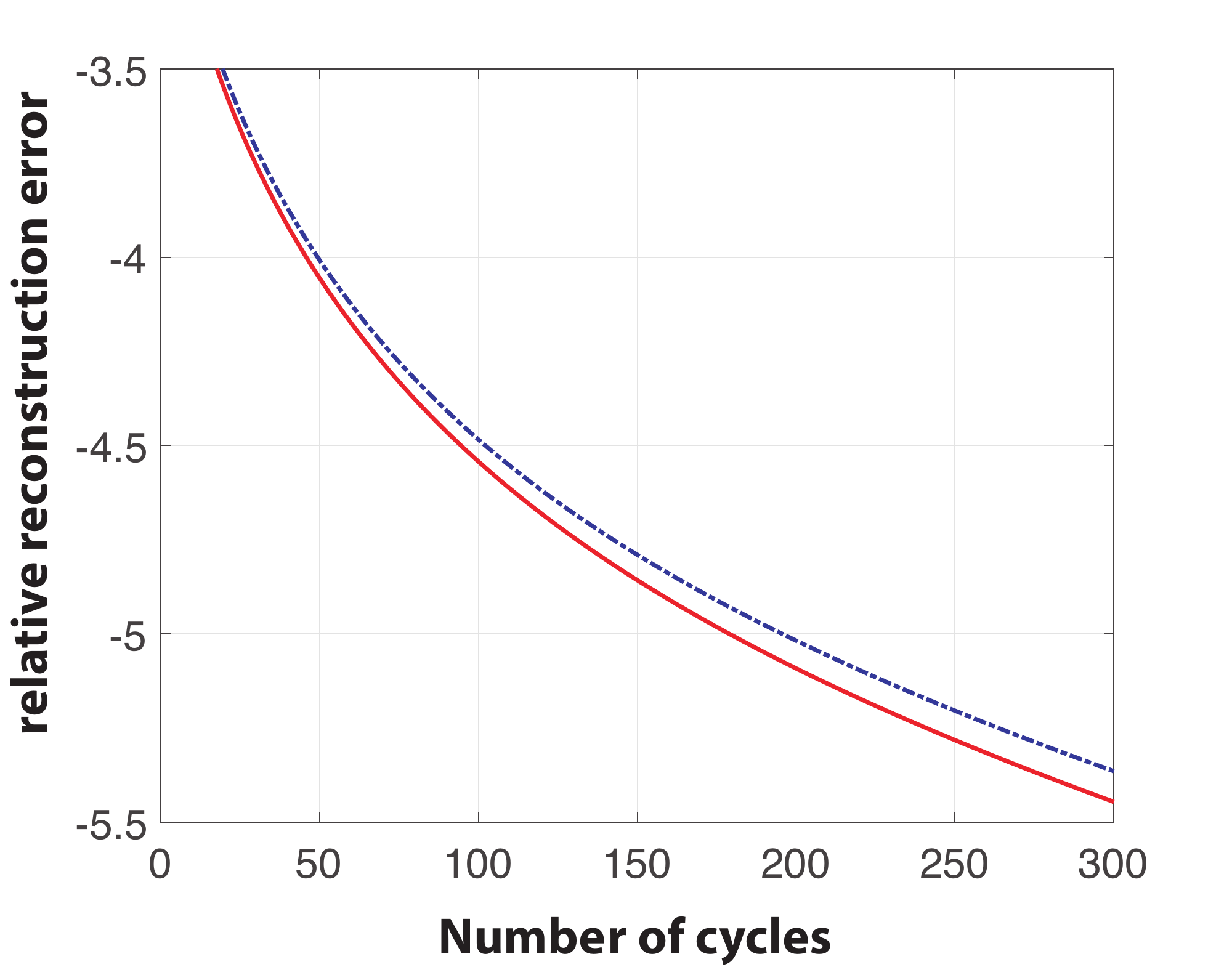}
\caption{\textbf{Relative reconstruction error for simulated data.}\label{err}
Left: Reconstructed brain density. Right:  Reconstructed  bone density.
The Landweber method is shown in dashed blue and the BCD method in solid red.}
\end{figure}

\subsection{Numerical results}

For the following results we  compare the performance of the BCD
method  with the standard gradient method as  reference method.
We use a cyclic control $\ibs(k) = (k-1) \mod \nb$ and  constant step sizes for
both methods.
Note that for the BCD as well as the Landweber method we  included a positivity
constraint.  Figure \ref{rec} shows  reconstruction results for the bone and brain material  map.
Due to the applied logarithmic scaling and preconditioning, both methods are
able to separate the materials after a reasonable number of iterations.
One observes that even the mixed part can be reconstructed as well.

Figure \ref{err} shows the relative squared reconstruction errors
\begin{equation*}
e[\ib] \coloneqq \frac{\norm{f[\ib]-f_{\rm rec}[\ib]}^2}{\norm{f_{\rm rec}[\ib]}^2}
\end{equation*}
of the bone and the brain map using the Landweber method and the BCD method. The horizontal axes show the number of iterations in the Landweber method
and the number of cycles (number of iterations divided by the  number of blocks)
in the BCD method. A cycle for the BCD method has the same numerical complexity
as one iteration for the Landweber method.
The BCD method delivers a lower relative error for the brain map, the relative error of the reconstruction for the bone map is similar for both methods.

Reconstruction results for noisy data are shown in Figure \ref{rec:noise}.
To generate the noisy data, we added Gaussian white noise with standard deviation equal
to $\SI{2}{\percent}$ of the maximal value of the exact data.
In order to maintain stability of both iterations we stopped the Landweber iteration after $116$ iterations, accordingly the BCD-method is stopped after $116$ cycles.
The relative squared reconstruction error is shown  in Figure~\ref{err:noise}.
Again, the BCD method is roughly  a  factor two faster than the Landweber
 method in recovering the brain map. For recovering the bone map, both
 methods are equally fast. We associate this  different behavior to the particular form of
 preconditioning. As can be seen  from the second line in  Figure~\ref{data},
 both  preconditioned  data pairs  contain significant parts of the
 data corresponding to the brain  whereas the bone data  is mainly contained
 in the second one.  Investigating  optimal weights  for the preconditioning  is an interesting aspect of future work.

\begin{figure}[htb!]
\centering
\includegraphics[width=0.49\textwidth]{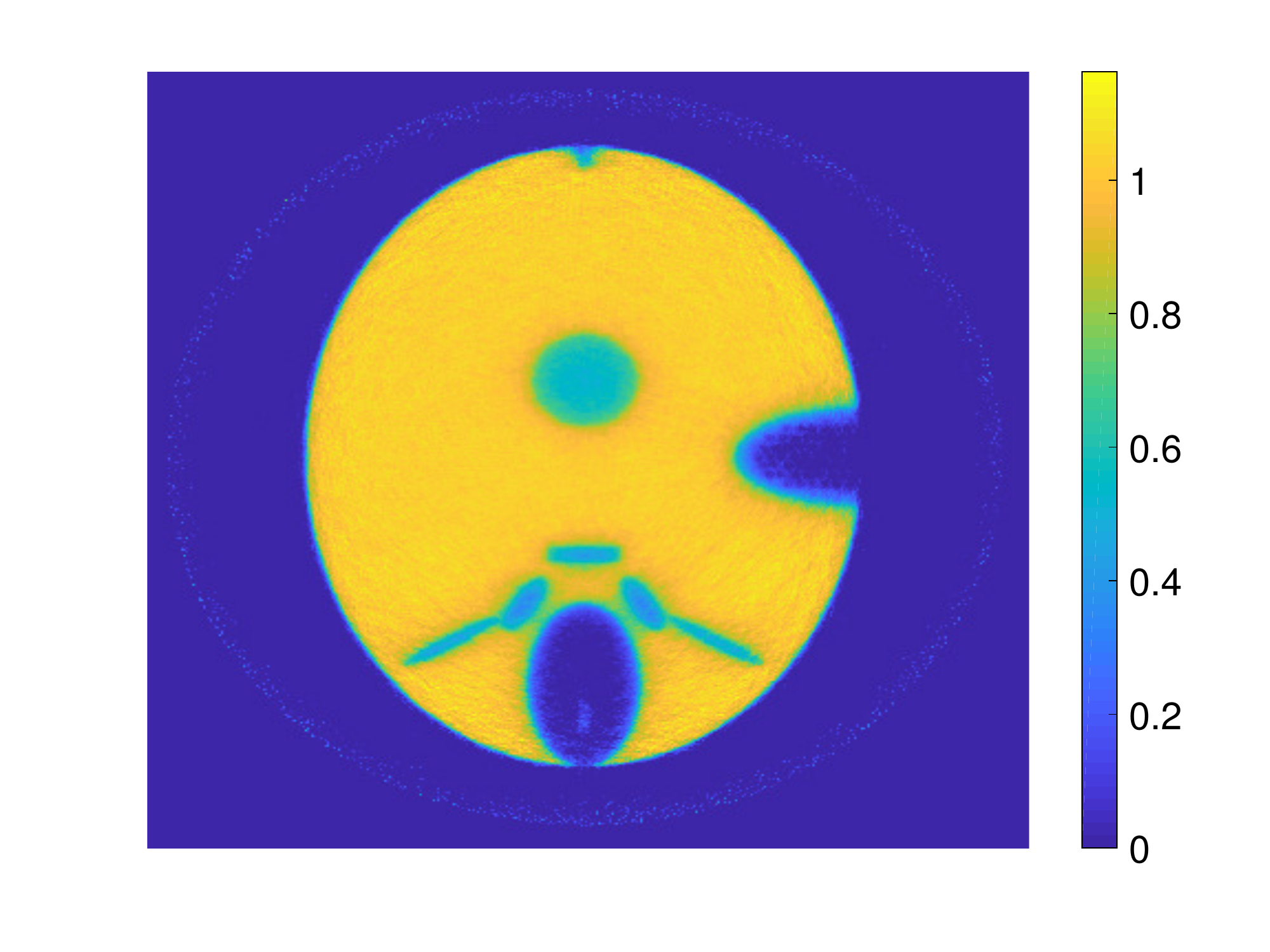}
\includegraphics[width=0.49\textwidth]{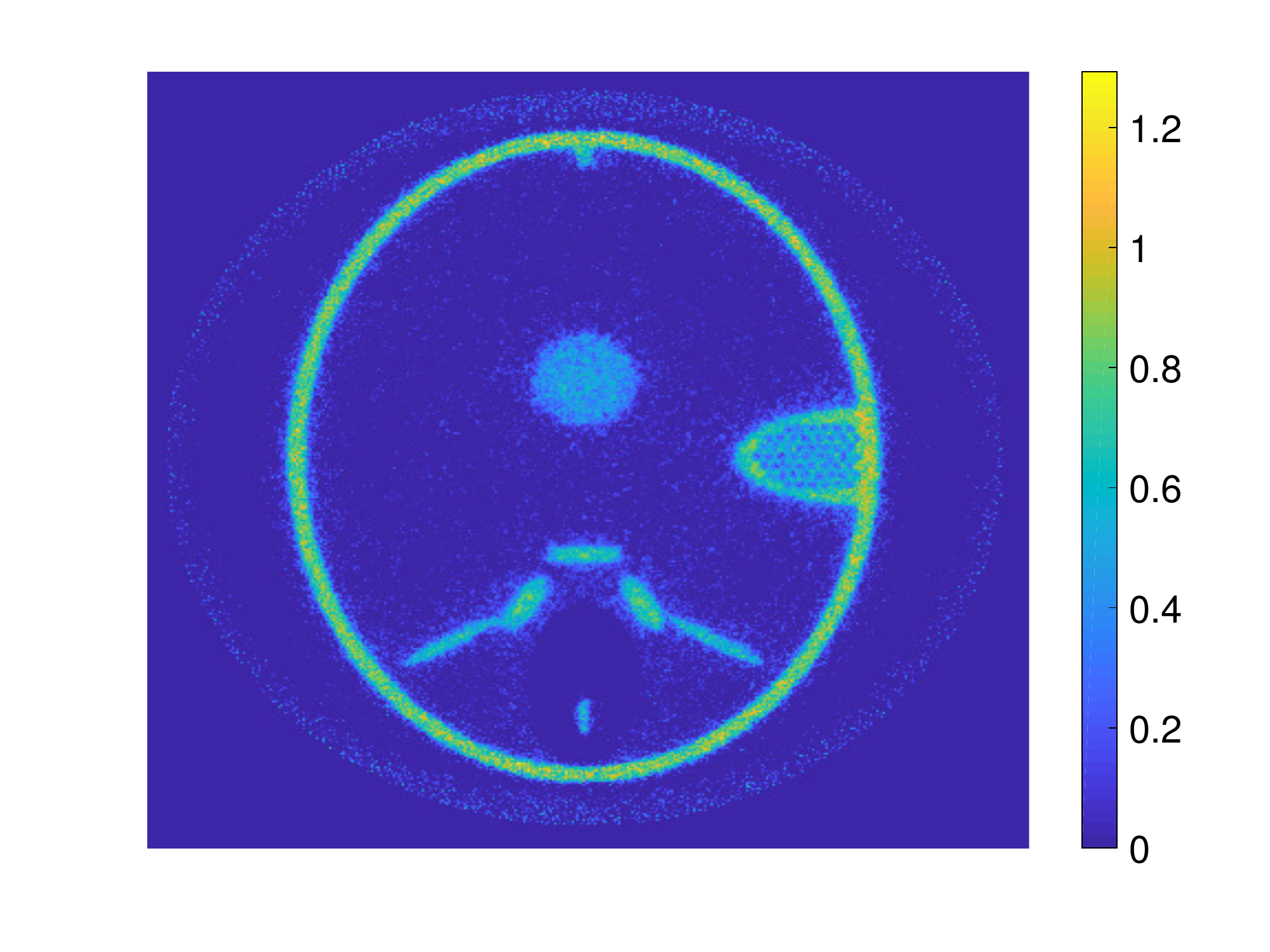}
\includegraphics[width=0.49\textwidth]{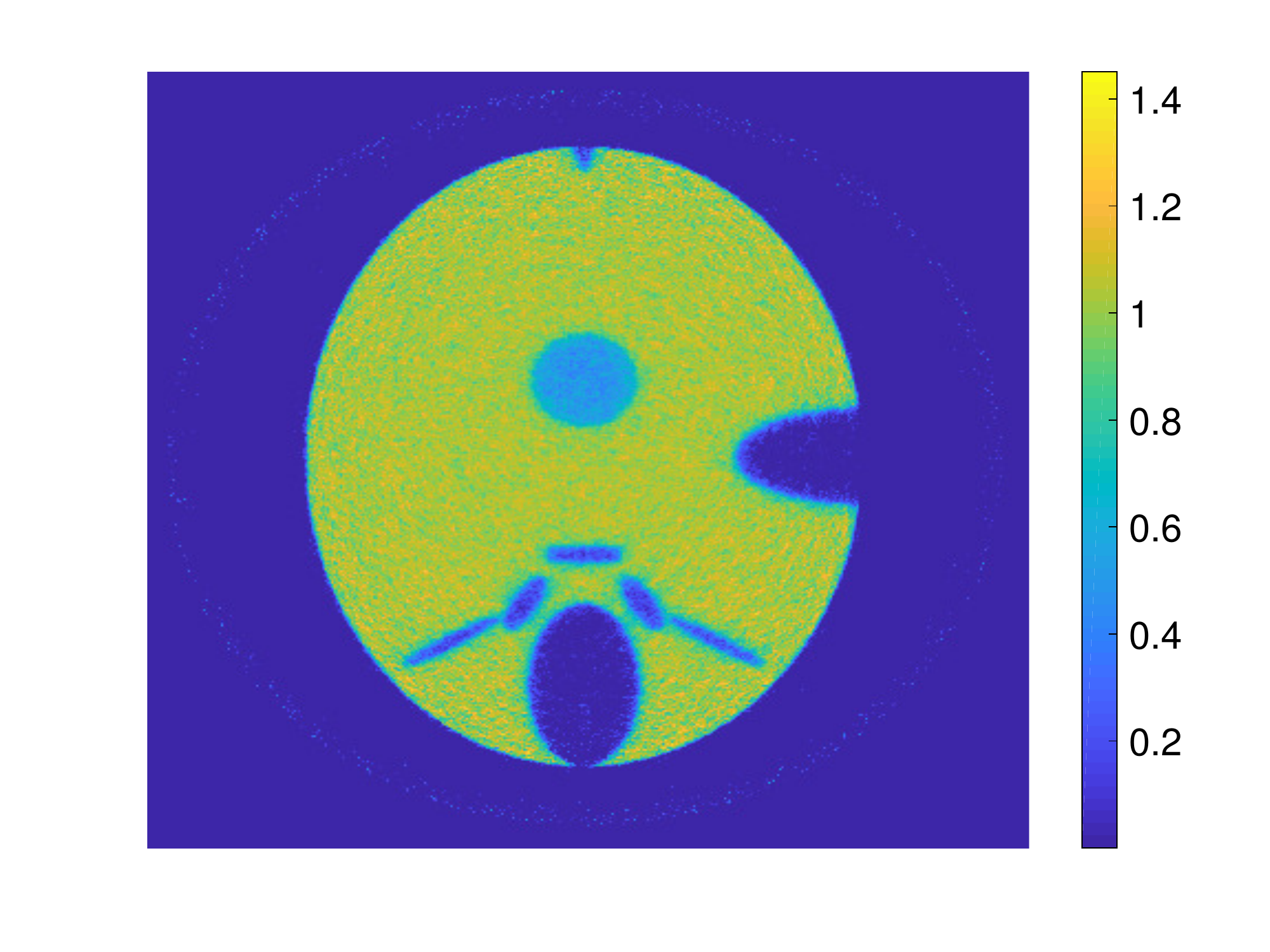}
\includegraphics[width=0.49\textwidth]{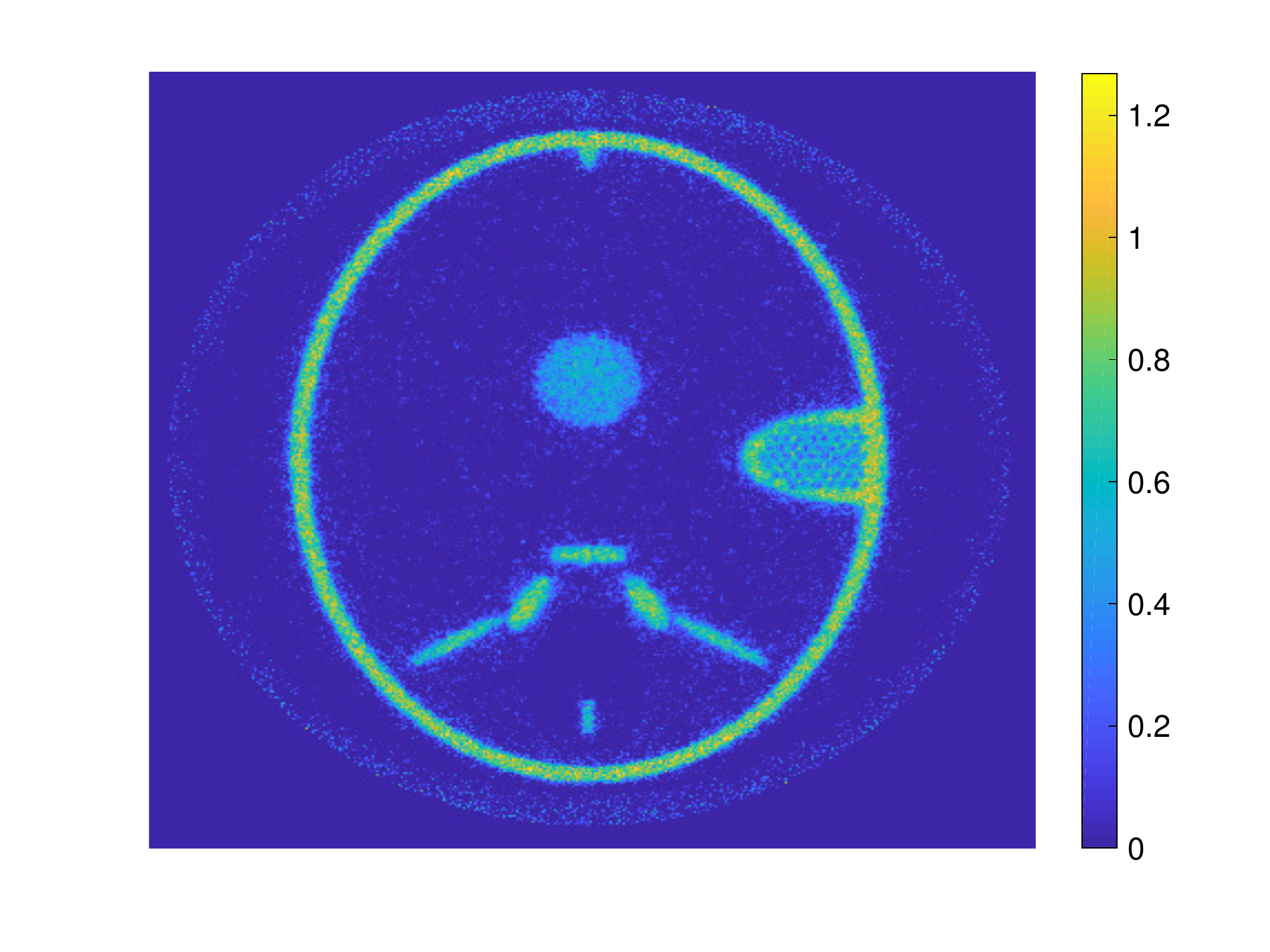}
\caption{\textbf{Reconstruction results for noisy data.}\label{rec:noise}
Top left: Reconstructed brain density with Landweber method.
Top right: Reconstructed bone density with Landweber method.
Bottom left: Reconstructed brain density with BCD method.
Bottom right: Reconstructed bone density with BCD method.
The Landweber method we have used $116$ iterations and for the BCD method $116$ cycles.}
\end{figure}

\begin{figure}[htb!]
\centering
\includegraphics[width=0.49\textwidth]{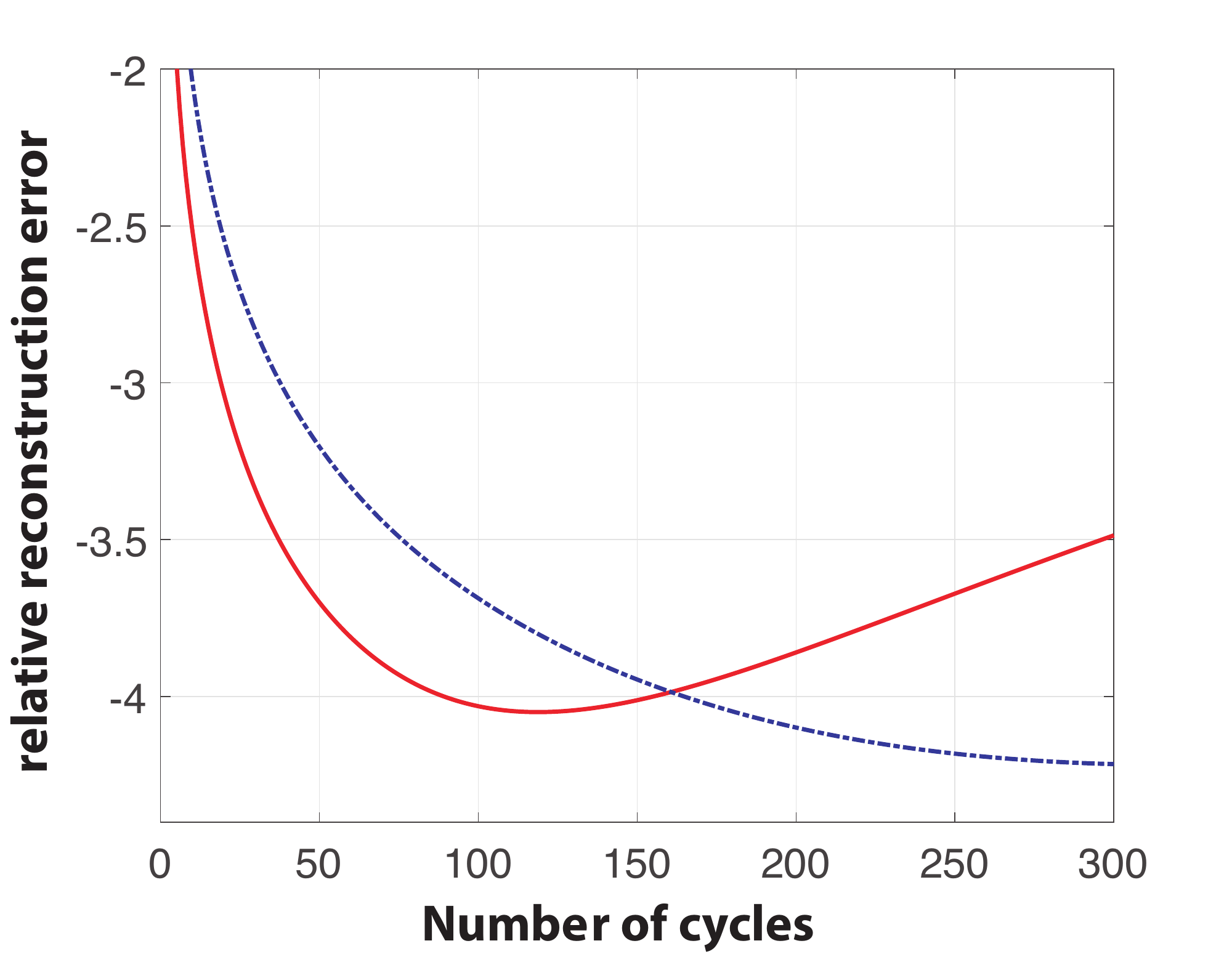}
\includegraphics[width=0.49\textwidth]{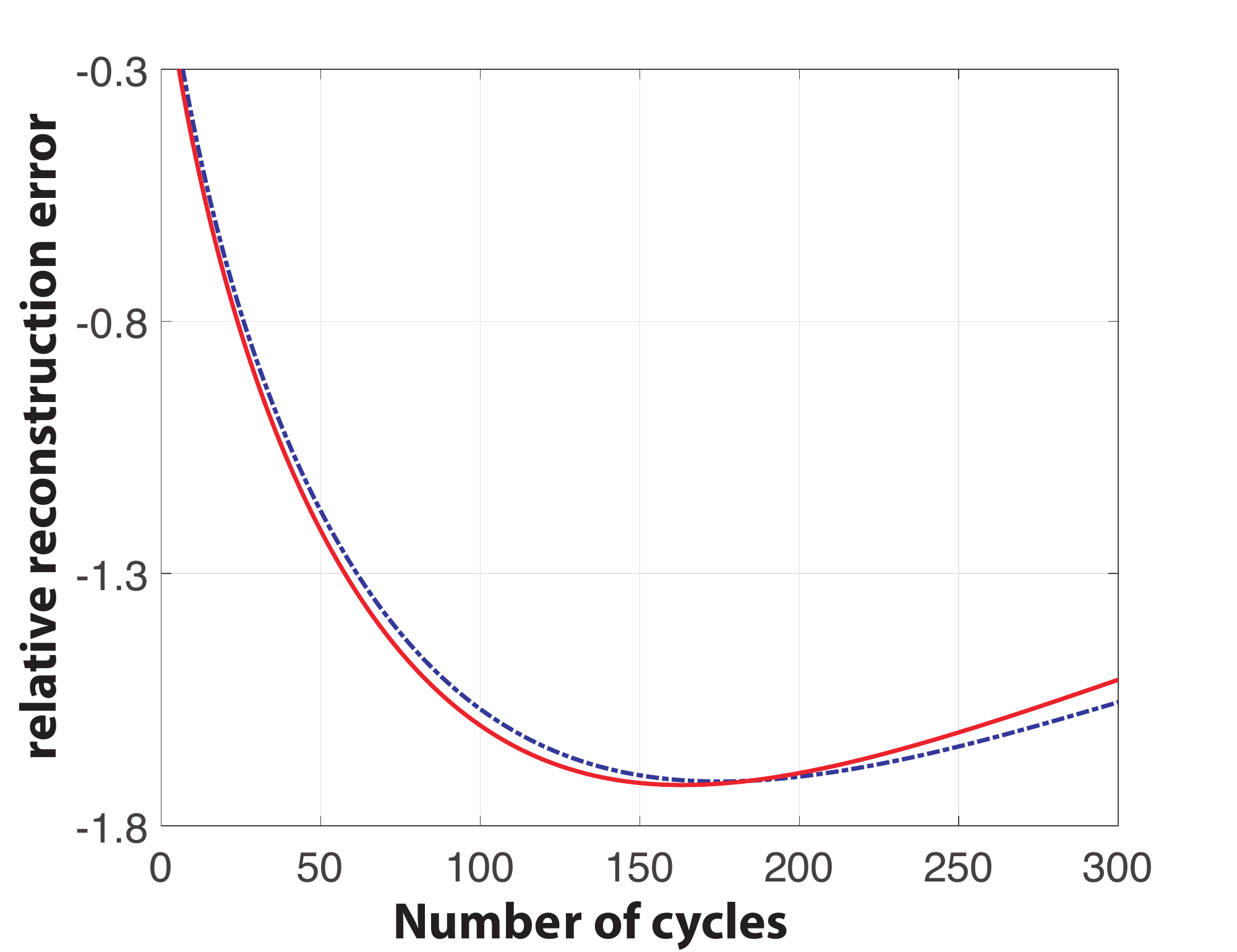}
\caption{\textbf{Relative reconstruction error for noisy data.}\label{err:noise}
Left: Reconstructed brain density. Right:  Reconstructed  bone density.
The Landweber method is shown in blue and the BCD method in red. We observed the typical
semi-convergence behaviour and therefore stopped the iterations at 116 cycles.}
\end{figure}

\section{Conclusion}
\label{sec:conclusion}

In this paper we analyzed the BCD (block coordinate descent)
method  for linear inverse problems.
For a particular tensor product form we have shown that
the  BCD  method combined with an appropriate loping and
stopping  strategy is a convergent regularization method for ill-posed
inverse problems.   The analysis in the present paper applies to operators  having the
 tensor product  form  $\Vo \otimes \Ko (x) =
 \Vo  ( \Ko (x[1]), \dots,  {\rot \Ko (x[\nb]) } )$,  where $\Vo  \in \R^{\nd \times \nb}$
 and $\Ko \colon \X \to \Y$ is linear. We presented two examples for numerically solving ill-posed problems with the BCD method.
The first one is concerns a system of linear integral equations that is covered by our theory. As an outlook
we applied the BCD method to an example not covered by our theory, namely one-step inversion in multi-spectral
X-ray computed tomography.

Future work will be done to  extend  our analysis of the BCD method to more
general forward {\rot operators, in particular non-linear problems including examples like
multi-spectral CT.} This is challenging as
the BCD  is not monotone in the reconstruction error
$\norm{x_k - x^*}$. However, we believe that  the technique introduced in this
paper of finding a  suitable norm where monotonicity  holds can  be extended
to more general situations.

  \section*{Acknowledgments}
The work Markus Haltmeier has been supported by the Austrian Science Fund (FWF),
project P 30747-N32. Simon Rabanser  acknowledges support of   the Austrian Academy of Sciences (\"OAW)
via the DOC Fellowship Programme. The authors thank the anonymous reviewers for
valuable comments that helped to significantly improve the manuscript.


\end{document}